\documentclass[%
a4paper,%
arxiv,%
defaults
]{myclass}

\usepackage[nowarnings]{hopf}

\usepackage{tikz}
\usetikzlibrary{matrix,arrows,shapes}

\newlength{\nodewidth}
\tikzstyle{format} = [draw, thin, fill=black!20, text width=\nodewidth, text centered, minimum height=31pt, outer sep=2pt]
\tikzstyle{justabove} = [above=3pt, text width=\nodewidth, outer sep=2pt]
\tikzstyle{justbelow} = [below=3pt, text width=\nodewidth, outer sep=2pt]
\tikzstyle{justleft}  = [left=3pt,  minimum height=31pt, outer sep=2pt]
\tikzstyle{justright} = [right=3pt, minimum height=31pt, outer sep=2pt]
\tikzstyle{moreabove} = [below=5pt, style=transparent]
\tikzstyle{morebelow} = [above=5pt, style=transparent]
\tikzstyle{moreleft}  = [right=20pt, style=transparent]
\tikzstyle{moreright} = [left=20pt, style=transparent]
\tikzstyle{diaglabel}     = []

\usepackage[%
bb,%
geom,%
]%
{mymacros}

\usepackage{smthcat}

\newtheorem*{remark*}{Remark}

\mybibliographystyle

\mytitle{Comparative Smootheology}

\myauthor{Andrew Stacey}

\myabstract{%
 We compare various different definitions of ``the category of smooth  objects''.
 The definitions compared are due to Chen, Fr\"olicher, Sikorski, Smith, and Souriau.
 The method of comparison is to construct functors between the categories that enable us to see how the categories relate to each other.
 This produces a diagram of categories with the category of Fr\"olicher spaces sitting at its centre.
 Our method of study involves finding a general context into which these categories can be placed.
 This involves considering categories wherein objects are considered in relation to a certain collection of standard test objects.
 This therefore applies beyond the question of categories of smooth spaces.
}

\begin{document}

\mymaketitle

\section{Introduction}
\label{sec:intro}

The purpose of this paper is to compare various definitions of categories of smooth objects.
The definitions compared are due to Chen \cite{kc3}, Fr\"olicher \cite{af}, Sikorski \cite{rs3} (see also Mostow \cite{mm2}), Smith \cite{js2}, and Souriau \cite{js}.
Each has the same underlying principle: we know what it means for a map to be smooth between certain subsets of Euclidean spaces and so in general we declare a function to be smooth if whenever, we examine it using these subsets, it is smooth.
This is a rather vague statement\emhyp{}what do we mean by ``examine''?\emhyp{}and the various definitions can all be seen as different ways of making this precise.

Most of the definitions were introduced because someone wished to extend some aspect of the theory of smooth manifolds to spaces which were not, traditionally, viewed as smooth manifolds.
These definitions, therefore, were motivated by an initial problem and this influenced the choice of category.
A definition was considered suitable if it helped solve the problem.
This has led to a proliferation of possible definitions for ``smooth objects''.
Such comparison as is already in the literature tends to focus on applications, see for example Mostow \cite[\S 4]{mm2}.

With so many different definitions of ``the category of smooth objects'' two questions naturally arise:
\begin{enumerate}
\item Which, if any, truly captures the essence of ``smoothness''?
\item How are the various categories related to each other?
\end{enumerate}
The first is, ultimately, subjective.
The answer depends on what one decides the essence of smoothness to be.
The second can be either subjective or (reasonably) objective.
It is subjective if one considers the question as attempting to order the definitions by one being better than the other.
It is (reasonably) objective if one looks for natural relationships between the categories and considers the nature of those relationships.
The qualifier ``reasonably'' is there because the word ``natural'' in the previous sentence is being used in its traditional English meaning of ``not contrived'' rather than its mathematical meaning.

This paper attempts to answer the second question.
Whilst the author has a definite opinion as to the correct answer to the first (Fr\"olicher spaces), he can also see the value in separating subjective opinion from objective mathematics.

\medskip

To look for relationships between categories means finding functors between them.
Thus our first goal is to find functors between the various categories.
Initially we want to consider uncontrived functors.
As each of the categories is an extension of that of smooth manifolds, we start by looking for functors that preserve this subcategory within each.

Once we have found such functors, the next step is to classify these functors (as before, I am using the word ``classify'' with its traditional English meaning).
The ideal situation here is to be able to say that one category is a reflective or co\hyp{}reflective subcategory of another (using the given functors as the inclusion and (co\hyp{})reflector).
We therefore look for adjunctions between the functors and for embeddings.

Our answer at this stage can be summarised in figure~\ref{fig:adjoints}.

\begin{figure}
\label{fig:adjoints}
\begin{centre}
\begin{tikzpicture}
% Directions: x = forcing condition (and direction of maps)
%             y = underlying category
%             z = test category
    \node[format]    (c73)  at (-6,4) {Chen 1973};
    \node[format]    (c75a) at (-6,2) {Chen 1975a};
    \node[format]    (c75b) at (-3,2) {Chen 1975b};
    \node[format]    (c77)  at (-6,0) {Chen 1977};
    \node[format]    (so)   at (-3,0) {Souriau};
    \node[format]    (fr)   at (0,0)  {Fr\"olicher};
    \node[format]    (sm)   at (0,2)  {Smith};
    \node[format]    (si)   at (2.85,2)  {Sikorski};
    \node[justabove] (fru)  at (fr)   {};
    \node[justabove] (sou)  at (so)   {};
    \node[justabove] (smu)  at (sm)   {};
    \node[justabove] (siu)  at (si)   {};
    \node[justabove] (c75bu) at (c75b) {};
    \node[justabove] (c75au) at (c75a) {};
    \node[justabove] (c77u) at (c77)  {};
    \node[justbelow] (sod)  at (so)   {};
    \node[justbelow] (c77d) at (c77)  {};
    \node[justright] (frr)  at (fr)   {};
    \node[justright] (smr)  at (sm)   {};
    \node[justleft]  (frl)  at (fr)   {};
    \node[justleft]  (sml)  at (sm)   {};
    \node[justright] (c73r) at (c73)  {};
    \node[justright] (c75ar) at (c75a) {};
\path[->>]
          (so)    edge (fr)
          (sm)    edge (fr)
          (c77)   edge (so)
          (c75a)  edge (c75b)
          (c75b)  edge (sm)
          (si)    edge (sm)
          (c75ar)  edge (c73r);
\path[right hook->]
          (c73)   edge (c75a)
          (smu)   edge (siu)
          (frl)   edge (sml)
          (frr)   edge (smr);
\path[left hook->]
          (smu)   edge (c75bu)
          (c75bu) edge (c75au)
          (fru)   edge (sou)
          (sou)   edge (c77u)
          (sod)   edge (c77d);

    \node at (intersection cs: first line={(fru)--(so)}, second line={(sou)--(fr)}) {\(\top\)};
    \node at (intersection cs: first line={(sou)--(c77)}, second line={(c77u)--(so)}) {\(\top\)};
    \node at (intersection cs: first line={(c77d)--(so)}, second line={(sod)--(c77)}) {\(\top\)};
    \node at (intersection cs: first line={(c75bu)--(c75a)}, second line={(c75b)--(c75au)}) {\(\top\)};
    \node at (intersection cs: first line={(c75bu)--(sm)}, second line={(c75b)--(smu)}) {\(\top\)};
    \node at (intersection cs: first line={(smu)--(si)}, second line={(sm)--(siu)}) {\(\bot\)};
    \node at (intersection cs: first line={(c73r)--(c75a)}, second line={(c73)--(c75ar)}) {\(\vdash\)};
\end{tikzpicture}
\caption{The relationships between the categories}
\end{centre}
\end{figure}
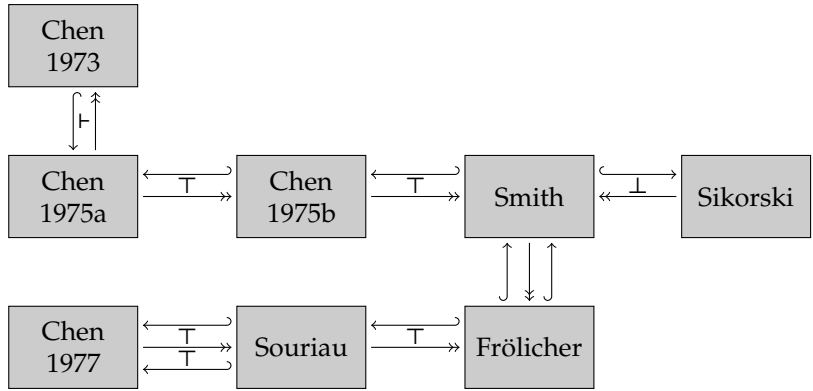

Once we have examined all the uncontrived functors, we consider the question of whether or not there are any contrived functors.
To limit this question slightly, we consider the question as to whether or not any of the various categories of smooth spaces are equivalent, where we allow contrived functors; that is to say, we allow any functors irrespective of how they behave on manifolds.

Our answer to this is satisfying: there are none.

\medskip

Let us now give a tour of this study, ensuring that we point out the main attractions.

In Section~\ref{sec:smoothcats} we recall the definitions of five of the various categories of smooth spaces that have appeared in the literature.
All of these proposals have the same basic shape and in Section~\ref{sec:recipe} we extract that shape and put it into a general setting.

To explain this basic shape we need to consider how the various categories of smooth spaces might have been devised.
Let us start with the definition of a smooth manifold.

\begin{defn}
A \emph{smooth manifold} consists of a topological manifold together with a smooth structure.
A \emph{smooth structure} consists of a maximal smooth atlas.
A \emph{smooth atlas} is a family of charts, which are continuous maps to or from open subsets of Euclidean spaces (satisfying certain other conditions).
\end{defn}

We say ``to or from'' because charts, being homeomorphisms, are invertible and it is a matter of taste as to whether the term ``chart'' means the map to the manifold or off it.

Let us compare this with one of the definitions for a category of smooth objects.

\begin{defn}[{Souriau~\cite{js}}]
 A \emph{Souriau space} or \emph{diffeological space} is a pair \((X, \m{D})\) where \(X\) is a set and \(\m{D}\) is a family of maps (also called \emph{plots}) into \(X\) with domains open subsets of Euclidean spaces.
 These have to satisfy the following conditions.
 \begin{enumerate}
 \item
  Every constant map is a plot.
 \item
  If \(\phi \colon U \to X\) is a plot and
   \(\theta \colon U' \to U\)
  is a \cimap between open subsets of Euclidean spaces then \(\phi\theta\) is a plot.
 \item
  If \(\phi \colon U \to X\) is a map which is everywhere locally a plot then it is a plot.
 \end{enumerate}
\end{defn}

By comparing these two definitions, we can see certain common themes.

\begin{enumerate}
\item In both there is an underlying category of objects to which one might wish to give a smooth structure.

\item In both there is a category of test spaces and a smooth structure consists of a family of morphisms to or from these test spaces and the object in question.

\item In both these families are not completely arbitrary.
There is a forcing condition which must be met.
\end{enumerate}

In Section~\ref{sec:recipe} we generalise this.
Starting with an underlying category, \(\ucat\), a test category, \(\tcat\), and a forcing condition we define \uFVtcat, \(\uFVtcat\).
A \uFVtobj consists of a \uobj together with families of \umors to and from the \tobjs with the property that the appropriate forcing condition is satisfied.

At this point it is worth making a comment about the direction of the test morphisms.
In the proposed categories of smooth spaces there is almost always a preferred direction, either to the smooth space or from it.
Only one, that of Fr\"olicher spaces, directly uses both.
However, it is possible to recast all of the definitions to use both without changing the actual category.
This makes the comparision simpler.

\medskip

Having set up out general recipe, we proceed to the search for functors in Section~\ref{sec:functor}.
As we are initially looking for uncontrived functors, we begin by looking at functors that are induced by one of three operations: changing the forcing condition, changing the category of test spaces, and changing the underlying condition.
As we have a specific situation to which we wish to apply this theory, we do not work in full generality but rather seek out that midpoint where there is sufficient generality to make matters clear but not so much that we lose sight of our goal.
In particular, when looking at what happens when we change the underlying category we consider only certain simple changes.

Having set up this general theory, we apply it in Section~\ref{sec:funwild} to the cases at hand.
This is where we produce all the functors used in Figure~\ref{fig:adjoints}.
We also give simple descriptions of the functors so that someone who is not interested in the detailed construction can still work out what the functors actually are.

The various properties of these functors that we are interested in all follow from the general work in Section~\ref{sec:functor}.
However, whilst we can use this general theory to say that a particular functor has a certain property, it cannot be used to say that it doesn't have another property.
Therefore in Sections~\ref{sec:diff} and~\ref{sec:adjoint} we show that all the properties that didn't follow from the work of Section~\ref{sec:functor} do not hold.
Thus the adjunctions indicated in Figure~\ref{fig:adjoints} are all the adjoints that exist that involve any of the functors in that diagram.

\medskip

Section~\ref{sec:equiv} concludes the main part of this paper by considering arbitrary relationships between the categories.
To make this a more specific question, we look for equivalences of categories.
Our strategy here is to look for \emph{invariants} of the categories to enable us to say that there are no such equivalences.

The required invariants are extremely simple.
We look at terminal objects, the two element set, and the real line.
The first two essentially say that any equivalence of these categories must be equivalent to one that preserves the underlying category.
The third, the real line, is the one that characterises the smooth structures.
One highlight of this section is the result that in the category of Fr\"olicher spaces it is possible to categorically find the real line.
That is, purely using categorical tools one can say ``This is \R.''.
I do not know whether or not this is true for the other categories.
Nonetheless, it implies that there are no interesting automorphisms of any of these categories.

\medskip

The final two sections deal with side issues to the main purpose of this paper.
Section~\ref{sec:topology} considers the r\^ole of topology in defining a smooth space.
Several of the definitions use a topological category as the underlying category but others just use \(\xcat\).
Section~\ref{sec:topology} considers how to remove the topology from those that start with it.

Section~\ref{sec:nonset} is concerned with non\hyp{}set\enhyp{}based theories.
In all of the categories considered in this paper the objects have underlying sets.
It is an interesting question as to how to remove this requirement, but the answer is not obvious.
Thus in Section~\ref{sec:nonset} we do no more than raise this question.

\medskip

Because the purpose of this paper is to examine the word ``smooth'', I shall use the term ``\cimap'' when referring to a map between locally convex subsets of Euclidean spaces which is smooth in the standard sense; that is, all definable directional derivatives exist and are continuous on their domain of definition.

As many people defining an extension of smooth manifolds have used the term ``differentiable space'', we adopt the convention that we shall refer to each type of structure by the name of its original author.
We even do this for diffeological spaces although that name is unambiguous and has a nice ring to it.

Finally, I am grateful to the various people a the \(n\)\enhyp{}Category Caf\'e who commented on the preliminary version of this article.
In particular I wish to acknowledge the helpful comments of Bruce Barlett and Urs Schreiber, the latter being the one who started the discussion.
I am especially grateful to Urs Schreiber for suggesting the title.

\section{Categories of Smooth Spaces}
\label{sec:smoothcats}

In this section we shall describe five of the various categories of smooth spaces that have appeared in the literature.

\begin{defn}[{Chen~\cite{kc3}}]
\label{def:chen}
 A \emph{Chen space} is a pair \((X, \m{P})\) where \(X\) is a set and \(\m{P}\) is a family of maps (called \emph{plots}) into \(X\) with domains convex subsets of Euclidean spaces.
 These have to satisfy the following conditions.
 \begin{enumerate}
 \item
  Every constant map is a plot.
 \item
  If \(\phi \colon C \to X\) is a plot and
   \(\theta \colon C' \to C\)
  is a \cimap between convex regions then \(\phi\theta\) is a plot.
 \item
  If \(\phi \colon C \to X\) is a map which is everywhere locally a plot then it is a plot.
 \end{enumerate}

 A morphism of Chen spaces is a map on the underlying sets which takes plots to plots.
\end{defn}

\begin{defn}[{Souriau~\cite{js}}]
 A \emph{Souriau space} is a pair \((X, \m{D})\) where \(X\) is a set and \(\m{D}\) is a family of maps (also called \emph{plots}) into \(X\) with domains open subsets of Euclidean spaces.
 These have to satisfy the following conditions.
 \begin{enumerate}
 \item
  Every constant map is a plot.
 \item
  If \(\phi \colon U \to X\) is a plot and
   \(\theta \colon U' \to U\)
  is a \cimap between open subsets of Euclidean spaces then \(\phi\theta\) is a plot.
 \item
  If \(\phi \colon U \to X\) is a map which is everywhere locally a plot then it is a plot.
 \end{enumerate}

 A morphism of Souriau spaces is a map on the underlying sets which takes plots to plots.
 The family \(\m{D}\) is sometimes called the \emph{diffeology} of the pair \((X, \m{D})\).
\end{defn}

\begin{defn}[{Sikorski~\cite{rs3}}]
 A \emph{Sikorski space} is a triple \((X, \m{T}, \m{F})\) where \(X\) is a set, \(\m{T}\) a topology on \(X\), and \(\m{F}\) is a subalgebra of the algebra of continuous real\hyp{}valued functions on \(X\) satisfying the following conditions.
 \begin{enumerate}
 \item
  Functionals in \(\m{F}\) are \emph{locally detectable}, in that \(f \colon X \to \R\) is in \(\m{F}\) if each point \(x \in X\) has a neighbourhood, say \(V\), for which there is a function \(g \in \m{F}\) with
   \(f \restrict_V = g \restrict_V\).

 \item
  If
   \(f_1, \dotsc, f_k \in \m{F}\)
  and \(g \in \Ci(\R^k, \R)\) then
   \(g(f_1, \dotsc, g_k) \in \m{F}\).
 \end{enumerate}

 A morphism of Sikorski spaces is a map on the underlying sets, \(g \colon X \to Y\), such that \(f g \in \m{F}_X\) for all \(f \in \m{F}_Y\).
\end{defn}

\begin{defn}[{Smith~\cite{js2}}]
A \emph{Smith space} is a triple \((X, \m{T}, \m{F})\) where \(X\) is a set, \(\m{T}\) a topology on \(X\), and \(\m{F}\) a set of continuous real\hyp{}valued functions on \(X\).
The set \(\m{F}\) has to satisfy a certain closure condition.
For an open set \(U \subseteq \R^n\), let \(\m{F}(U)\) denote the set of continuous maps \(\phi \colon U \to X\) with the property that \(f \phi \in \Ci(U, \R)\) for all \(f \in \m{F}\).
The closure condition is that \(\m{F}\) contains all continuous functions \(g \colon X \to \R\) with the property that for all open sets \(U \subseteq \R^n\) (\(n\) arbitrary) and \(\phi \in \m{F}(U)\), \(g \phi \in \Ci(U, \R)\).
\end{defn}

\begin{defn}[{Fr\"olicher~\cite{af}}]
 A \emph{Fr\"olicher space} is a triple \((X, \m{C}, \m{F})\) where \(X\) is a set, \(\m{C}\) is a family of curves in \(X\), i.e.~a subset of \(\map(\R, X)\), and \(\m{F}\) is a family of functionals on \(X\), i.e.~a subset of \(\map(X, \R)\).
 The sets \(\m{C}\) and \(\m{F}\) have to satisfy the following compatibility condition: a curve \(c \colon \R \to X\) is in \(\m{C}\) if and only if \(f c \in \Ci(\R,\R)\) for all functionals \(f \in \m{F}\), and similarly a functional \(f \colon X \to \R\) is in \(\m{F}\) if and only if \(f c \in \Ci(\R, \R)\) for all curves \(c \colon \R \to X\).

 A map of Fr\"olicher spaces is a map \(\phi \colon X \to Y\) on the underlying sets satisfying the following (equivalent) conditions.
 \begin{enumerate}
 \item
   \(\phi c \in \m{C}_Y\) for all \(c \in \m{C}_X\),
 \item
   \(f \phi \in \m{F}_X\) for all \(f \in \m{F}_Y\),
 \item
   \(f \phi c \in \Ci(\R, \R)\)
  for all \(c \in \m{C}_X\) and \(f \in \m{F}_Y\).
 \end{enumerate}
\end{defn}

\begin{remark}
\begin{enumerate}
\item
Chen modified his definition considerably as he worked with it.
The earliest definition seems to be from \cite{kc} and the latest from \cite{kc3}.
In between these two lies \cite{kc5} which includes two further definitions (one, it should be said, is an incorrect recollection of the definition from \cite{kc}).
Although all his definitions are based on the same theme, there is considerable variation from the first to the last.
We shall comment a little on this in the next section.

\item 
A Souriau space is very similar to a Chen space except that the domains of the test functions are different.
This is related to a very interesting fact.
All of the above definitions consist of a set together with certain ``test functions'' between the set and certain ``test spaces'' with the functions either to the set or out of it.
What is worth noting is that for models consisting of maps to the set, several choices of test spaces have been proposed.
However, for models consisting of maps out of the set, all have used only \R for the test space.
One could speculate that the reason for this is that it is well\hyp{}known that a map into a subset of a Euclidean space is smooth if and only if all the coordinate projections are smooth; the corresponding result for maps out of a (suitable) subset of a Euclidean space\emhyp{}namely, Boman's theorem and Kriegl and Michor's extension\emhyp{}is much less well\hyp{}known.
\end{enumerate}
\end{remark}

\section{A General Recipe}
\label{sec:recipe}

All of the heretofore proposed categories of ``smooth objects'' can be put into a \emph{standard form}.
This standard form is determined by certain choices.
In short, these choices are of an underlying category, of test spaces, and of forcing conditions.
The underlying category should be thought of as ``those objects to which one might wish to give a smooth structure''.
The test spaces should be thought of as ``those objects for which there is an indisputable smooth structure''.
The forcing conditions should be thought of as ensuring that ``morphisms which ought to be smooth actually are smooth''.

A close relative of this structure is of sheaves on a site.
The site is the category of test spaces and the sheaf condition is the forcing condition.
To see where the underlying category fits in, one should consider so\hyp{}called \emph{quasi\hyp{}representable} sheaves on a site, or \emph{concrete} sheaves on a \emph{concrete} site, see \cite{0807.1704} and the references therein.

Another close relative of this is the \emph{Isbell envelope} of an essentially small category.
The Isbell envelope of \(\acat\), written \(\Eacat\), is a category whose objects consist of a contravariant functor \(\itest \colon \Oacat \to \xcat\), a covariant functor \(\otest \colon \acat \to \xcat\), and a natural transformation \(\itest \times \otest \to \Hom{\acat}{-}{-}\).
Morphisms in this category are pairs of natural transformations between the functors.
One can think of an \Eaobj as a \emph{virtual} object of \(\acat\) in that it can be experimented on using \aobjs.
This notion can also be encoded using \emph{profunctors}.
We shall comment on this relationship later.

\subsection{Virtual Objects}

Let us now describe our structure precisely.
We are not aiming for the most general approach here; rather we wish to find a setting that helps with our study of the actual examples already posited.
Therefore, we wish to keep our recipe as close as possible to the definitions in section~\ref{sec:smoothcats}.

\medskip

The easy part is the two categories.
We fix these: \ucat, \(\ucat\), and \tcat, \(\tcat\).
We also choose a faithful functor \(\ufunc \colon \tcat \to \ucat\) from \tcat to \ucat (although not necessary, it will usually be the case that \tcat will be a subcategory of \ucat).
This allows us to define the first stage.

\begin{defn}
Let \(\ucat\) and \(\tcat\) be categories\emhyp{}\ucat and \tcat respectively.
Let \(\ufunc \colon \tcat \to \ucat\) be a faithful functor.

\uVtcatu, \(\uVtcat\), is the following category.
A \emph{\uVtobj{}} consists of a triple \((X, \itest, \otest)\) where
\begin{itemize}
\item \(X\) is an \uobj,
\item \(\itest \colon \tcat \to \xcat\) is a subfunctor of the functor \(\tobj \mapsto \Hom{\ucat}{\ufunc(\tobj)}{\uobj}\),
\item \(\otest \colon \tcat \to \xcat\) is a subfunctor of the functor \(\tobj \mapsto \Hom{\ucat}{\uobj}{\ufunc(\tobj)}\).
\end{itemize}
These have to satisfy the following compatibility condition.
Consider the functors \(\tcat \times \Otcat \to \xcat\),
\[
(\tobj, \tobj') \to \Hom{\ucat}{\ufunc(\tobj)}{\uobj} \times \Hom{\ucat}{\uobj}{\ufunc(\tobj')},  \qquad (\tobj, \tobj') \to \Hom{\ucat}{\ufunc(\tobj)}{\ufunc(\tobj')}.
\]
Composition defines a natural transformation from the first to the second.
This natural transformation must induce a natural transformation from \(\itest \times \otest\) to \(\Hom{\tcat}{-}{-}\) (with the latter viewed as a subfunctor of \(\Hom{\ucat}{\ufunc(-)}{\ufunc(-)}\).

A \uVtmoralt, say \((\uobj_1, \itest_1, \otest_1)\) to  \((\uobj_2, \itest_2, \otest_2)\), is a \umor \(\umor \colon \uobj_1 \to \uobj_2\) with the properties that \(\umor \ifn \in \itest_2\) for all \(\ifn \in \itest_1\) and \(\ofn \umor \in \otest_1\) for all \(\ofn \in \otest_2\).

We shall denote the obvious forgetful functor \(\uVtcat \to \ucat\) by \(\uVtobj \mapsto \forfunc[\uVtobj]\).
\end{defn}

\begin{remark}
\label{rk:presmth}
\begin{enumerate}
\item When we say ``for all \(\ifn \in \itest\)'' we mean ``for all \tobjs, \(\tobj\), and all \(\ifn \in \itest(\tobj)\)''.
This is a shorthand that we shall frequently use to avoid having to introduce unnecessary dummies.
For example, given a \umor, \(\umor\), we write ``\(\umor \in \im \ufunc\)'' to mean that there is a \tmor, \(\tmor\), such that \(\umor = \ufunc(\tmor)\).
Implicit in this is that the domain and codomain of \(\umor\) come from \tobjs via \(\ufunc\).

\item In our definition we have allowed for test maps both into and out of our test spaces.
This seems a little at variance with the definitions that we are generalising.
We shall see that, in the presence of certain forcing conditions, one of these families can be effectively removed.
Putting both in at the start allows us to consider both situations at the same time.

\item
\label{it:sub}
When we say ``subfunctor'' we mean this in the strictest sense: that \(\itest(\tobj)\) is a subset of \(\Hom{\ucat}{\ufunc(\tobj)}{\uobj}\).
This is not technically necessary but makes the notation and exposition considerably simpler.
This means that \(\uVtcat\) is an amnestic construct over \(\ucat\).
It is obvious that it is (uniquely) transportable as well.
Moreover, if \(\tcat\) has a small skeleton then the fibres of the forgetful functor \(\forfunc \colon \uVtcat \to \ucat\) are small.
\end{enumerate}
\end{remark}

Now let us consider the forcing conditions.
As we have said before, we are more interested in finding a context into which we can place all known examples than in finding the most general setting.
A forcing condition will consist of two parts: an input forcing condition and an output forcing condition.
The two are formally similar, related by an obvious ``flip'', so to ease the exposition we shall focus on one type.
We choose, for no good reason, the input forcing condition.

To define an input forcing condition, we proceed via several steps.
The first is to record an obvious lemma that says that \tcat naturally embeds in \uVtcat.

\begin{lemma}
\label{lem:tembed}
There is a functor \(\sfunc \colon \tcat \to \uVtcat\) which embeds \(\tcat\) as a full subcategory of \(\uVtcat\).
The \uVtobj[\sfunc(\tobj)] has underlying \uobj[\ufunc(\tobj)],  \itest \(\itest(\tobj') = \Hom{\tcat}{\tobj'}{\tobj}\), and \otest \(\otest(\tobj') = \Hom{\tcat}{\tobj}{\tobj'}\).
On morphisms, \(\sfunc(\tmor)\) is determined by the requirement that \(\forfunc[\sfunc(\tmor)] = \ufunc(\tmor)\).

This functor has several important properties.
As functors \(\tcat \to \ucat\), we have \(\forfunc[\sfunc] = \ufunc\).
For a \uVtobj[\uVtobj] and \tobj[\tobj] the subsets \(\itest(\tobj)\) and \(\Hom{\uVtcat}{\sfunc(\tobj)}{\uVtobj}\) of \(\Hom{\ucat}{\ufunc(\tobj)}{\abs{\uVtobj}}\) are the same; similarly for the output test functions.
\noproof
\end{lemma}

The next stage is to define the notion of a \emph{trial} for a pair \((\tobj,\uVtobj)\) where \(\tobj\) is a \tobj and \(\uVtobj\) a \uVtobj.
This provides a way to test whether a \umor  \(\umor \colon \ufunc(\tobj) \to \abs{\uVtobj}\) ``ought'' to be in \(\itest_{\uVtobj}(\tobj)\).
The idea being that if such a morphism succeeds at sufficiently many trials, it is ``forced''.

\begin{defn}
Let \(\uVtobj\) be a \uVtobj.
Let \(\tobj\) be a \tobj.
A \emph{trial} from \(\tobj\) to \(\uVtobj\) is defined to be a pair \(\trial{}\) where \(\tmor\) is a \tmor with target \(\tobj\) and \(\uVtmor\) is a \uVtmor with source \(\uVtobj\).

A \umor \(\umor \colon \ufunc(\tobj) \to \abs{\uVtobj}\) \emph{succeeds at the trial} if the \umor \(\trial{\umor}\) underlies a \uVtmor.
For a fixed \umor, \(\umor \colon \ufunc(\tobj) \to \abs{\uVtobj}\), we define \(\trials{\umor}\) to be the class of trials from \(\tobj\) to \(\uVtobj\) at which \(\umor\) succeeds.

Let \(\tobj_1, \tobj_2\) be \tobjs and \(\uVtobj_1, \uVtobj_2\) be \uVtobjs.
For \(i = 1,2\), let \(\trial[_i]{}\) be a trial from \(\tobj_i\) to \(\uVtobj_i\).
Let \(\tmor \colon \tobj_2 \to \tobj_1\) be a \tmor and \(\uVtmor \colon \uVtobj_1 \to \uVtobj_2\) be a \uVtmor.
We say that \emph{\(\trial[_2]{}\) is compatible with \(\trial[_1]{}\) along \(\tmor\) and \(\uVtmor\)} if there exist a \tmor \(\tmor'\) and \uVtmor \(\uVtmor'\) such that \(\tmor \tmor_2 = \tmor_1 \tmor'\) and \(\uVtmor_2 \uVtmor = \uVtmor' \uVtmor_1\).
For a family of trials, \(\m{F}\), from \(\tobj_1\) to \(\uVtobj_1\) we define \(\uVtmor \m{F} \tmor\) to be the family of trials from \(\tobj_2\) to \(\uVtobj_2\) with the property that each trial in \(\uVtmor \m{F} \tmor\) is compatible with a trial in \(\m{F}\).

We define \rcat, \(\rcat\), to be the category whose objects are triples \((\tobj, \uVtobj, \m{F})\) with \(\tobj\) a \tobj, \(\uVtobj\) a \uVtobj, and \(\m{F}\) a family of trials from \(\tobj\) to \(\uVtobj\).
The morphisms from \((\tobj_1, \uVtobj_2, \m{F}_1)\) to \((\tobj_2, \uVtobj_2, \m{F}_2)\) are pairs \((\tmor, \uVtmor)\) where \(\tmor \colon \tobj_2 \to \tobj_1\) is a \tmor and \(\uVtmor \colon \uVtobj_1 \to \uVtobj_2\) is a \uVtmor such that \(\uVtmor \m{F}_1 \tmor \subseteq \m{F}_2\).
\end{defn}

\begin{remark}
\begin{enumerate}
\item We can illustrate a trial, \(\trial{}\), diagramatically as follows.

\begin{centre}
\begin{tikzpicture}[node distance=1.5cm, auto,>=latex', thick]
    \path node (lblank) {};
    \path[->] node[right of=lblank] (tobj) {\(\tobj\)}
                  (lblank) edge node[diaglabel] {\(\tmor\)} (tobj);
    \path[->, dotted] node[right of=tobj] (uVtobj) {\(\uVtobj\)}
                  (tobj) edge (uVtobj);
    \path[->] node[right of=uVtobj] (rblank) {}
                  (uVtobj) edge node[diaglabel] {\(\uVtmor\)} (rblank);
\end{tikzpicture}
\end{centre}

The idea is to try to fill in the dotted arrow, though of course this does not make sense as the categories on the left and right are not the same.
More precisely, let us write \(\tobj'\) for the source of \(\tmor\) and \(\uVtobj'\) for the target of \(\uVtmor\).
Then a \umor, \(\umor \colon \ufunc(\tobj) \to \abs{\uVtobj}\), succeeds at this trial if the \umor

\begin{centre}
\begin{tikzpicture}[node distance=2cm, auto,>=latex', thick]
    \path[->] node (lblank) {\(\ufunc(\tobj')\)};
    \path[->] node[right of=lblank] (tobj) {\(\ufunc(\tobj)\)}
                  (lblank) edge node[diaglabel] {\(\ufunc(\tmor)\)} (tobj);
    \path[->] node[right of=tobj] (uVtobj) {\(\abs{\uVtobj}\)}
                  (tobj) edge node[diaglabel] {\(\umor\)} (uVtobj);
    \path[->] node[right of=uVtobj] (rblank) {\(\abs{\uVtobj'}\)}
                  (uVtobj) edge node[diaglabel] {\(\abs{\uVtmor}\)} (rblank);
\end{tikzpicture}
\end{centre}

lifts to a \uVtmor \(\sfunc(\tobj') \to \uVtobj'\).
Equivalently, if the above \umor is in \(\itest_{\uVtobj'}(\tobj')\).

\item
The compatibility condition is illustrated by the following diagram.

\begin{centre}
\begin{tikzpicture}[node distance=1.5cm, auto, >=latex', thick]
\path node (ulblank) {};
\path node[right of=ulblank] (tobj1) {\(\tobj_1\)};
\path node[right of=tobj1] (uVtobj1) {\(\uVtobj_1\)};
\path node[right of=uVtobj1] (urblank) {};
\path node[below of=ulblank] (dlblank) {};
\path node[right of=dlblank] (tobj2) {\(\tobj_2\)};
\path node[right of=tobj2] (uVtobj2) {\(\uVtobj_2\)};
\path node[right of=uVtobj2] (drblank) {};
\path[->] (ulblank) edge node {\(\tmor_1\)} (tobj1)
(tobj1) edge[dotted] (uVtobj1)
(uVtobj1) edge node[diaglabel] {\(\uVtmor_1\)} (urblank)
(dlblank) edge node[diaglabel] {\(\tmor_2\)} (tobj2)
(tobj2) edge[dotted] (uVtobj2)
(uVtobj2) edge node[diaglabel] {\(\uVtmor_2\)} (drblank)
(tobj2) edge node[diaglabel] {\(\tmor\)} (tobj1)
(uVtobj1) edge node[diaglabel] {\(\uVtmor\)} (uVtobj2)
(dlblank) edge node[diaglabel] {\(\tmor'\)} (ulblank)
(urblank) edge node[diaglabel] {\(\uVtmor'\)} (drblank);
\end{tikzpicture}
\end{centre}

The point being that if a \umor, \(\umor \colon \ufunc(\tobj_1) \to \abs{\uVtobj_1}\), succeeds at the first trial then \(\abs{\uVtmor} \umor \ufunc(\tmor)\) will succeed at the second.

\item \rcatu has an obvious functor to \(\Otcat \times \uVtcat\).
The fibre of this functor at \((\tobj, \uVtobj)\) is the partially ordered class of subclasses of trials from \(\tobj\) to \(\uVtobj\).
Observe that every triple \((\tobj, \uVtobj, \umor)\) with \(\umor\) a \umor from \(\ufunc(\tobj)\) to \(\abs{\uVtobj}\) defines an \robj, \((\tobj, \uVtobj, \trials{\umor})\).
Given a \tmor, \(\tmor \colon \tobj' \to \tobj\), and a \uVtmor, \(\uVtmor \colon \uVtobj \to \uVtobj'\), we have that \(\uVtmor \trials{\umor} \tmor \subseteq \trials{\abs{\uVtmor} \umor \ufunc(\tmor)}\) whence there is a \rmor from \((\tobj, \uVtobj, \trials{\umor})\) to \((\tobj', \uVtobj', \trials{\abs{\uVtmor} \umor \ufunc(\tmor)})\).
\end{enumerate}
\end{remark}

\begin{defn}
\label{def:forcing}
An \emph{input forcing condition} for \uVtobjs is a functor \(\iforce \colon \rcat \to \{0 \to 1\}\) with the property that for \tobjs, \(\tobj_1\) and \(\tobj_2\), \(\iforce(\tobj_1, \sfunc(\tobj_2), \m{F}) = 0\) if \((1,1) \notin \m{F}\).

For a \tobj, \(\tobj\), and a \uVtobj, \(\uVtobj\), we shall say that a family of trials, \(\m{F}\), from \(\tobj\) to \(\uVtobj\) is \emph{sufficient} if \(\iforce(\tobj, \uVtobj, \m{F}) = 1\).

For a \tobj, \(\tobj\), and a \uVtobj, \(\uVtobj\), we shall say that a \umor \(\umor \colon \ufunc(\tobj) \to \abs{\uVtobj}\) is \emph{forced} if \(\iforce(\tobj, \uVtobj, \trials{\umor}) = 1\).
\end{defn}

\begin{remark}
\begin{enumerate}
\item There is an obvious generalisation to output forcing conditions by ``flipping'' all the arrows.
We shall write \(\oforce\) for the corresponding functor.

\item From examining the definition of morphisms in \rcat we see that if \(\umor \colon \ufunc(\tobj) \to \abs{\uVtobj}\) is forced and \(\tmor \colon \tobj' \to \tobj\), \(\uVtmor \colon \uVtobj \to \uVtobj'\) are suitable morphisms, then \(\abs{\uVtmor} \umor \ufunc(\tmor)\) is also forced.

\item Similarly, if \(\m{F}\) is a sufficient family of trials from \(\tobj\) to \(\uVtobj\) and \(\m{F} \subseteq \m{F}'\) then \(\m{F}'\) is also sufficient.

\item The one restraint on the functor translates into the statement: ``Nothing obviously non\hyp{}smooth should ever be forced to be smooth''.
\end{enumerate}
\end{remark}

\begin{defn}
A \emph{forcing condition} is a choice of input forcing condition, \(\iforce\), and output forcing condition, \(\oforce\).

A \uVtobj, \(\uVtobj\), \emph{satisfies} the forcing condition \((\iforce, \oforce)\) if, whenever \(\tobj\) is a \tobj and \(\umor \colon \ufunc(\tobj) \to \abs{\uVtobj}\) is forced then \(\umor \in \itest_{\uVtobj}(\tobj)\) and, similarly, whenever \(\umor \colon \abs{\uVtobj} \to \ufunc(\tobj)\) is forced then \(\umor \in \otest_{\uVtobj}(\tobj)\).

Given a forcing condition, we write \(\uFVtcat\) for the full subcategory of \(\uVtcat\) consisting of \uVtobjs satisfying this forcing condition.
\end{defn}

Since forcing conditions take values in \(\{0 \to 1\}\), they form a partially ordered set and can thus be combined using logical connectors.

\subsection{Smooth Objects in the Wild}
\label{sec:smthwild}

Let us now specify to the matter in hand.
The above describes an extremely general set\hyp{}up, far broader than we shall need.
The first step to reducing to our examples is to find a minimal setting containing all of them.
For this, we limit our choices for test category, underlying category, and forcing condition.

As the forcing conditions are the most unfamilar part, we start by introducing a list of examples.
This list contains all that are needed to specify the given categories of smooth spaces.
We shall then translate the various categories of smooth spaces into this language.
We conclude with some remarks as to the characteristics of the test categories and underlying categories.

The way that we define, say, an input forcing condition is to list certain significant families of trials from a generic \tobj to a generic \uVtobj.
As the input forcing condition is a functor to \(\{0 \to 1\}\), this will force many other families to also be significant\emhyp{}any family that is the target of a morphism from one of the specified ones.
All the others are insignificant.

To check that these are well\hyp{}defined, the only thing to check is the ``non\hyp{}stupid'' condition, namely that nothing is forced that really shouldn't be forced.
We shall not do this here.

\begin{defn}
In the following, \(\uVtobj\) will be a \uVtobj and \(\tobj\) a \tobj.
We list the determining families of trials.

\begin{description}
\item[The input saturation condition]
\[
  \{(1_{\tobj}, \ofn) : \ofn \in \otest_{\uVtobj}\}
\]
Here, we use the fact that \(\otest_{\uVtobj}(\tobj') = \Hom{\uVtcat}{\uVtobj}{\sfunc(\tobj')}\) to regard \(\ofn\) as a \uVtmor.

\item[The output saturation condition]
\[
  \{(\ifn, 1_{\tobj}) : \ifn \in \itest_{\uVtobj}\}
\]

\item[The input determined condition]
\[
  \{(\tmor_\lambda, 1_{\uVtobj}) : \lambda \in \Lambda\}
\]
where \(\{\tmor_\lambda \colon \tobj_\lambda \to \tobj : \lambda \in \Lambda\}\) is a family of \tmors with the property that a \umor \(\umor \colon \ufunc(\tobj) \to \ufunc(\tobj')\) is in the image of \(\ufunc\) if \(\umor \ufunc(\tmor_\lambda)\) is in the image of \(\ufunc\) for all \(\lambda\).

\item[The output determined condition]
\[
  \{(1_{\uVtobj}, \tmor_\lambda) : \lambda \in \Lambda\}
\]
where \(\{\tmor_\lambda \colon \tobj \to \tobj_\lambda : \lambda \in \Lambda\}\) is a family of \tmors with the property that a \umor \(\umor \colon \ufunc(\tobj') \to \ufunc(\tobj)\) is in the image of \(\ufunc\) if \( \ufunc(\tmor_\lambda) \umor\) is in the image of \(\ufunc\) for all \(\lambda\).

\item[The input specifically\enhyp{}determined condition]
This is the same as the input determined condition except that the families \(\{\tmor_\lambda\}\) must be drawn from a pre\hyp{}specified list.

\item[The output specifically\enhyp{}determined condition]
This is to the output determined condition as the input specifically\enhyp{}determined condition is to the input determined condition.

\item[The input sheaf condition]
To define this, we need to assume that the test category is a site.
\[
  \{(\tmor_\lambda, 1_{\uVtobj}) : \lambda \in \Lambda\}
\]
where \(\{\tmor_\lambda \colon \tobj_\lambda \to \tobj : \lambda \in \Lambda\}\) is a covering of \(\tobj\).

\item[The ouptut sheaf condition]
To define this, we require \ucat to be of a topological flavour.
Also given a subset \(\uobj \subseteq \forfunc[\uVtobj]\) we define the \emph{subspace} \Vtalgobj structure on \(\uobj\) to be given by
\begin{align*}
\itest(\tobj') &= \{\ifn \colon \ufunc(\tobj') \to \uobj : \forfunc[\iota_\lambda]\ifn \in \itest_{\uVtobj}(\tobj')\}, \\
\otest(\tobj') &= \{\ofn\forfunc[\iota_\lambda] \colon \uobj \to \ufunc(\tobj') : \ofn \in \otest_{\uVtobj}(\tobj')\}.
\end{align*}
With this we take families of trials of the form
\[
  \{(\iota_\lambda, 1_{\tobj}) : \lambda \in \Lambda\}
\]
where \(\iota_\lambda \colon \uVtobj_\lambda \to \uVtobj\) is a family of \uVtmors such that \(\forfunc[\iota_\lambda] \colon \forfunc[\uVtobj_\lambda] \to \forfunc[\uVtobj]\) is the inclusion of an open subset of \(\forfunc[\uVtobj]\), the \(\forfunc[\uVtobj_\lambda]\) cover \(\forfunc[\uVtobj]\), and the \Vtalgobj structure on \(\uVtobj_\lambda\) is the subspace \Vtalgobj.

\item[The input terminal condition]
Assume that \tcat has a terminal object, say \(\term{t}\).
The empty family of trials from \(\term{t}\) to \(\uVtobj\) is significant.

\item[The output terminal condition]
Assume that \tcat has a terminal object, say \(\term{t}\).
The empty family of trials from \(\uVtobj\) to \(\term{t}\) is significant.

\item[The empty input condition]
No family of trials is significant.

\item[The empty output condition]
No family of trials is significant.
\end{description}
\end{defn}

\begin{remark}
\begin{enumerate}
\item 
One caveat of this method of specifying forcing conditions is that, due to the functorial nature of a forcing condition, there may be ``unexpected'' significant families.
In the list above we gave, for most of the conditions, some significant families of trials for any pair \((\tobj, \uVtobj)\) (thinking of input forcing conditions).
It is tempting to think that a \umor \(\umor \colon \ufunc(\tobj) \to \forfunc[\uVtobj]\) is forced if \emph{and only if} \(\trials{\umor}\) contains one of these generating families.
This may not be true; for example, if \(\umor\) is forced then \(\umor \ufunc(\tmor)\) must also be forced even if \(\trials{\umor \ufunc(\tmor)}\) doesn't contain a generating family.

However, most of the conditions do have this property: that if a \umor, \(\umor\), is forced then \(\trials{\umor}\) contains one of the given generating families.
The terminal conditions do not have this property, and if the families are not chosen wisely then the specifically\enhyp{}determined conditions may not.

\item 
In the input saturation condition, a \umor \(\umor \colon \ufunc(\tobj) \to \forfunc[\uVtobj]\) is forced if \(\ofn \umor \colon \ufunc(\tobj) \to \ufunc(\tobj')\) lifts to a \uVtmor \(\sfunc(\tobj) \to \sfunc(\tobj')\) for all \(\tobj'\); equivalently, if it comes from a \tmor \(\tobj \to \tobj'\).

\item
The difference between the determined and specifically\enhyp{}determined conditions is that in the former \emph{any} family of morphisms satisfying the requirement may be used in the test.
In the latter, the family of morphisms has to come from a list which is drawn up at the start.
The reason that one might prefer to use this is that the full list of families of morphisms that satisfy the requirement may be difficult to write down.
It can therefore be some work to decide whether or not a given morphism is forced.
With a fixed list, the task becomes much easier.
However, as remarked above, unless these lists are chosen carefully there may still be some morphisms that are forced which cannot be tested merely by the families on the stated lists.

\item Let us consider the output sheaf condition.
Consider a trial

\begin{centre}
\begin{tikzpicture}[node distance=1.5cm, auto,>=latex', thick]
\node (uobj) {\(\uobj\)};
\node[right of=uobj] (uVtobj) {\(\uVtobj\)};
\node[right of=uVtobj] (tobj) {\(\tobj\)};
\node[right of=tobj] (tobj2) {\(\tobj\)};
\path[->]
(uobj) edge node[diaglabel] {\(\iota\)} (uVtobj)
(uVtobj) edge[dotted] (tobj)
(tobj) edge node[diaglabel] {\(=\)} (tobj2);
\end{tikzpicture}
\end{centre}

The statement that \(\umor \colon \forfunc[\uVtobj] \to \ufunc(\tobj)\) succeeds this trial means that \(\umor\restrict_{\uobj} \colon \uobj \to \ufunc(\tobj)\) is an output test morphism for \(\uobj\).
By definition, therefore, there is an output test morphism \(\umor' \colon \forfunc[\uVtobj] \to \ufunc(\tobj)\) such that \(\umor' \restrict_{\uobj} = \umor\restrict_{\uobj}\).

\item The saturation conditions allow us to effectively ignore the corresponding family of test functions when looking at whether a morphism on the underlying \uobjs lifts to a \uFVtmoralt.
For example, suppose that the output test functions are saturated.
Then if \(\umor \colon \forfunc[\uFVtobj_1] \to \forfunc[\uFVtobj_2]\) is such that \(\umor \ifn \in \itest_{\uFVtobj_2}\) for all \(\ifn \in \itest_{\uFVtobj_1}\) then for \(\ofn \in \otest_{\uFVtobj_2}\) and \(\ifn \in \itest_{\uFVtobj_1}\), \(\ofn \umor \ifn\) came from a morphism in \(\tcat\).
Whence \(\ofn \umor\) is forced and hence \(\ofn \umor \in \otest_{\uFVtobj_1}\).
Thus \(\umor\) underlies a \uFVtmor.

\item The empty forcing conditions translate to the fact that no morphisms are forced.

\item The terminal forcing conditions translate to the fact that any morphism which factors through the terminal object of \tcat is forced.

\item The saturation condition is the ``top'' of the family of forcing conditions whilst the empty condition is the ``bottom''.
The forcing conditions thus form a complete (possibly large) lattice.
\end{enumerate}
\end{remark}
\medskip

Let us now translate the examples already ``in the wild'' into our formalism.
We shall also include all of Chen's definitions to provide a wider scope for comparisions.

\paragraph{Fr\"olicher spaces}

\begin{enumerate}
\item The underlying category is simply \(\xcat\).
\item There is one test space, \R.
\item The input forcing conditions is saturation.
\item The output forcing conditions is saturation.
\end{enumerate}

\paragraph{Chen spaces}
Although we have only given Chen's last definition in section~\ref{sec:smoothcats}, we shall give the standard form of all four of his definitions.

\begin{description}
\item[\cite{kc}]

\begin{enumerate}
\item The underlying category is that of all Hausdorff topological spaces.
\item The test spaces are all closed, convex subsets of Euclidean spaces.
\item The input forcing condition is the terminal condition.
\item The output forcing condition is saturation.
\end{enumerate}

\item[\cite{kc5}]

\begin{enumerate}
\item The underlying category is that of all topological spaces.
\item The test spaces are all closed, convex subsets of Euclidean spaces.
\item The input forcing condition is terminal condition.
\item The output forcing condition is saturation.
\end{enumerate}

\item[\cite{kc5}]

\begin{enumerate}
\item The underlying category is that of all topological spaces.
\item The test spaces are all closed, convex subsets of Euclidean spaces.
\item The input forcing condition is the determined and the terminal condition.
\item The output forcing condition is saturation.
\end{enumerate}

\item[\cite{kc3}]

\begin{enumerate}
\item The underlying category is \(\xcat\).
\item The test spaces are all convex subsets of Euclidean spaces.
\item The input forcing condition is the sheaf condition and the terminal condition.
\item The output forcing condition is saturation.
\end{enumerate}
\end{description}

\paragraph{Souriau Spaces}

\begin{enumerate}
\item The underlying category is \(\xcat\).
\item The test spaces are all open subsets of Euclidean spaces.
\item The input forcing condition is the sheaf condition and the terminal condition.
\item The output forcing condition is saturation.
\end{enumerate}

\paragraph{Sikorski Spaces}

\begin{enumerate}
\item The underlying category is that of topological spaces.
\item The test spaces are the Euclidean spaces.
\item The input forcing condition is saturation.
\item The output forcing condition is the sheaf condition, the specifically\enhyp{}determined condition, and the terminal condition.
\end{enumerate}

These spaces are, perhaps, the hardest to see how to make them fit our standard form.
There appears to be one test space, \R, and three conditions that need to be satisfied: that of being an algebra, the locally detectable condition, and the last condition involving \(k\)\enhyp{}tuples of maps.
The locally detectable condition is clearly the output sheaf condition.
The last condition leads us to expand our set of test spaces.
This condition appears to be saying that the composition of a test morphism to \(\R^k\) and a \cimap \(\R^k \to \R\) is again a test morphism.
Thus if we expand the test spaces to all Euclidean spaces, we appear to get this condition for free.
The catch is that we need to force the condition that if \(f_1, \dotsc, f_k\) are test maps to \R then \((f_1, \dotsc, f_k)\) is a test map to \(\R^k\).
This is the \emph{specifically\enhyp{}determined} condition with the coordinate projections as the determining family.
The condition that the functions be an algebra is then almost vaccuous.
For providing that \(\otest(\R)\) is not empty we obtain the structure of an algebra using the smooth maps \(1 \in \Ci(\R,\R)\) and \(\alpha, \mu \in \Ci(\R^2, \R)\) defined by \(1(t) = 1\), \(\alpha(s,t) = s + t\), and \(\mu(s,t) = s \cdot t\).
To ensure that \(\otest(\R)\) is not empty, we impose the output terminal condition.

\paragraph{Smith Spaces}

\begin{enumerate}
\item The underlying category is that of topological spaces.
\item There is one test space, \R.
\item The input forcing condition is saturation.
\item The output forcing condition is saturation.
\end{enumerate}

From the definition, it would appear that Smith requires different families for the input and output test spaces.
However, due to Boman's theorem, \cite{jb3}, it is clear that in the closure condition it is sufficient to consider only those maps from \R.
It is interesting to note that the definition of a Smith space preceeded Boman's result.

\medskip

Let us comment now on the characteristics of the categories appearing in the above.
Firstly, let us consider the test categories.
It is possible to embed these categories as full subcategories of one ``maximal'' category.

\begin{defn}
\emph{\mcatu{}} is the amnestic, transportable construct generated by the following category.
The \mobjs are those subsets \(\mobj \subseteq \R^n\) with the property that each \(\melt \in \mobj\) has an open neighbourhood, say \(V\), in \(\R^n\) and a diffeomorphism \(\psi \colon V \cong U\) with \(U\) also an open subset of \(\R^n\) such that \(\psi(V \cap \mobj)\) is convex.
The \mmoralts are the \cimaps.
\end{defn}

This category contains all convex subsets of finite dimensional affine spaces, all open subsets of affine spaces, all smooth manifolds, and all smooth manifolds with boundary, as well as a good deal else.

Our reason for doing this embedding is that when comparing the categories of smooth objects we shall want to consider modifications of the test category.
By embedding our categories in this way, we can always factor our modifications as passing from a category to a full subcategory or vice versa.
This will simplify the exposition without diminishing its relevance.

Now let us consider \ucat.
In the examples given, it is either \(\xcat\) or some category of topological spaces; either all topological spaces or Hausdorff topological spaces.
One might conceivably wish to restrict ones attention further to, say, regular, normal, paracompact, or metrisable spaces.

When considering how to define a category of smooth spaces, the main distinction is between \(\xcat\) and the others; the question being as to whether ``smoothness'' is a property that is built on continuity or whether it stands in its own right.

For the purpose of comparing the categories, the important factor is in how the categories relate to each other.
The category of Hausdorff spaces is a reflective full subcategory of that of all topological spaces, whilst topological spaces forms a topological category over \(\xcat\).
Thus these are the features that we shall consider.

We shall return briefly to the issue of topology in section~\ref{sec:topology}.

\section{Functors}
\label{sec:functor}

Our purpose now is to define certain functors between our categories.
We shall start by defining the ``obvious'' functors.
Later, we shall show that these contain all the ``interesting'' functors.

The functors fall into three types: change of underlying category, change of family of test spaces, and change of forcing condition.
By composing these functors, we obtain functors between any two of our categories.
In addition, we can restrict our attention to functors where the change is particularly simple.
\tcatu is specified by a subclass of the class of objects of \mcat.
We can therefore order \tcats by inclusion and need only consider the case where one is contained in the other.
Similarly, we can order forcing conditions.
Informally, we say that one forcing condition is less than another if every forced morphism for the first is forced for the second.
Formally, we say that \((\iforce^1, \oforce^1) \preceq (\iforce^2, \oforce^2)\) if \(\iforce^1 \le \iforce^2\) and \(\oforce^1 \le \oforce^2\) (using the ordering on \(\{0 \to 1\}\)).
Since we can ``or'' and ``and'' forcing conditions, we need only consider the case where one forcing condition is less than another.

The functor in one direction is usually straightforward.
For functors in the opposite direction we use a little lattice theory.

\subsection{The Fibre Categories}

Let us fix \ucat, \(\ucat\), \tcat, \(\tcat\), and a forcing condition \((\iforce, \oforce)\).
Let \(\uFVtcat\) be \uFVtcat.
Let \(\uobj\) be an \uobj.
Let \(\uFVtcat_{\uobj}\) be the fibre at \(\uobj\) of the forgetful functor \(\uFVtcat \to \ucat\).
This is a, possibly large, partially ordered class.
(Recall that in the definition of a \uVtobj we insisted that the input and output test functors be strict subfunctors of the hom\hyp{}functor; this ensures that the obvious quasi\hyp{}ordering on \(\uFVtcat_{\uobj}\) is a partial ordering; also if \(\tcat\) has a small skeleton then this class will be an actual set.)

\begin{proposition}
\label{prop:lattice}
\(\uFVtcat_{\uobj}\) is a complete lattice.
\end{proposition}

\begin{proof}
Let \(\family{S}\) be a family in \(\uFVtcat_{\uobj}\).
Let \(\family{S}^c\) be the (possibly empty) family
\[
  \family{S}^c \coloneqq \{\uFVtobj \in \uFVtcat_{\uobj} : \uFVtobj \preceq \uFVtobj' \text{ for all } \uFVtobj' \in \family{S}\}.
\]

For a \tobj, \(\tobj\), define
\begin{align*}
\itest(\tobj) &\coloneqq \bigcap_{\uFVtobj \in \family{S}} \itest_{\uFVtobj}(\tobj), \\
\otest(\tobj) &\coloneqq \bigcap_{\uFVtobj \in \family{S}^c} \otest_{\uFVtobj}(\tobj), \\
\itest'(\tobj) &\coloneqq \{ \ifn \colon \ufunc(\tobj) \to \uobj : \ofn \ifn \in \im \ufunc \text{ for all } \ifn \in \itest\}, \\
\otest'(\tobj) &\coloneqq \{ \ofn \colon \uobj \to \ufunc(\tobj) : \ofn \ifn \in \im \ufunc \text{ for all } \ofn \in \otest\}.
\end{align*}
Although the families may be large, the intersections are all taking place within the \emph{sets} \(\Hom{\ucat}{\ufunc(\tobj)}{\uobj}\) and \(\Hom{\ucat}{\uobj}{\ufunc(\tobj)}\).
In particular, if \(\family{S}^c\) is empty, \(\otest(\tobj) = \Hom{\ucat}{\uobj}{\ufunc(\tobj)}\).

That \(\itest\) and \(\otest\) are subfunctors of the requisite hom\hyp{}functors is obvious: if \(\tmor\) is a \tmor and \(\ifn \in \itest\) then \(\ifn \in \itest_{\uFVtobj}\) for all \(\uFVtobj \in \family{S}\) whence \(\ifn \ufunc(\tmor) \in \itest_{\uFVtobj}\) for all \(\uFVtobj \in \family{S}\) and so \(\ifn \ufunc(\tmor) \in \itest\).
For \(\itest'\), let \(\ifn \in \itest'(\tobj)\) and let \(\tmor \colon \tobj' \to \tobj\) be a \tmor.
Let \(\ofn \in \itest'\).
By definition, \(\ofn \ifn \in \im \ufunc\) whence \(\ofn \ifn \ufunc(\tmor) \in \im \ufunc\).
As this holds for all \(\ofn\), \(\ifn \ufunc(\tmor) \in \itest'(\tobj')\).
Hence \(\itest\) is a functor \(\tcat \to \xcat\).
It is clearly a subfunctor of the requisite hom\hyp{}functor.

By construction, \(\itest\) and \(\otest'\) are compatible, as are \(\otest\) and \(\itest'\).
Hence \((\uobj, \itest, \otest')\) and \((\uobj, \itest', \otest)\) are \uVtobjs.
We observe that, for any \(\uFVtobj \in \family{S}\), \(\itest \subseteq \itest_{\uFVtobj}\) and\emhyp{}by the compatibility condition\emhyp{}\(\otest' \supseteq \otest_{\uFVtobj}\).
Hence the identity on \(\uobj\) lifts to a \uVtmor \((\uobj, \itest, \otest') \to \uFVtobj\) for any \(\uFVtobj \in \family{S}\).
Similarly, the identity on \(\uobj\) lifts to a \uVtmor \(\uFVtobj \to (\uobj, \itest', \otest)\) for any \(\uFVtobj \in \family{S}^c\).

To show that both are \uFVtobjs, we need to show that they satisfy the  forcing condition.
Let us consider \((\uobj, \itest, \otest')\).
Let \(\tobj\) be a \tobj and let \(\umor \colon \ufunc(\tobj) \to \uobj\) be a \umor that is forced for the pair \((\tobj, (\uobj, \itest, \otest'))\).
Let \(\uFVtobj \in \family{S}\).
As the identity on \(\uobj\) lifts to a \uVtmor \((\uobj, \itest, \otest') \to \uFVtobj\), \(\umor\) is also forced for the pair \((\tobj, \uFVtobj)\).
Hence \(\umor \in \itest_{\uFVtobj}\).
As this holds for all \(\uFVtobj \in \family{S}\), \(\umor \in \itest\).
Now let \(\umor \colon \uobj \to \ufunc(\tobj)\) be forced for the pair \((\uobj, \itest, \otest'), \tobj)\).
Let \(\ifn \in \itest(\tobj')\).
By Lemma~\ref{lem:tembed}, \(\ifn\) lifts to a \uVtmor \(\sfunc(\tobj') \to (\uobj, \itest, \otest')\).
Hence \(\umor \ifn \colon \ufunc(\tobj') \to \ufunc(\tobj)\) is forced.
From Definition~\ref{def:forcing}, \(\umor \ifn \in \im \ufunc\).
Hence \(\umor \in \otest'(\tobj)\).
Thus \((\uobj, \itest, \otest')\) is a \uFVtobj.

Similarly, \((\uobj, \itest', \otest)\) is a \uFVtobj.
Clearly, \((\uobj, \itest, \otest') \in \family{S}^c\).
Hence \((\uobj, \itest, \otest') \preceq (\uobj, \itest', \otest)\).
Thus \(\otest \subseteq \otest'\) and \(\itest \subseteq \itest'\); hence \((\uobj, \itest, \otest)\) is a \uVtobj.

For any \(\uFVtobj \in \family{S}\), clearly \(\itest \subseteq \itest_{\uFVtobj}\).
Then for \(\uFVtobj' \in \family{S}^c\), \(\otest_{\uFVtobj} \subseteq \otest_{\uFVtobj'}\) so \(\otest_{\uFVtobj} \subseteq \otest\).
Hence the identity on \(\uobj\) lifts to a \uVtmor \((\uobj, \itest, \otest) \to \uFVtobj\).
Similarly, for \(\uFVtobj \in \family{S}^c\), the identity on \(\uobj\) lifts to a \uVtmor \(\uFVtobj \to (\uobj, \itest, \otest)\).
We therefore have our candidate for the meet of \(\family{S}\).
It remains to show that it is a \uFVtobj.

This is similar to the arguments above for \((\uobj, \itest, \otest')\) and \((\uobj, \itest', \otest)\).
If \(\umor \colon \ufunc(\tobj) \to \uobj\) is forced for \((\tobj, (\uobj, \itest, \otest))\) then \(\umor\) is forced for \((\tobj, \uFVtobj)\) for all \(\uFVtobj \in \family{S}\).
Hence \(\umor \in \itest_{\uFVtobj}\) for all \(\uFVtobj \in \family{S}\), whence \(\umor \in \itest\).
Similarly, \(\otest\) satisfies the forcing condition.
Thus \((\uobj, \itest, \otest)\) is the meet of \(\family{S}\).

It is obvious how to adapt this to define the join of \(\family{S}\).
\end{proof}

It is interesting to see what is the maximum of \(\uFVtcat_{\uobj}\).
From the above proof, it will have test functors
\begin{align*}
\itest(\tobj) &= \Hom{\ucat}{\ufunc(\tobj)}{\uobj}, \\
\otest(\tobj) &= \bigcap_{\uFVtobj \in \uFVtcat_{\uobj}} \otest_{\uFVtobj}(\tobj).
\end{align*}
By sending a \uobj to the maximum and minimum of its fibre categories, we obtain functors from \(\ucat\) to \(\uFVtcat\).

\begin{defn}
Let \(\indfunc \colon \ucat \to \uFVtcat\) and \(\disfunc \colon \ucat \to \uFVtcat\) be the functors
\begin{align*}
\indfunc \colon \uobj &\mapsto \meet \uFVtcat_{\uobj}, \\
\disfunc \colon \uobj &\mapsto \join \uFVtcat_{\uobj}.
\end{align*}
We refer to these as, respectively, the \emph{indiscrete} and \emph{discrete} \(\uFVtcat\)\enhyp{}functors.
\end{defn}

In the following we shall use standard lattice notation.
That is, in a complete lattice \(L\), \(\top\) and \(\bot\) refer to the maximum and minimum respectively, \(\meet\) is the meet (intersection), \(\join\) is the join (union), and we use the standard interval notation, so for \(a, b \in L\) with \(a \preceq b\) we  write \([a,b]\) for \(\{c \in L : a \preceq c \preceq b\}\).
Recall that if \(f \colon L_1 \to L_2\) is an order\hyp{}preserving map then \(a \mapsto \meet f^{-1}[a, \top]\) and \(a \mapsto \join f^{-1}[\bot, a]\) are also order\hyp{}preserving. 

\subsection{Forcing Functors}
\label{sec:force}

For this section, we fix \ucat, \(\ucat\), and \tcat, \(\tcat\).
We choose two forcing conditions, \((\iforce^1, \oforce^1)\) and \((\iforce^2, \oforce^2)\), with \((\iforce^1, \oforce^1) \preceq (\iforce^2, \oforce^2)\).
These define two categories of smooth objects, \(\uAFVtcat\) and \(\uBFVtcat\), which are full subcategories of \uVtcat.

\begin{proposition}
The inclusion functor \(\uBFVtcat \to \uVtcat\) factors through \(\uAFVtcat\).
\end{proposition}

\begin{proof}
Let \(\uBFVtobj\) be an \uBFVtobjalt.
Let \(\tobj\) be a \tobj.
Let \(\umor \colon \ufunc(\tobj) \to \forfunc[\uBFVtobj]\) be a \umor.
Suppose that \(\umor\) is forced by \((\iforce^1, \oforce^1)\).
Then \(\iforce^1(\trials{\umor}) = 1\).
Since \(\iforce^2 \ge \iforce^1\), \(\iforce^2(\trials{\umor}) = 1\) also.
Hence as \(\uBFVtobj\) is an \uBFVtobjalt, \(\umor \in \itest_{\uBFVtobj}(\tobj)\).
The same holds for \umors out of \(\uBFVtobj\), whence \(\uBFVtobj\) is an \uAFVtobjalt.
\end{proof}

\begin{defn}
We write \(\incfunc \colon \uBFVtcat \to \uAFVtcat\) for the inclusion functor.
\end{defn}

Now we shall construct a functor in the opposite direction.
Given a \uAFVtobjalt, we wish to define a ``nearest'' \uBFVtobjalt.
It is obvious that there are\emhyp{}usually\emhyp{}two choices.

\begin{defn}
Define two \emph{forcing functors} \(\uAFVtcat \to \uBFVtcat\) by
\begin{align*}
\Jforfunc \colon \uAFVtobj &\mapsto \join \incfunc^{-1} \left[\bot, \uAFVtobj\right], \\
\Mforfunc \colon \uAFVtobj & \mapsto \meet \incfunc^{-1} \left[\uAFVtobj, \top\right].
\end{align*}
\end{defn}

The idea, if not the fact, of these functors is that \(\Jforfunc(\uAFVtobj)\) should be the nearest \uBFVtobjalt below \(\uAFVtobj\) and \(\Mforfunc(\uAFVtobj)\) should be the nearest \uBFVtobjalt above \(\uAFVtobj\) (comparisions actually happening in \(\uAFVtcat\)).
However, these na\"ive expectations may not be met as it is entirely possible that, for example, \(\incfunc\Jforfunc(\uAFVtobj)\) is actually above \(\uAFVtobj\).

The restriction of \(\incfunc\) to a fibre is the inclusion of one lattice in another and the order on the first is that induced from the second.
From this we can deduce some elementary properties of \(\Jforfunc\) and \(\Mforfunc\).

\begin{lemma}
The compositions \(\Jforfunc\incfunc\) and \(\Mforfunc\incfunc\) are the identity on \(\uBFVtcat\).
For all \uAFVtobjalts[\uAFVtobj] \(\Jforfunc(\uAFVtobj) \preceq \Mforfunc(\uAFVtobj)\).
\end{lemma}

\begin{proof}
The first comes from the fact that, as \(\incfunc\) is an inclusion, for \(\uBFVtobj\) a \uBFVtobjalt, \(\incfunc^{-1} \left[\bot, \incfunc(\uBFVtobj)\right] = \left[\bot, \uBFVtobj\right]\).
Thus \(\Jforfunc \incfunc(\uBFVtobj) = \uBFVtobj\).
The case of \(\Mforfunc\) is similar.

For the second, observe that if \(\uBFVtobj_1 \in \incfunc^{-1} [\bot, \uAFVtobj]\) and \(\uBFVtobj_2 \in \incfunc^{-1} [\uAFVtobj, \bot]\) then \(\incfunc(\uBFVtobj_1) \preceq \incfunc(\uBFVtobj_2)\).
Hence \(\uBFVtobj_1 \preceq \uBFVtobj_2\).
Thus \(\Jforfunc(\uAFVtobj) \preceq \Mforfunc(\uAFVtobj)\).
\end{proof}

We shall be particularly interested in the question of when \(\uAFVtobj \preceq \incfunc \Mforfunc(\uAFVtobj)\) and \(\incfunc \Jforfunc(\uAFVtobj) \preceq \uBFVtobj\) hold.
To do this, we need explicit descriptions of \(\incfunc \Mforfunc(\uAFVtobj)\) and \(\incfunc \Jforfunc(\uAFVtobj)\).

Let \(\uAFVtobj\) be a \uAFVtobjalt.
Let us write \(\mMo{\uAFVtobj}\) for \(\incfunc\Mforfunc(\uAFVtobj)\).
From the proof of Proposition~\ref{prop:lattice}, we see that
\[
  \itest_{\mMo{\uAFVtobj}} = \bigcap \itest_{\uAFVtobj'}
\]
where the indexing family is over \uBFVtobjalts[\uBFVtobj'] such that \(\uAFVtobj \preceq \incfunc(\uBFVtobj')\).
Thus \(\itest_{\uAFVtobj} \subseteq \itest_{\uBFVtobj'}\) and so \(\itest_{\uAFVtobj} \subseteq \itest_{\mMo{\uAFVtobj}}\).
Thus to test whether or not \(\uAFVtobj \preceq \mMo{\uAFVtobj}\) it is sufficient to test whether or not \(\otest_{\mMo{\uAFVtobj}} \subseteq \otest_{\uAFVtobj}\).
This is not guaranteed\emhyp{}we shall see some examples later\emhyp{}but we can give some conditions for when it does hold.

\begin{proposition}
\label{prop:foradj}
\begin{enumerate}
\item If the output forcing condition \(\mBo{\oforce}\) is independent of the output test functor then \(\uAFVtobj \preceq \mMo{\uAFVtobj}\) if and only if \(\mMo{\uAFVtobj} = (\forfunc[\uAFVtobj], \itest_{\mMo{\uAFVtobj}}, \otest_{\uAFVtobj})\).
\item If the output forcing conditions \(\mAo{\oforce}\) and \(\mBo{\oforce}\) are the same then \(\mMo{\uAFVtobj} = (\forfunc[\uAFVtobj], \itest_{\mMo{\uAFVtobj}}, \otest_{\uAFVtobj})\).
\end{enumerate}
\end{proposition}

\begin{proof}
The key to both parts is the same: showing that \((\forfunc[\uAFVtobj], \itest_{\mMo{\uAFVtobj}}, \otest_{\uAFVtobj})\) is a \uBFVtobjalt.
Once this is shown, it is obviously \(\mMo{\uAFVtobj}\).
In particular, in the first condition the reverse implication is obvious.
Let us consider the conditions in turn.

\begin{enumerate}
\item If \(\uAFVtobj \preceq \mMo{\uAFVtobj}\) then \(\mMo{\uAFVtobj}\) is a \uBFVtobjalt with the property that \(\otest_{\mMo{\uAFVtobj}}(\tobj) \subseteq \otest_{\uAFVtobj}(\tobj)\) for all \tobjs[\tobj].
Suppose that \(\umor \colon \forfunc[\uAFVtobj] \to \tobj\) is forced (via \(\mBo{\oforce}\)) for \((\forfunc[\uAFVtobj], \itest_{\mMo{\uAFVtobj}}, \otest_{\uAFVtobj})\).
By assumption, it is therefore also forced for \((\forfunc[\uAFVtobj], \itest_{\mMo{\uAFVtobj}}, \otest_{\mMo{\uAFVtobj}}) = \mMo{\uAFVtobj}\).
Hence \(\umor \in \otest_{\mMo{\uAFVtobj}}(\tobj)\) whence \(\umor \in \otest_{\uAFVtobj}(\tobj)\).

Now consider the input side.
Suppose that \(\umor \colon \tobj \to \forfunc[\uAFVtobj]\) is forced (via \(\mBo{\iforce}\)) for \((\forfunc[\uAFVtobj], \itest_{\mMo{\uAFVtobj}}, \otest_{\uAFVtobj})\).
Then as the identity on \(\forfunc[\uAFVtobj]\) lifts to a \uVtmor \(\uVtmor \colon (\forfunc[\uAFVtobj], \itest_{\mMo{\uAFVtobj}}, \otest_{\uAFVtobj}) \to \mMo{\uAFVtobj}\), \(\umor = \forfunc[\uVtmor]\umor\) is forced for \(\mMo{\uAFVtobj}\).
Thus as \(\mMo{\uAFVtobj}\) is a \uBFVtobjalt, \(\umor \in \itest_{\mMo{\uAFVtobj}}(\tobj)\).

Hence \((\forfunc[\uAFVtobj], \itest_{\mMo{\uAFVtobj}}, \otest_{\uAFVtobj})\) is a \uBFVtobjalt.

\item Suppose that  \(\umor \colon \forfunc[\uAFVtobj] \to \tobj\) is forced (via \(\mBo{\oforce}\)) for \((\forfunc[\uAFVtobj], \itest_{\mMo{\uAFVtobj}}, \otest_{\uAFVtobj})\).
By assumption, it is therefore also forced via \(\mAo{\oforce}\).
Since the identity on \(\forfunc[\uAFVtobj]\) lifts to a \uVtmor \(\uVtmor \colon \uAFVtobj \to (\forfunc[\uAFVtobj], \itest_{\mMo{\uAFVtobj}}, \otest_{\uAFVtobj})\), \(\umor = \umor \forfunc[\uVtmor]\) is forced for \(\uAFVtobj\).
Thus as \(\mMo{\uAFVtobj}\) is a \uAFVtobjalt, \(\umor \in \otest_{\uAFVtobj}(\tobj)\).

Now consider the input side.
Observe that if \(\uBFVtobj'\) is a \uBFVtobjalt such that \(\uAFVtobj \preceq \uBFVtobj'\) then \((\forfunc[\uAFVtobj], \itest_{\mMo{\uAFVtobj}}, \otest_{\uAFVtobj}) \preceq \uBFVtobj'\) (comparisions in \(\uVtcat\) for simplicity).
Therefor if \(\umor \colon \tobj \to \forfunc[\uAFVtobj]\) is forced for \((\forfunc[\uAFVtobj], \itest_{\mMo{\uAFVtobj}}, \otest_{\uAFVtobj})\) then it is forced for \(\uBFVtobj'\).
Hence it is in \(\itest_{\mMo{\uAFVtobj}}(\tobj)\) as this is the intersection of the corresponding \(\itest_{\uBFVtobj'}(\tobj)\).

Hence \((\forfunc[\uAFVtobj], \itest_{\mMo{\uAFVtobj}}, \otest_{\uAFVtobj})\) is a \uBFVtobjalt. \qedhere
\end{enumerate}
\end{proof}

A particular example of when the first condition holds is when the output forcing condition is saturation.
The second example shows that if we change the forcing conditions one component at a time then we get good control over how the changes occur.

There are obvious analogues for the input forcing condition.

Let us conclude this section with an example of when \(\mJo{\uAFVtobj} = \Jforfunc(\uAFVtobj) \preceq \uAFVtobj\) fails.
In this example, \ucat is \xcat and \tcat is the category of open subsets of \R with \cimaps between them.
Both forcing conditions have saturation as output forcing condition.
The weaker forcing condition has no input forcing condition whilst the stronger has the sheaf condition.
As the output forcing conditions are saturation, both an \uAFVtobjalt and an \uBFVtobjalt are determined by their input test functions.

Consider the \uAFVtobjalt[\uAFVtobj] with \(\forfunc[\uAFVtobj] = \R\) and \(\itest_{\uAFVtobj}(\tobj)\) those \cimaps \(\ifn \colon \tobj \to \R\) which are bounded.
To find \(\mJo{\uAFVtobj}\) we take the join of all \uBFVtobjalts below \(\uAFVtobj\) with the same \uobj.
For any bounded open interval, \(I \subseteq \R\), we can define a \uBFVtobjalt[\uBFVtobj_I] with \(\forfunc[\uBFVtobj_I] = \R\) and \(\itest_{\uBFVtobj_I}(\tobj)\) those \cimaps \(\ifn \colon \tobj \to \R\) which factor through \(I\).
This is a \uBFVtobjalt.
(Note that we have not assumed the constant forcing condition on inputs, if we had we would have to include constant maps but this makes so substantial difference to the example.)
Moreover, this \uBFVtobjalt is below \(\uAFVtobj\).
However, the identity on \R is locally in some \(\uBFVtobj_I\) and so, because of the sheaf condition, is in the join of the \(\uBFVtobj_I\).
Hence \(\mJo{\uAFVtobj}\) is at least the \uBFVtobjalt with input test functor \(\itest(\tobj) = \Ci(\tobj,\R)\).
In fact, it is exactly that.

Thus \(\mJo{\uAFVtobj} \not\preceq \uAFVtobj\).

\subsection{Change of Test Spaces}

In this section we wish to examine what happens when we change the test spaces.
\ucatu remains the same and we choose two categories of test spaces, \(\Atcat\) and \(\Btcat\).
We therefore obtain \uVAtcat, \(\uVAtcat\), and \(\uVBtcat\).

As remarked at the start of this section, since we are viewing our \tcats as being subcategories of \mcat, we restrict to the case where \(\Atcat\) is a subcategory of \(\Btcat\).
Restriction of the test functors defines an obvious functor \(\resfunc \colon \uVBtcat \to \uVAtcat\).

As the forcing condition depends slightly on \tcat, we must consider how to relate the two.
Clearly, we wish to meddle with this as little as possible since we can apply a forcing functor afterwards.

Consider \rcats, \(\Arcat\) and \(\Brcat\).
There is an obvious inclusion functor \(\Arcat \to \Brcat\) induced from the inclusion \(\Atcat \to \Btcat\).
We shall not give this a symbol, leaving to context the r\^ole of distinguishing.
This has the property that if \(\tobj\) is a \tobj in \(\Atcat\), \(\uVtobj\) a \uVtobj, and \(\umor \colon \ufunc(\tobj) \to \forfunc[\uVtobj]\) a \umor, then \(\trials[_1]{\umor} \subseteq \trials[_2]{\umor}\).

Let \((\iforce^2, \oforce^2)\) be a forcing condition with respect to \(\Btcat\).
Then we can restrict \(\iforce^2\) and \(\oforce^2\) to \(\Arcat\).
Let us call the resulting functors \(\iforce^1\) and \(\oforce^1\).
If, for \(\tobj_1\) and \(\tobj_2\) in \(\Atcat\), \(\umor \colon \ufunc(\tobj_1) \to \ufunc(\tobj_2)\) is not in the image of \(\ufunc\) then \(\iforce^1(\trials[_1]{\umor}) = \iforce^2(\trials[_1]{\umor}) \le \iforce^2(\trials[_2]{\umor}) = 0\), and similarly for \(\oforce^1\).
Hence \((\iforce^1, \oforce^1)\) is a forcing condition.

Thus we have \(\uFVAtcat\), and \(\uFVBtcat\).

\begin{proposition}
The restriction of a \uFVBtobj is a \uFVAtobj.
\end{proposition}

\begin{proof}
Let \(\uFVBtobj\) be a \uFVBtobj.
Let us consider the input test functor.
Let \(\Atobj\) be a \Atobj in \(\Atcat\).
Let \(\umor \colon \ufunc(\tobj) \to \forfunc[\uFVBtobj]\) be a \umor which is forced via \(\iforce^1\).
Then \(\iforce^1(\trials[_1]{\umor}) = 1\).
By definition, therefore, \(\iforce^2(\trials[_1]{\umor}) = 1\).
Since \(\trials[_1]{\umor} \subseteq \trials[_2]{\umor}\), \(\umor\) is forced via \(\iforce^2\).
Hence \(\umor \in \itest_{\uFVBtobj}\).
Since the source \(\umor\) is a \Atobj from \(\Atcat\), it persists in the restriction.
Hence the restriction of \(\uFVBtobj\) is again a \uFVAtobj.
\end{proof}

Since \(\resfunc\) preserves the underlying \uobjs, for each \uobj[\uobj] it restricts to a functor, i.e.\ an order\hyp{}preserving map, \(\uFVBtcat_{\uobj} \to \uFVAtcat_{\uobj}\).
We can therefore define two reverse functors using the lattice structure of \(\uFVBtcat_{\uobj}\) in a similar fashion to the change of forcing condition functors.

\begin{defn}
Define two \emph{extension functors} \(\uFVAtcat \to \uFVBtcat\) by
\begin{align*}
  \Jextfunc \colon \uFVAtobj &\mapsto \join \resfunc^{-1}\left[\bot,\uFVAtobj\right], \\
  \Mextfunc \colon \uFVAtobj &\mapsto \meet \resfunc^{-1}\left[\uFVAtobj,\top\right].
\end{align*}
\end{defn}

In studying functors from \(\uFVAtcat\) to \(\uFVBtcat\) one encounters an obvious question: if \(\Btobj\) is a \Btobj in \(\Btcat\) that is \emph{not} in \(\Atcat\), which \umors \(\ufunc(\tobj) \to \forfunc[\uFVAtobj]\) should be included?
The functors \(\Mextfunc\) and \(\Jextfunc\) are intended to give, respectively, the minimum and maximum answers to this question.
However, as with the forcing functors, these intentions are not always carried out.

Under a mild assumption on the relationship between \(\Atcat\) and \(\Btcat\) we can factor \(\Jextfunc\) and \(\Mextfunc\) through \(\Jforfunc\) and \(\Mforfunc\) respectively.
This assumption is closely related to the concept of an \emph{adequate} subcategory as studied in \cite{ji3}.
In essence, it says that when looking in \(\ucat\), \(\Btcat\) can be determined by looking at morphisms to and from \Atobjalts.

\begin{defn}
We say that \(\Atcat\) is \emph{\(\ucat\)\enhyp{}adequate} in \(\Btcat\) if, in \(\ucat\), it determines \(\im \ufunc\).
That is to say, if \(\Btobj_1\) and \(\Btobj_2\) are \Btobjalts and \(\umor \colon \ufunc(\tobj_1) \to \ufunc(\tobj_2)\) is a \umor which is \emph{not} in the image of \(\ufunc\) then there are \Atobjalts \(\Atobj_1'\) and \(\Atobj_2'\) and morphisms \(\Btmor_1 \colon \Atobj_1' \to \Btobj_1\), \(\Btmor_2 \colon \Btobj_2 \to \Atobj_2'\) such that \(\ufunc(\Btmor_2) \umor \ufunc(\Btmor_1)\) is not in the image of \(\ufunc\).
\end{defn}

In the following we shall assume that this condition holds.
This allows us to extend the embedding of \(\Atcat\) in \(\uAFVtcat\) to the whole of \(\Btcat\).
As this is an extension of \(\sfunc \colon \Atcat \to \uAFVtcat\) we shall use the same symbol.

\begin{lemma}
\label{lem:tembedx}
There is a functor \(\sfunc \colon \Btcat \to \uVAtcat\) which embeds \(\Btcat\) as a full subcategory of \(\uVAtcat\).
The \uVAtobjalt[\sfunc(\Btobj)] has underlying \uobj[\ufunc(\Btobj)],  \itest \(\itest(\Atobj') = \Hom{\Btcat}{\Atobj'}{\Btobj}\), and \otest \(\otest(\Atobj') = \Hom{\Btcat}{\Btobj}{\Atobj'}\).
On morphisms, \(\sfunc(\Btmor)\) is determined by the requirement that \(\forfunc[\sfunc(\Btmor)] = \ufunc(\Btmor)\).

The restriction of this functor to \(\Atcat\) is the functor from Lemma~\ref{lem:tembed}.
\end{lemma}

\begin{proof}
The only part we need to worry about is showing that \(\Btcat\) embeds as a \emph{full} subcategory of \(\AuVtcat\).
This is where we need our assumption.

We know that it is true when restricted to \(\Atcat\).
Let \(\Btobj_1\) and \(\Btobj_2\) be \Btobjalts.
Let \(\umor \colon \forfunc[\Btobj_1] \to \forfunc[\Btobj_2]\) be a \umor which is not in the image of \(\ufunc\).
Then, by assumption, there are \Btmors \(\Btmor_1 \colon \Atobj_1' \to \Btobj_1\) and \(\Btmor_2 \colon \Btobj_2 \to \Atobj_2'\) such that \(\ufunc(\Btmor_2) \umor \ufunc(\Btmor_1)\) is not in the image of \(\ufunc\).
However from Lemma~\ref{lem:tembed}, \(\Btmor_1\) underlies a morphism of \uVBtobjalts, \(\sfunc(\Atobj_1') \to \sfunc(\Btobj_1)\); similarly for \(\Btmor_2\).
Thus if \(\umor\) underlay a \uVBtmor, we would have that \(\ufunc(\Btmor_2)\umor \ufunc(\Btmor_1)\) underlay a \uVBtmor from \(\sfunc(\Atobj_1')\) to \(\sfunc(\Atobj_2')\).
As these come from \(\Atcat\), we know that this would mean that \(\ufunc(\Btmor_2)\umor \ufunc(\Btmor_1)\) lay in the image of \(\ufunc\), a contradiction.
Hence \(\sfunc\) is full.
\end{proof}

Let us write \(\PJextfunc\) and \(\PMextfunc\) for the extension functors \(\uVAtcat \to \uVBtcat\); i.e.\ when there are no forcing conditions.
Under our assumption on the relationship of \(\Atcat\) to \(\Btcat\) we can give explicit descriptions of \(\PJextfunc(\uVAtobj)\) and \(\PMextfunc(\uVAtobj)\).

\begin{proposition}
\label{prop:extpre}
Let \(\uVAtobj\) be a \uVAtobj.
For \(\Btobj\) a \Btobjalt, define
\begin{align*}
  \itest_a(\Btobj) &\coloneqq \Hom{\uVAtcat}{\sfunc(\Btobj)}{\uVAtobj}, \\
  \otest_a(\Btobj) &\coloneqq \{\ofn \colon \forfunc[\uVAtobj] \to \ufunc(\Btobj) : \ofn = \ufunc(\Btmor)\ofn' \text{ for some } \Atobj' \in \Atcat, \ofn' \in \otest_{\uVAtobj}(\Atobj'), \Btmor \colon \Btobj \to \Atobj'\}, \\
  \itest_b(\Btobj) &\coloneqq \{\ifn \colon \ufunc(\Btobj) \to \forfunc[\uVAtobj] : \ifn = \ifn' \ufunc(\Btmor) \text{ for some } \Atobj' \in \Atcat, \ifn' \in \itest_{\uVAtobj}(\Atobj'), \Btmor \colon \Atobj' \to \Btobj\}, \\
  \otest_b(\Btobj) &\coloneqq \Hom{\uVAtcat}{\uVAtobj}{\sfunc(\Btobj)}.
\end{align*}
Then
\begin{align*}
\PJextfunc(\uVAtobj) &= (\forfunc[\uVAtobj], \itest_a, \otest_a), \text{ and } \\
\PMextfunc(\uVAtobj) &= (\forfunc[\uVAtobj], \itest_b, \otest_b).
\end{align*}
\end{proposition}

\begin{proof}
There is an obvious symmetry here so we shall concentrate on \((\forfunc[\uVAtobj], \itest_a, \otest_a)\).
Let us start by showing that this is a \uVBtobj.
It is clear that the test functors are subfunctors of the requisite hom\hyp{}functors.
Therefore we just need to check the compatibility condition.
Let \(\ifn \in \itest_a(\Btobj)\) and \(\ofn \in \otest_a(\Btobj')\).
Then \(\ofn = \ufunc(\tmor) \ofn'\) for some \(\Atobj'' \in \Atcat\), \(\ofn' \in \otest_{\uVAtobj}(\Atobj'')\), and \(\tmor \colon \Btobj' \to \Atobj''\).
As \(\ofn' \in \otest_{\uVAtobj}(\Atobj'')\) it underlies a \uVBtmor \(\uVAtobj \to \sfunc(\Atobj'')\).
Hence \(\ofn' \ifn\) underlies a \uVBtmor \(\sfunc(\Btobj) \to \sfunc(\Atobj'')\) which, as \(\sfunc\) is full, comes from a morphism in \(\Btcat\).
Hence \(\ofn \ifn = \ufunc(\Btmor) \ofn' \ifn \in \im \ufunc\).

For \(\Atobj\) an \Atobjalt,
\[
  \itest_a(\Atobj) = \Hom{\uVAtcat}{\sfunc(\Atobj)}{\uVAtobj} = \itest_{\uVAtobj}(\Atobj),
\]
whilst in \(\otest_a(\Atobj)\) we can take \(1_{\Atobj}\) as the auxilliary morphism whence \(\otest_a(\Atobj) = \otest_{\uVAtobj}(\Atobj)\).
Hence
\[
  \resfunc(\forfunc[\uVAtobj], \itest_a, \otest_a) = \uVAtobj.
\]
We therefore have \((\forfunc[\uVAtobj], \itest_a, \otest_a) \preceq \PJextfunc(\uVAtobj)\).

Let \(\uVAtobj'\) be a \uVAtobjalt with \(\forfunc[\uVAtobj'] = \forfunc[\uVAtobj]\) such that \(\resfunc(\uVAtobj') \preceq \uVAtobj\).
Then for \(\Atobj\) an \Atobjalt, \(\itest_{\uVAtobj'}(\Atobj) \subseteq \itest_{\uVAtobj}(\Atobj)\) and \(\otest_{\uVAtobj'}(\Atobj) \supseteq \otest_{\uVAtobj}(\Atobj)\).
Let \(\Btobj'\) be an \Btobjalt.
Let \(\ofn \in \otest_{\uVAtobj}(\Atobj)\) and \(\Btmor \colon \Atobj \to \Btobj'\) a \Btmor.
Then \(\ofn \in \otest_{\uVAtobj'}(\Atobj)\) so, by functorality, \(\ufunc(\Btmor) \ofn \in \otest_{\uVBtobj'}(\Atobj)\).
Hence \(\otest_{\uVAtobj'}(\Btobj') \supseteq \otest_a(\Btobj')\).
Let \(\ifn \in \itest_{\uVAtobj'}(\Btobj')\).
Let \(\Btmor \colon \Atobj \to \Btobj'\) be a \Btmor.
Then \(\ifn \ufunc(\Btmor) \in \itest_{\uVAtobj'}(\Atobj) \subseteq \itest_{\uVAtobj}(\Atobj)\).
Hence \(\ifn \colon \ufunc(\Btobj') \to \forfunc[\uVAtobj]\) underlies a \uVBtmor \(\sfunc(\Btobj') \to \uVAtobj\), whence \(\ifn \in \itest_a(\Btobj')\).
Thus \(\itest_{\uVAtobj'}(\Btobj') \subseteq \itest_a(\Btobj')\).
Hence \(\uVAtobj' \preceq (\forfunc[\uVAtobj], \itest_a, \otest_a)\).

Thus \((\forfunc[\uVAtobj], \itest_a, \otest_a) = \join \resfunc^{-1}\left[\bot, \uVAtobj\right] = \PJextfunc(\uVAtobj)\).
\end{proof}

From this we can deduce a useful factorisation for when we do have a forcing condition.

\begin{corollary}
\(\Jextfunc = \Jforfunc \Jextfunc_\pre\) and \(\Mextfunc = \Mforfunc \Mextfunc_\pre\).
\end{corollary}

\begin{proof}
As part of the previous proof we showed that \(\resfunc \Jextfunc_\pre(\uVAtobj) = \uVAtobj\).
Thus
\[
  \resfunc^{-1} \left[\bot, \uVAtobj\right] = \left[\bot, \Jextfunc_\pre(\uVAtobj)\right].
\]
If we are extremely careful on the functors, we ought to write
\[
  \Jextfunc(\uVAtobj) = \join (\resfunc \incfunc)^{-1} \left[\bot,\uVAtobj \right].
\]
From which we deduce that
\begin{align*}
  \Jextfunc(\uVAtobj) &= \join \incfunc^{-1} \resfunc^{-1} \left[\bot, \uVAtobj \right] \\
&= \join \incfunc^{-1} \left[\bot, \Jextfunc_\pre(\uVAtobj) \right] \\
&= \Jforfunc \Jextfunc_\pre(\uVAtobj).
\end{align*}
Equality on morphisms is a formality.
\end{proof}

The idea here is that \(\Jextfunc_\pre(\uVAtobj)\) ought to be the maximum extension of \(\uVAtobj\) by a \uFVBtobj.
Therefore any extension of \(\uVAtobj\) to a \uFVBtobj must lie below it.
Since \(\Jextfunc(\uVAtobj)\) is meant to be the maximum such \uFVBtobj, we can find it by looking at the ``nearest'' \uFVBtobj below \(\Jextfunc_\pre(\uVAtobj)\).
In other words, by applying \(\Jforfunc\).

Since we know that it is not always true that \(\Jforfunc(\uFVBtobj) \preceq \uFVBtobj\), the obvious question is whether or not \(\Jextfunc(\uFVAtobj) \preceq \PJextfunc(\uFVAtobj)\).

Before proving this we observe that from proposition~\ref{prop:extpre} we obtain another extension functor which will help us establish the relationships between the other various extension functors.

\begin{lemma}
The assignment \(\uVAtobj \mapsto (\forfunc[\uVAtobj], \itest_a, \otest_b)\) defines another functor \(\extfunc \colon \uFVAtcat \to \uFVBtcat\).
This functor has the following properties:
\begin{enumerate}
\item \(\resfunc \extfunc\) is the identity on \(\uFVAtcat\),
\item \(\Mextfunc(\uFVAtobj) \preceq \extfunc(\uFVAtobj) \preceq \Jextfunc(\uFVAtobj)\) for all \uFVAtobjs[\uFVAtobj], and
\item \(\extfunc \sfunc = \sfunc \colon \Btcat \to \uFVBtcat\).
\end{enumerate}
\end{lemma}

\begin{proof}
That this is a \uVBtobj comes from the fullness of \(\sfunc \colon \Btcat \to \uVAtcat\).
Composition defines a map
\[
  \Hom{\uVAtcat}{\sfunc(\Btobj')}{\uFVAtobj} \times   \Hom{\uVAtcat}{\uFVAtobj}{\sfunc(\Btobj)} \to \Hom{\uVAtcat}{\sfunc(\Btobj')}{\sfunc(\Btobj)} \cong \Hom{\Btcat}{\Btobj'}{\Btobj}.
\]
It is clear that the restriction of this to \Atcat is \(\uFVAtobj\).

To show that it satisfies the forcing conditions, let \(\ifn \colon \ufunc(\Btobj) \to \forfunc[\uFVAtobj]\) be forced for this \uVBtobjalt.
Then for any \Btmor \(\Btmor \colon \Atobj' \to \Btobj\) for an \Atobj[\Atobj'] the composition \(\ifn \Btmor\) is forced.
Since \(\uFVAtobj\) satisfies the forcing conditions with respect to \Atcat, we therefore have that \(\ifn \Btmor \in \itest_{\uFVAtobj}(\Atobj')\).
As this holds for all such \(\Btmor\), \(\ifn\) underlies a \uVAtmor \(\sfunc(\Btobj) \to \uFVAtobj\).
It is therefore in \(\itest_a(\Btobj)\).
The case for the output forcing condition is similar, and hence we have a \uFVBtobjalt.

The properties are obvious.
\end{proof}

\begin{lemma}
For all \uFVAtobjs[\uFVAtobj],
\[
  \PMextfunc(\uFVAtobj) \preceq \Mextfunc(\uFVAtobj) \preceq \extfunc(\uFVAtobj) \preceq \Jextfunc(\uFVAtobj) \preceq \PJextfunc(\uFVAtobj).
\]
\end{lemma}

\begin{proof}
The relationship of \(\extfunc\) to the others is straightforward.
From its description we have
\[
  \PMextfunc(\uFVAtobj) \preceq \extfunc(\uFVAtobj) \preceq \PJextfunc(\uFVAtobj).
\]
Hence as \(\Mextfunc = \Mforfunc \PMextfunc\) and \(\Jextfunc = \Jforfunc \PJextfunc\) we see that
\[
  \Mextfunc(\uFVAtobj) \preceq \extfunc(\uFVAtobj) \preceq \Jextfunc(\uFVAtobj).
\]

The two outer relations are proved in the same fashion as each other so let us take the right\hyp{}hand one.
From section~\ref{sec:force} we see that this depends on the input test functors.
Let us write \(\mMo{\uFVAtobj}\) for \(\Jextfunc(\uFVAtobj)\).
Consider the \uVBtobj[\widehat{\uFVAtobj}] with \uobj[{\forfunc[\uFVAtobj]}], input test functor \(\itest_a\), and output test functor \(\itest_{\mMo{\uFVAtobj}}\).
Note that \(\itest_a\) is the input test functor for \(\PJextfunc(\uFVAtobj)\) and so \(\widehat{\uFVAtobj} \preceq \PJextfunc(\uFVAtobj)\).

Suppose that \(\ifn \colon \ufunc(\Btobj) \to \forfunc[\uFVAtobj]\) is forced for \(\widehat{\uFVAtobj}\).
Then for every \Atobj[\Atobj'] and \tmor \(\tmor \colon \Atobj' \to \Btobj\), \(\ifn \ufunc(\tmor) \colon \ufunc(\Atobj') \to \forfunc[\uFVAtobj]\) is forced for \(\widehat{\uFVAtobj}\) and hence for \(\resfunc(\widehat{\uFVAtobj})\).

It is possible to show that \(\resfunc(\widehat{\uFVAtobj}) = \uFVAtobj\) but for this step it is sufficient to note that as \(\widehat{\uFVAtobj} \preceq \PJextfunc(\uFVAtobj)\), the identity on \(\forfunc[\uFVAtobj]\) lifts to a \uFVAtmor \(\resfunc(\widehat{\uFVAtobj}) \to \uFVAtobj\).
Thus \(\ifn \ufunc(\tmor) \colon \ufunc(\Atobj') \to \forfunc[\uFVAtobj]\) is forced for \(\uFVAtobj\) and thus in \(\itest_{\uFVAtobj}(\Atobj')\).
This is sufficient to show that \(\ifn \colon \ufunc(\Btobj) \to \forfunc[\uFVAtobj]\) lifts to a \uVAtmor \(\sfunc(\Btobj) \to \uFVAtobj\) since for \(\ofn \in \otest_{\uFVAtobj}(\Atobj'')\) and \(\tmor \colon \Atobj' \to \Btobj\), \(\ofn \ifn \tmor \in \Hom{\Btcat}{\Atobj'}{\Atobj''}\) and hence \(\ofn \ifn \in \otest_{\sfunc(\Btobj)}(\Atobj'')\) by the assumption that \Atcat is \(\ucat\)\enhyp{}adequate in \Btcat.

Thus \(\ifn \in \itest_a(\Btobj)\) and so \(\widehat{\uFVAtobj}\) satisfies the input forcing condition.

Now suppose that \(\uFVBtobj'\) is a \uFVBtobj with underlying \uobj[{\forfunc[\uFVAtobj]}] such that the identity on \(\forfunc[\uFVAtobj]\) lifts to a \uVBtmor \(\uFVBtobj' \to \PJextfunc(\uFVAtobj)\).
Then \(\itest_{\uFVBtobj'} \subseteq \itest_a\) and \(\otest_{\uFVBtobj'} \supseteq \otest_{\mMo{\uFVAtobj}}\).
Hence the identity on \(\forfunc[\uFVAtobj]\) lifts to a \uVBtmor \(\uFVBtobj' \to \widehat{\uFVAtobj}\).
Thus if \(\ofn \colon \forfunc[\uFVAtobj] \to \ufunc(\Btobj)\) is forced for \(\widehat{\uFVAtobj}\), \(\ofn\) is forced for \(\uFVBtobj'\) as well.
Hence as \(\uFVBtobj'\) is a \uFVBtobj, \(\ofn \in \otest_{\uFVBtobj'}(\Btobj)\).
As this holds for all such \(\uFVBtobj'\), \(\ofn \in \otest_{\widehat{\uFVAtobj}}(\Btobj)\) and thus \(\widehat{\uFVAtobj}\) satisfies the output forcing condition.

Hence \(\widehat{\uFVAtobj}\) is a \uFVBtobj and so is \(\Jextfunc(\uFVAtobj)\).
Thus \(\Jextfunc(\uFVAtobj) \preceq \PJextfunc(\uFVAtobj)\).
\end{proof}

\begin{corollary}
\(\resfunc \Jextfunc\) and \(\resfunc \Mextfunc\) are the identity on \(\uFVAtcat\) and \(\Mextfunc \resfunc(\uFVBtobj) \preceq \uFVBtobj \preceq \Jextfunc \resfunc(\uFVBtobj)\) for all \uFVBtobjalts[\uFVBtobj].
 \noproof
\end{corollary}

One further interesting question is as to when \(\Jextfunc\) and \(\Mextfunc\) agree, or if either is the same as \(\extfunc\).
It transpires that it is sufficient to check this on the image of \(\sfunc \colon \Btcat \to \uFVAtcat\).

\begin{proposition}
If any of \(\Jextfunc\), \(\Mextfunc\), and \(\extfunc\) agree on the image of \(\sfunc \colon \Btcat \to \uFVAtcat\) then they agree on the whole of \(\uFVAtcat\).
\end{proposition}

\begin{proof}
Suppose that, say, \(\Jextfunc\) and \(\extfunc\) agree on the image of \(\sfunc\).
Since, for a \uFVAtobj[\uFVAtobj], \(\extfunc(\uFVAtobj) \preceq \Jextfunc(\uFVAtobj) \preceq \PJextfunc(\uFVAtobj)\), and \(\extfunc(\uFVAtobj)\) and \(\PJextfunc(\uFVAtobj)\) have the same input test functor, to show that \(\Jextfunc\) and \(\extfunc\) agree it is sufficient to show that whenever \(\ofn \colon \forfunc[\uFVAtobj] \to \ufunc(\Btobj)\) is in \(\otest_{\extfunc(\uFVAtobj)}(\Btobj)\) then it is in \(\otest_{\Jextfunc(\uFVAtobj)}(\Btobj)\).

Let \(\ofn \in \otest_{\extfunc(\uFVAtobj)}(\Btobj)\).
Then for every \(\Btmor \colon \Btobj \to \Atobj'\) with \(\Atobj'\) an \Atobjalt, \(\Btmor \ofn \in \otest_{\uFVAtobj}(\Atobj')\).
Hence \(\ofn\) underlies a \uVAtmor \(\uFVAtobj \to \sfunc(\Btobj)\).
Applying \(\Jextfunc\) we obtain a \uVBtmor \(\Jextfunc(\uFVAtobj) \to \Jextfunc \sfunc(\Btobj)\).
By assumption, \(\Jextfunc \sfunc(\Btobj) = \extfunc \sfunc(\Btobj)\).
Since \(\extfunc \sfunc(\Btobj) = \sfunc(\Btobj)\), we deduce that \(\ofn\) defines a \uVBtmor \(\Jextfunc(\uFVAtobj) \to \sfunc(\Btobj)\) and hence \(\ofn \in \otest_{\Jextfunc(\uFVAtobj)}\).

Thus \(\Jextfunc = \extfunc\) as required.

The case for \(\Mextfunc\) and \(\extfunc\) is similar.
Then if \(\Mextfunc\) and \(\Jextfunc\) agree on the image of \(\sfunc\) then they both agree with \(\extfunc\) and hence they agree on the whole of \(\uFVAtcat\).
\end{proof}

As a concluding remark for this section, let us consider a subtly different question on extensions.
We have assumed in this section that we have a forcing condition on the larger test category which we restrict to the smaller.
It is possible to consider the situation from the opposite angle: suppose we have a forcing condition on the smaller test category which we wish to extend to the larger.
Can this be done, and in how many ways?

The work of this section shows that it can be done, and that the two extremes are straightforward to describe.
The minimal extension of the forcing conditions is, say for the input forcing condition, to say that a morphism \(\umor \colon \ufunc(\Btobj) \to \forfunc[\uVBtobj]\) is forced if there is a factorisation of \(\umor\) as \(\umor' \ufunc(\Btmor)\) where \(\Btmor \colon \Btobj \to \Atobj'\) is such that \(\Atobj'\) is an \Atobjalt and \(\umor' \colon \ufunc(\Atobj') \to \forfunc[\uVBtobj]\) is forced.
This is the minimum that can be done to preserve the fact that the forcing conditions are defined by functors.
With this extension of the forcing conditions, the extension functors are \(\Jextfunc = \PJextfunc\) and \(\Mextfunc = \PMextfunc\).
The maximal extension of the forcing conditions is, again for the input, to say that a morphism \(\umor \colon \ufunc(\Btobj) \to \forfunc[\uVBtobj]\) is forced if whenever \(\Btmor \colon \Atobj' \to \Btobj\) is a \Btmor with \(\Atobj'\) an \Atobjalt then \(\umor \ufunc(\Btmor)\) is forced.
With this extension of the forcing conditions, the extension functors are \(\Jextfunc = \extfunc = \Mextfunc\) and, in fact, the categories \(\uFVAtcat\) and \(\uFVBtcat\) are isomorphic via the extension and restriction functors.

\subsection{Change of Underlying Category}

Now we turn to the possibility of changing the underlying category.
We wish to consider two ways to do this.
The first is when we have a functor from one underlying category to another.
This will allow us to consider inclusions of subcategories and reflections and coreflections.
The second is when we have a topological category over another category and wish to transfer \uFVtobjs from the lower category to the higher one.

\subsubsection{Simple Transfers}

The most obvious way to change the underlying category is to have two underlying categories, say \(\Aucat\) and \(\Bucat\), together with a covariant functor \(\subfunc \colon \Aucat \to \Bucat\).
Given that we want to say something about \uVtobjs and \uFVtobjs, we should ensure that this functor ``plays nicely'' with the structures defining them.
Since we are focussing on the underlying categories here, we wish to avoid changing the rest of the structure as much as we can.
Therefore we fix a \tcat, \(\tcat\), with functors to \(\Aucat\) and \(\Bucat\), both of which we shall denote by \(\ufunc\).
The compatibility that we require is that \(\subfunc \ufunc = \ufunc\).

In the absence of forcing conditions, we get some simple results.

\begin{proposition}
For \(\AuVtobj\) a \AuVtobj, \((\subfunc(\forfunc[\uVtobj]), \subfunc \itest_{\uVtobj}, \subfunc \otest_{\uVtobj})\) is a \BuVtobj.
This assignment is functorial and covers \(\subfunc\).
\end{proposition}

\begin{proof}
The only thing to check is the compatibility between the input and output test morphisms.
Let \(\ifn \colon \ufunc(\tobj) \to \subfunc(\forfunc[\uVtobj])\) be an input test morphism and \(\ofn \colon \subfunc(\forfunc[\uVtobj]) \to \tobj'\) be an output test morphism.
By definition, \(\ifn = \subfunc(\ifn')\) for some \(\ifn' \colon \ufunc(\tobj) \to \forfunc[\uVtobj]\) and \(\ofn = \subfunc(\ofn'\) for some \(\ofn' \colon \ufunc(\tobj') \to \forfunc[\uVtobj]\).
Hence \(\ofn \ifn = \subfunc(\ofn' \ifn')\).
As \(\uVtobj\) is a \uVtobj, \(\ofn' \ifn' \in \im \ufunc\) and so \(\ofn \ifn \in \im \subfunc \ufunc = \im \ufunc\).
\end{proof}

We shall denote this functor \(\mVo{\subfunc}\).
Let us write \(\Fun(\Aucat, \Bucat; \ufunc)\) for the category of functors \(\Aucat \to \Bucat\) which intertwine \(\ufunc\), with natural transformations that are the identity on \(\im \ufunc\).

\begin{corollary}
The assignment \(\subfunc \mapsto \mVo{\subfunc}\) defines a functor
\[
  \Fun(\Aucat, \Bucat; \ufunc) \to \Fun(\AuVtcat, \BuVtcat).
\]
In particular, adjoint pairs map to adjoint pairs. \noproof
\end{corollary}

Now let us consider forcing functors.
Here we run into difficulties.
What we want to be able to say is that if \(\AuFVtobj\) is a \AuFVtobj then \(\mVo{\subfunc}(\AuFVtobj)\) is a \BuFVtobj.
To do this we need to be able to assert that, for input morphisms, if \(\uFVtobj\) is a \AuFVtobj and \(\Bumor \colon \ufunc(\tobj) \to \sabs{\mVo{\subfunc}(\uFVtobj)}\) is a forced \Bumor then there is a forced \Aumor \(\Aumor' \colon \ufunc(\tobj) \to \sabs{\uFVtobj}\) such that \(\subfunc(\Aumor') = \Bumor\).
This is not a simple condition.
Firstly, we observe that we can only force \Aumors that lie in the image of \(\subfunc\).
Secondly, there is a potential problem if more than one \AuFVtobj maps to the same \BuVtobj: we need to know that a given forced \Bumor has a forced preimage for each of the preimages of the \BuVtobj.
This is particularly tricky when the \BuVtobj under question is the natural image of a \tobj since there we have additional restrictions on which \Bumors can be forced.
Thirdly, when testing whether a \Bumor is forced, we may need to use trials that have target (or source) outside the image of \(\subfunc\).
Thus it may be that a \Bumor would not be forced if we only considered those trials in the image of \(\subfunc\) but would be forced if we allowed all trials.

At the loss of some generality we can simplify matters a little.
Let \((\mBo{\iforce}, \mBo{\oforce})\) be a forcing condition on \BuVtcat.
We apply \(\subfunc\) to \rcat; this is defined by sending \((\tmor, \uVtmor)\) to \((\tmor, \mVo{\subfunc}(\uVtmor))\).
The temptation at this point would be to define \((\mAo{\iforce}, \mAo{\oforce})\) by composition.
However, this does not help resolve the potential difficulties given above.
Rather, we define
\[
  \mAo{\iforce}(\tobj, \AuVtobj, \m{F}) = \max\{\mBo{\iforce}(\tobj, \mVo{\subfunc}(\AuVtobj), \m{F}') : \m{F}' \cap \im \subfunc = \subfunc(\m{F}) \}.
\]
We define \(\mAo{\oforce}\) similarly.

\begin{lemma}
\label{lem:injforce}
\((\mAo{\iforce}, \mAo{\oforce})\) is a forcing condition.
If \(\subfunc\) is injective on \Aumors out of \(\im \ufunc\) then \(\mAo{\iforce}\) has the property that for a \Aumor, \(\Aumor \colon \ufunc(\tobj) \to \forfunc[\AuVtobj]\), \(\Aumor\) is forced if \(\subfunc(\Aumor)\) is forced.
Similarly, if \(\subfunc\) is injective on \Aumors into \(\im \ufunc\) then \(\mAo{\oforce}\) has the property that for a \Aumor, \(\Aumor \colon \forfunc[\AuVtobj] \to \ufunc(\tobj)\), \(\Aumor\) is forced if \(\subfunc(\Aumor)\) is forced.
\end{lemma}

\begin{proof}
It is straightforward to show that \(\mAo{\iforce}\) and \(\mAo{\oforce}\) are functors.
Thus we consider the ``non\hyp{}stupid'' condition.
This is immediate from the fact that for \tobjs \(\tobj_1\) and \(\tobj_2\), \((1_{\tobj_1}, 1_{\sfunc(\tobj_2)}) \in \im \subfunc\) so if \((\tobj_1, \sfunc(\tobj_2), \m{F})\) is such that \((1,1) \notin \m{F}\) and \((\tobj_1, \sfunc(\tobj_2), \m{F}')\) is such that \(\m{F}' \cap \im \subfunc = \subfunc(\m{F})\) then \((1,1) \notin \m{F}'\) and so \(\mBo{\iforce}(\tobj_1, \sfunc(\tobj_2), \m{F}') = 0\).
Thus \(\mAo{\iforce}(\tobj_1, \sfunc(\tobj_2), \m{F}) = 0\).
The case for \(\mAo{\oforce}\) is similar.

Let \(\Aumor \colon \ufunc(\tobj) \to \forfunc[\AuVtobj]\) be a \Aumor.
We need to show that, under the stated condition, \(\trials[_2]{\subfunc(\Aumor)} \cap \im \subfunc = \trials[_1]{\subfunc(\Aumor)}\).
It is clear that \(\subfunc(\trials[_1]{\Aumor}) \subseteq \trials[_2]{\subfunc(\Aumor)}\).
Let \((\tmor, \AuVtmor)\) be a trial for \(\Aumor\) such that \(\subfunc(\Aumor)\) succeeds at \((\tmor, \mVo{\subfunc}(\AuVtmor))\).
Let \(\AuVtobj'\) be the target of \(\AuVtmor\).
Then \(\sabs{\mVo{\subfunc}(\AuVtmor)} \subfunc(\Aumor) \ufunc(\tmor) \in \itest_{\mVo{\subfunc}(\AuVtobj')}\).
By definition of \(\mVo{\subfunc}(\AuVtobj)\), this means that there is some \(\ifn \in \itest_{\AuVtobj'}\) such that \(\subfunc(\ifn) = \sabs{\mVo{\subfunc}(\AuVtmor)} \subfunc(\Aumor) \ufunc(\tmor)\).
The right\hyp{}hand side of this is equal to \(\subfunc(\sabs{\AuVtmor} \Aumor \ufunc(\tmor))\).
Thus, by the assumption, \(\ifn = \sabs{\AuVtmor} \Aumor \ufunc(\tmor)\) so \(\Aumor\) succeeds at the trial \((\tmor, \AuVtmor)\) as required.

The case for the outputs is similar.
\end{proof}

\begin{corollary}
If \(\subfunc\) is a full functor and is injective on \Aumors in and out of \(\im \ufunc\) then \(\mVo{\subfunc}\) takes \AuFVtobjs to \BuFVtobjs. \noproof
\end{corollary}

For example, if \(\subfunc\) is the inclusion of a full subcategory then we have this condition.
If we don't have this condition then the best strategy is to use one of the compositions
\[
  \AuFVtcat \to \AuVtcat \xrightarrow{\mVo{\subfunc}} \BuVtcat \xrightarrow{\Jforfunc, \Mforfunc} \BuFVtcat.
\] 
If we have one of the two conditions then there is an obvious choice for which of the forcing functors to choose.

\subsubsection{Topological Transfers}

The second way of changing the underlying category that we shall consider concerns the situation when we have a topological category over another category.
We assume the same set\hyp{}up as at the start of the previous part, namely two underlying categories \(\Aucat\) and \(\Bucat\) with a functor \(\subfunc \colon \Aucat \to \Bucat\) and a \tcat, \(\tcat\), with functors to \(\Aucat\) and \(\Bucat\) intertwined by \(\subfunc\).
Using the above, we obtain a functor \(\mVo{\subfunc} \colon \AuVtcat \to \BuVtcat\) which may, or may not, restrict to a functor on some subcategories of \uFVtobjs.

Now we assume, in addition, that \(\Aucat\) is topological over \(\Bucat\).

\begin{lemma}
\(\AuVtcat\) is topological over \(\BuVtcat\).
\end{lemma}

\begin{proof}
Let \(\BuVtobj\) be a \BuVtobj and let \(\BuVtmor_\lambda \colon \BuVtobj \to \mVo{\subfunc}(\AuVtobj_\lambda)\) be a \(\mVo{\subfunc}\)\enhyp{}structured source.

Let us consider the fibre of \(\mVo{\subfunc}\) at \(\BuVtobj\).
If \(\AuVtobj'\) is a \AuVtobj with \(\mVo{\subfunc}(\AuVtobj') = \BuVtobj\) then the input and output test functors of \(\AuVtobj'\) are completely determined by \(\BuVtobj\).
This is because we have, for example, \(\itest_{\BuVtobj} = \subfunc\itest_{\AuVtobj'}\) and \(\subfunc\) is faithful.
Thus the only possible variation in \(\AuVtobj'\) is in regard to \(\sabs{\AuVtobj'}\).
Here we must have \(\subfunc(\sabs{\AuVtobj'}) = \sabs{\BuVtobj}\).

Thus we are looking for a lift of \(\sabs{\BuVtobj}\) with certain properties.
Firstly, the \Bumors \(\ifn \colon \ufunc(\tobj) \to \sabs{\BuVtobj}\) and \(\ofn \colon \sabs{\BuVtobj} \to \ufunc(\tobj)\) for \(\ifn \in \itest_{\BuVtobj}(\tobj)\) and \(\ofn \in \otest_{\BuVtobj}(\tobj)\) must lift to \Aumors.
Secondly, the \BuVtmors \(\BuVtmor_\lambda \colon \BuVtobj \to \mVo{\subfunc}(\AuVtobj_\lambda)\) must lift to \AuVtmors.
However it is straightforward to see that if the first condition holds then the second is equivalent to lifting the \Bumors \(\sabs{\BuVtmor_\lambda} \colon \sabs{\BuVtobj} \to \subfunc(\sabs{\AuVtobj_\lambda})\) to \Aumors.

If we ignore the input test morphisms for the moment, then we have a family of \Bumors out of \(\sabs{\BuVtobj}\).
Since \(\Aucat\) is topological over \(\Bucat\) there is a unique inital lift of this family, say \(\Auobj\).
Now we see that this also lifts the input test morphisms since the composition of \(\ifn\) with either \(\ofn\) or with \(\sabs{\BuVtmor_\lambda}\) lifts to an \Aumor.
Hence \((\Auobj, \subfunc^{-1}\itest_{\BuVtobj}, \subfunc^{-1}\otest_{\BuVtobj})\) is the desired lift.
\end{proof}

\begin{corollary}
The functor \(\mVo{\subfunc} \colon \AuVtcat \to \BuVtcat\) has both a left and a right adjoint. \noproof
\end{corollary}

\begin{defn}
Let us write \(\finfunc\) for the left adjoint and \(\coafunc\) for the right adjoint.
\end{defn}

The notation is to suggest ``finest'' and ``coarsest''.

Under favourable conditions these adjoints allow us to transfer forcing conditions neatly from \(\AuVtcat\) to \(\BuVtcat\).
The idea is that a \Bumor \(\Bumor \colon \ufunc(\tobj) \to \sabs{\BuVtobj}\) should be forced if it is forced for at least one of the preimages of \(\BuVtobj\).
But if it is forced for at least one, it will be forced for the maximal one.
Thus we need only test \(\coafunc(\BuVtobj)\).
Similarly, for input test morphisms we need only look at \(\finfunc(\BuVtobj)\).
The caveat is, as always, that we need to check the ``non\hyp{}stupid'' condition.

\begin{proposition}
Suppose that \(\finfunc\) and \(\coafunc\) satisfy \(\finfunc \sfunc = \sfunc\) and \(\coafunc \sfunc = \sfunc\).
Let \((\mAo{\iforce}, \mAo{\oforce})\) be a forcing condition on \AuVtobjs.
Define
\begin{align*}
  \mBo{\iforce}(\tobj, \BuVtobj, \m{F}) &\coloneqq \max\{\mAo{\iforce}(\tobj, \finfunc(\BuVtobj), \m{F}') : \m{F}' \cap \im \finfunc = \finfunc (\m{F})\}, \\
  \mBo{\oforce}(\tobj, \BuVtobj, \m{F}) &\coloneqq \max\{\mAo{\oforce}(\coafunc(\BuVtobj), \tobj, \m{F}') : \m{F}' \cap \im \coafunc = \coafunc (\m{F})\}.
\end{align*}
Then \((\mBo{\iforce}, \mBo{\oforce})\) is a forcing condition.
\end{proposition}

\begin{proof}
Functorality is immediate.
Let us check the ``non\hyp{}stupid'' condition.
Observe that the assumption on \(\finfunc\) ensures that
\[
  \mBo{\iforce}(\tobj, \sfunc(\tobj'), \m{F}) = \max\{\mAo{\iforce}(\tobj, \sfunc(\tobj'), \m{F}') : \m{F}' \cap \im \finfunc = \finfunc (\m{F})\}.
\]
Moreover, \((1,1) \in \im\finfunc\) so \((1,1) \in \m{F}'\) if and only if \((1, 1) \in \m{F}\).
Hence if \((1,1) \notin \m{F}\), \(\mBo{\iforce}(\tobj, \sfunc(\tobj'), \m{F}) = 0\).
The case of \(\mBo{\oforce}\) is similar.
\end{proof}

\begin{proposition}
\(\finfunc\) and \(\coafunc\) take \BuFVtobjs to \AuFVtobjs.
\end{proposition}

\begin{proof}
This follows from the argument in lemma~\ref{lem:injforce} since both \(\finfunc\) and \(\coafunc\) are full and faithful.
\end{proof}

\section{Functors in the Wild}
\label{sec:funwild}

We can now define functors between the various categories of smooth spaces.
We define the following notation for the categories of test spaces (morphisms are always \cimaps):
\begin{itemize}
\item \(\tccat\): objects are convex subsets of Euclidean spaces.
\item \(\tdcat\): objects are open subsets of Euclidean spaces.
\item \(\tkcat\): objects are Euclidean spaces.
\item \(\tfcat\): single object, \R.
\end{itemize}

We define the following notation for the forcing conditions:
\begin{itemize}
\item \(\emptyset\): no condition.
\item \(\Term\): the terminal condition.
\item \(\Sheaf\): the sheaf condition.
\item \(\Detm\): the determination condition.
\item \(\Sat\): the saturation condition.
\end{itemize}
We specify a forcing condition as (input, output) where both are lists of conditions which are to be ``anded'' together.
The underlined categories correspond to the ones in section~\ref{sec:smthwild}, the others are intermediate ones put in to enable us to define the functors.
We denote \ycat by \(\ycat\) and the functors between \(\ycat\) and \(\xcat\) by the obvious notations \(\indfunc, \disfunc \colon \xcat \to \ycat\) and \(\forfunc \colon \ycat \to \xcat\).

\begin{centre}
\begin{tikzpicture}[node distance=1.8cm, auto,>=latex', thick]
\path[use as bounding box] (-2.5cm,.5cm) rectangle (8.5cm,-10cm);
\node (chen) {\(\big(\underline{\xcat, \tccat, (\Sheaf \Term,\Sat)}\big)\)};
\node[moreright] (chenr) at (chen.east) {\(\big(\underline{\xcat, \tccat, (\Sheaf \Term,\Sat)}\big)\)};
\node[moreleft] (chenl) at (chen.west) {\(\big(\underline{\xcat, \tccat, (\Sheaf \Term,\Sat)}\big)\)};
\node[below of=chen] (chso) {\(\big(\xcat, \tdcat \cap \tccat, (\Sheaf \Term,\Sat)\big)\)};
\node[moreright] (chsor) at (chso.east) {\(\big(\xcat, \tdcat \cap \tccat, (\Sheaf \Term,\Sat)\big)\)};
\node[moreleft] (chsol) at (chso.west) {\(\big(\xcat, \tdcat \cap \tccat, (\Sheaf \Term,\Sat)\big)\)};
\node[below of=chso] (sour) {\(\big(\underline{\xcat, \tdcat, (\Sheaf \Term,\Sat)}\big)\)};
\node[moreright] (sourr) at (sour.east) {\(\big(\underline{\xcat, \tdcat, (\Sheaf \Term,\Sat)}\big)\)};
\node[moreleft] (sourl) at (sour.west) {\(\big(\underline{\xcat, \tdcat, (\Sheaf \Term,\Sat)}\big)\)};
\node[node distance=6cm, right of=sour] (siko) {\(\big(\underline{\ycat, \tkcat, (\Sat, \Sheaf  \Detm \Term)}\big)\)};
\node[moreright] (sikor) at (siko.east) {\(\big(\underline{\ycat, \tkcat, (\Sat, \Sheaf  \Detm \Term)}\big)\)};
\node[moreleft] (sikol) at (siko.west) {\(\big(\underline{\ycat, \tkcat, (\Sat, \Sheaf  \Detm \Term)}\big)\)};
\node[below of=sour] (sofr) {\(\big(\xcat, \tdcat, (\Sat,\Sat)\big)\)};
\node[moreright] (sofrr) at (sofr.east) {\(\big(\xcat, \tdcat, (\Sat,\Sat)\big)\)};
\node[moreleft] (sofrl) at (sofr.west) {\(\big(\xcat, \tdcat, (\Sat,\Sat)\big)\)};
\node[below of=siko] (sism) {\(\big(\ycat, \tkcat, (\Sat, \Sat)\big)\)};
\node[moreright] (sismr) at (sism.east) {\(\big(\ycat, \tkcat, (\Sat, \Sat)\big)\)};
\node[moreleft] (sisml) at (sism.west) {\(\big(\ycat, \tkcat, (\Sat, \Sat)\big)\)};
\node[below of=sofr] (frol) {\(\big(\underline{\xcat, \tfcat, (\Sat, \Sat)}\big)\)};
\node[moreright] (frolr) at (frol.east) {\(\big(\underline{\xcat, \tfcat, (\Sat, \Sat)}\big)\)};
\node[moreleft] (froll) at (frol.west) {\(\big(\underline{\xcat, \tfcat, (\Sat, \Sat)}\big)\)};
\node[moreabove] (frolu) at (frol.north) {\(\big(\underline{\xcat, \tfcat, (\Sat, \Sat)}\big)\)};
\node[morebelow] (frold) at (frol.south) {\(\big(\underline{\xcat, \tfcat, (\Sat, \Sat)}\big)\)};
\node[below of=sism] (smit) {\(\big(\underline{\ycat, \tfcat, (\Sat, \Sat)}\big)\)};
\node[moreright] (smitr) at (smit.east) {\(\big(\underline{\ycat, \tfcat, (\Sat, \Sat)}\big)\)};
\node[moreleft] (smitl) at (smit.west) {\(\big(\underline{\ycat, \tfcat, (\Sat, \Sat)}\big)\)};
\node[moreabove] (smitu) at (smit.north) {\(\big(\underline{\ycat, \tfcat, (\Sat, \Sat)}\big)\)};
\node[morebelow] (smitd) at (smit.south) {\(\big(\underline{\ycat, \tfcat, (\Sat, \Sat)}\big)\)};
\node[below of=frol] (frdi) {\(\big(\xcat, \tfcat, (\emptyset, \emptyset)\big)\)};
\node[moreright] (frdir) at (frdi.east) {\(\big(\xcat, \tfcat, (\emptyset, \emptyset)\big)\)};
\node[moreleft] (frdil) at (frdi.west) {\(\big(\xcat, \tfcat, (\emptyset, \emptyset)\big)\)};
\node[below=8.5pt, style=transparent] (frdiu) at (frdi.north) {\(\big(\xcat, \tfcat, (\emptyset, \emptyset)\big)\)};
\node[above=8.5pt, style=transparent] (frdid) at (frdi.south) {\(\big(\xcat, \tfcat, (\emptyset, \emptyset)\big)\)};
\node[below of=smit] (smdi) {\(\big(\ycat, \tfcat, (\emptyset, \emptyset)\big)\)};
\node[moreright] (smdir) at (smdi.east) {\(\big(\ycat, \tfcat, (\emptyset, \emptyset)\big)\)};
\node[moreleft] (smdil) at (smdi.west) {\(\big(\ycat, \tfcat, (\emptyset, \emptyset)\big)\)};
\node[below=8.5pt, style=transparent] (smdiu) at (smdi.north) {\(\big(\ycat, \tfcat, (\emptyset, \emptyset)\big)\)};
\node[above=8.5pt, style=transparent] (smdid) at (smdi.south) {\(\big(\ycat, \tfcat, (\emptyset, \emptyset)\big)\)};
\path[->] (chsol) edge node[diaglabel,swap] {\(\Jextfunc\)} (chenl)
 (chsor) edge node[diaglabel] {\(\Mextfunc\)} (chenr)
 (chen) edge node[diaglabel,fill=white,auto=false] {\(\resfunc\)} (chso)
 (chsol) edge node[diaglabel] {\(\Jextfunc\)} (sourl)
 (chsor) edge node[diaglabel,swap] {\(\Mextfunc\)} (sourr)
 (sour) edge node[diaglabel,fill=white,auto=false] {\(\resfunc\)} (chso)
 (froll) edge node[diaglabel,swap] {\(\Jextfunc\)} (sofrl)
 (frolr) edge node[diaglabel] {\(\Mextfunc\)} (sofrr)
 (sofr) edge node[diaglabel,fill=white,auto=false] {\(\resfunc\)} (frol)
 (smitl) edge node[diaglabel,swap] {\(\Jextfunc\)} (sisml)
 (smitr) edge node[diaglabel] {\(\Mextfunc\)} (sismr)
 (sism) edge node[diaglabel,fill=white,auto=false] {\(\resfunc\)} (smit)
 (sourl) edge node[diaglabel] {\(\Mforfunc\)} (sofrl)
 (sourr) edge node[diaglabel,swap] {\(\Jforfunc\)} (sofrr)
 (sofr) edge node[diaglabel,fill=white,auto=false] {\(\incfunc\)} (sour)
 (frdil) edge node[diaglabel,swap] {\(\Mforfunc\)} (froll)
 (frdir) edge node[diaglabel] {\(\Jforfunc\)} (frolr)
 (frol) edge node[diaglabel,fill=white,auto=false] {\(\incfunc\)} (frdi)
 (sikol) edge node[diaglabel] {\(\Mforfunc\)} (sisml)
 (sikor) edge node[diaglabel,swap] {\(\Jforfunc\)} (sismr)
 (sism) edge node[diaglabel,fill=white,auto=false] {\(\incfunc\)} (siko)
 (smdil) edge node[diaglabel,swap] {\(\Mforfunc\)} (smitl)
 (smdir) edge node[diaglabel] {\(\Jforfunc\)} (smitr)
 (smit) edge node[diaglabel,fill=white,auto=false] {\(\incfunc\)} (smdi)
 (frolu) edge node[diaglabel,swap] {\(\coafunc\)} (smitu)
 (frold) edge node[diaglabel] {\(\finfunc\)} (smitd)
 (frdiu) edge node[diaglabel,swap] {\(\mVo{\disfunc}\)} (smdiu)
 (frdid) edge node[diaglabel] {\(\mVo{\indfunc}\)} (smdid)
 (smdi) edge node[diaglabel,fill=white,auto=false] {\(\mVo{\forfunc}\)} (frdi)
;
\end{tikzpicture}
\end{centre}

There are various adjunctions and isomorphisms that can be read off this diagram.
Firstly we note that, due to Boman's theorem in \cite{jb3}, under the saturation condition then the extension functors always agree.
Thus restriction and extension define isomorphisms of categories
\[
  \big(\xcat, \tfcat, (\Sat, \Sat)\big) \cong \big(\xcat, \tdcat,(\Sat,\Sat)\big)
\]
and
\[
  \big(\ycat, \tfcat, (\Sat, \Sat)\big) \cong \big(\ycat, \tkcat, (\Sat,\Sat)\big).
\]
Furthermore, since any open subset of Euclidean space is locally convex, the sheaf condition is sufficient to ensure that restriction and extension define an isomorphism of categories
\[
  \big(\xcat, \tdcat \cap \tccat, (\Sheaf \Term,\Sat)\big) \cong      \big(\xcat, \tdcat, (\Sheaf \Term,\Sat)\big).
\]

The extension functors are always adjoint to the restriction functors and thus the functor
\[
  \big(\xcat, \tccat, (\Sheaf \Term,\Sat)\big) \to \big(\xcat, \tdcat, (\Sheaf \Term,\Sat)\big)
\]
has both a left and right adjoint.

Saturation is sufficient to ensure that proposition~\ref{prop:foradj} holds, at least for one direction, and thus the functors (effectively given by inclusion)
\begin{align*}
  \big(\xcat, \tfcat, (\Sat, \Sat)\big) &\to \big(\xcat, \tdcat, (\Sheaf\Term, \Sat)\big) \\
  \big(\ycat, \tfcat, (\Sat, \Sat)\big) &\to \big(\ycat, \tkcat, (\Sat,\Sheaf\Detm\Term)\big)
\end{align*}
have adjoints.
In the first case, the adjoint is a left adjoint given by \(\Mforfunc\).
In the second case, the adjoint is a right adjoint given by \(\Jforfunc\).

As the functors are constructed in a fairly abstract manner, let us now give brief descriptions of how they actually work.
We shall use the usual descriptions of the spaces which, usually, means only considering one family of test morphisms.

\paragraph{Chen and Souriau.}

Let \((X, \m{P})\) be a Chen space.
For an open set \(U\) of some Euclidean space, let \((U, \m{P}_U)\) denote its natural structure as a Chen space; i.e.\ a map \(C \to U\) is a plot if and only if it is a \cimap.
Let \(\m{D}\) be the family of maps \(U \to X\) which underlie morphisms of Chen spaces \((U, \m{P}_U) \to (X, \m{P})\).
Then \((X, \m{D})\) is a Souriau space.

Conversely, let \((X, \m{D})\) be a Souriau space.
For a convex set \(C\) of some Euclidean space, let \((C, \m{D}_C)\) denote its natural structure as a Souriau space; i.e.\ a map \(U \to C\) is a plot if and only if it is a \cimap.
Let \(\m{P}^+\) be the family of maps \(C \to X\) whichy underlie morphisms of Souriau spaces \((C, \m{D}_C) \to (X, \m{D})\).
Then \((X, \m{P}^+)\) is a Chen space.

Define a second family \(\m{P}^-\) on \(X\) by taking all the maps \(C \to X\) which locally factor through the maps in \(\m{D}\); as above, \(C\) runs over the family of convex subsets of Euclidean spaces.
Then \((X, \m{P}^-)\) is a Chen space.

\paragraph{Souriau and Fr\"olicher.}

Let \((X, \m{D})\) be a Souriau space.
Let \(\m{F}\) be the set of morphisms of Sourian spaces from \((X, \m{D})\) to \R, viewed as having its standard diffeology.
Let \(\m{C}\) be the family of curves determined by \(\m{F}\).
That is to say, \(\m{C}\) is the family of curves \(c \colon \R \to X\) with the property that \(f c \in \Ci(\R,\R)\) for all \(f \in \m{F}\).
Then \((X, \m{C}, \m{F})\) is a Fr\"olicher space.

Let \((X, \m{C}, \m{F})\) be a Fr\"olicher space.
Define a family \(\m{D}\) on \(X\) by taking all maps \(U \to X\) which underlie maps of Fr\"olicher spaces with \(U\) an open subset of a Euclidean space equipped with its obvious Fr\"olicher space structure.
Then \((X, \m{D})\) is a Souriau space.

\paragraph{Smith and Fr\"olicher.}

The case of Smith spaces and Fr\"olicher spaces is an interesting one because the definitions are so similar.
One way to phrase the difference is to say that in a Smith space the smooth functions are a subset of the continuous functions whilst in a Fr\"olicher space the continuous functions are a superset of the smooth functions.

Let \((X, \m{T}, \m{F})\) be a Smith space.
Let \(\m{C}_X\) be the set of maps \(c \colon \R \to X\) for which \(f c \in \Ci(\R, \R)\) for all \(f \in \m{F}\).
Let \(\m{F}_X\) be the set of maps \(f \colon X \to \R\) for which \(f c \in \Ci(\R,\R)\) for all \(c \in \m{C}_X\).
Then \((X, \m{C}, \m{F})\) is a Fr\"olicher space.

For the reverse direction, let \((X, \m{C}, \m{F})\) be a Fr\"olicher space.
To make this a Smith space we need to define a topology on \(X\).
There are two obvious choices: the \emph{curvaceous} topology on \(X\) and the \emph{functional} topology.
The first is the finest topology for which all the smooth curves are continuous.
The second is the coarsest topology for which all the smooth functionals are continuous.
Let us denote them by \(\m{T}_c\) and \(\m{T}_f\) respectively.
It is immediate that \(\m{T}_c\) is finer than \(\m{T}_f\).
Then \((X, \m{T}_c, \m{F})\) and \((X, \m{T}_f, \m{F})\) are Smith spaces.

\paragraph{Sikorski and Smith.}

Let \((X, \m{T}, \m{F})\) be a Sikorski space.
For \(U \subseteq \R^n\) an open subset, let \(\m{F}(U)\) be the family of continuous maps \(\phi \colon U \to X\) for which \(f \phi \in \Ci(U, \R)\) for all \(f \in \m{F}\).
Let \(\overline{\m{F}}\) be the family of continuous functions \(f \colon X \to \R\) for which \(f \phi \in \Ci(U, \R)\) for all \(\phi \in \m{F}(U)\) and for all open \(U \subseteq \R^n\).
Then \((X, \m{T}, \overline{\m{F}})\) is a Smith space.

Conversely, if \((X, \m{T}, \m{F})\) is a Smith space then it is automatically a Sikorski space.

\section{The Differences}
\label{sec:diff}

We saw in the previous section that the category of Fr\"olicher spaces is (isomorphic to) a full subcategory of that of Souriau spaces, likewise the category of Sourian spaces in that of Chen spaces, Fr\"olicher spaces in Smith spaces, and Smith spaces in Sikorski spaces.
An obvious question is whether any of these embeddings is dense.

As an aid to answering that, note that it is easy to characterise Souriau spaces and Chen spaces according to whether or not they have the same underlying Fr\"olicher space: simply examine the set of morphisms to \R.
In both categories, we regard \R as having its ``standard'' structure, namely the smallest which contains the identity map.

For Sikorski spaces and Smith spaces, the characterisation is slightly more problematical due to the topology.
One cannot simply examine the set of morphisms from \R with its standard Sikorski or Smith structure as these might not be all the smooth curves; the problem being that a smooth curve need not be continuous.
Thus one has to simply compute the family of smooth curves irrespective of the topology and this cannot be described in terms of morphism sets.
We shall comment further on topology later.

\medskip

Let us now construct non\hyp{}isomorphic objects in each of the categories which have the same underlying structure in the ``next category down''.
The most interesting of these examples is probably the one showing that Chen spaces and Souriau spaces are not the same.

We shall use the following notation.
\begin{itemize}
\item \(\fcat\) for \fcat.
\item \(\ccat\) for \ccat.
\item \(\dcat\) for \dcat.
\item \(\kcat\) for \kcat.
\item \(\scat\) for \scat.
\end{itemize}
We shall label the functors between them by the \emph{target} category.
For all but two pairs there is a single sensible functor in each direction.
There are two functors from \dcat to \ccat and from \fcat to \scat.
The latter differ only in a very minor way and can be treated together so we shall not specify which is which.
Thus
\begin{itemize}
\item \(\FrSo\colon \fcat \to \dcat\)
\item \(\SoFr\colon \dcat \to \fcat\)
\item \(\SoCht\colon \dcat \to \ccat\) (corresponding to \(\Jextfunc\))
\item \(\SoChb\colon \dcat \to \ccat\) (corresponding to \(\Mextfunc\))
\item \(\ChSo\colon \ccat \to \dcat\)
\item \(\FrSm\colon \fcat \to \scat\)
\item \(\SmFr\colon \scat \to \fcat\)
\item \(\SmSi\colon \scat \to \kcat\)
\item \(\SiSm\colon \kcat \to \scat\)
\end{itemize}

\paragraph{Smith spaces.}
For Smith spaces, consider \R with its usual topology and the set of all continuous functions versus \R with the discrete topology and the set of all continuous functions (i.e.\ continuous for the discrete topology, whence all functions).
The smooth curves for the latter are the constant curves.
Let us show that the same holds for the former.
Firstly, observe that if \(c \colon \R \to \R\) is a map such that \(f c \in \Ci(\R,\R)\) for all continuous functions \(f \colon \R \to \R\) then \(c\) is a \cimap as the identity is continuous.
Suppose that \(c \colon \R \to \R\) is a \cimap which is not constant.
Then \(c'(t) \ne 0\) for some \(t \in \R\) and so there is an interval around \(t\) on which the restriction of \(c\) is a diffeomorphism.
Thus if \(f \colon \R \to \R\) is a map for which \(f c \in \Ci(\R,\R)\) then \(f\) must be a \cimap in a neighbourhood of \(c(t)\).
As there is a continuous map which is not a \cimap in this neighbourhood, \(c\) cannot be a smooth curve in the Fr\"olicher structure underlying this Smith space.
Thus all the smooth curves in this space are constant.
Hence these two different Smith spaces define the same underlying Fr\"olicher space.

\paragraph{Sikorski spaces.}
In the above, it sufficed that there was a functional that was non\hyp{}smooth at each point in \R.
From this we see that if \(\m{F}\) is a family of functions \(\R \to \R\) which contains the identity, is translation\hyp{}invariant, and contains a non\hyp{}smooth function, then the only plots \(\phi \colon U \to \R\) are the constant ones.
The corresponding Smith space of such a family thus has family of functions all the continuous functions.
One can find a family of functions \(\m{F}\) satisfying these conditions and the conditions for a Sikorski structure on \R which is not the set of continuous functions.
For example, one of the families of piecewise\hyp{}smooth functions will do (there are various possibilities, any will do here).
Thus the inclusion of Smith spaces in Sikorski spaces is not dense.

\paragraph{Souriau spaces.}
As Souriau spaces, \(\R^2\) with its standard diffeology and with its ``wire'' diffeology (generated by the smooth curves) are not isomorphic.
Nonetheless, by Boman's result, \cite{jb3}, they have the same smooth functionals and thus the same underlying Fr\"olicher structure.

\paragraph{Chen spaces.}
Let us now compare Chen spaces and Souriau spaces.
We can show that they are different by exhibiting a Souriau space whose images under the two extension functors are different.

This Souriau space is \([0,1]\) with its usual diffeology.
The Chen space \(\SoCht([0,1])\) is easily seen to be the standard Chen space structure on \([0,1]\).
In particular, it contains the identity on \([0,1]\).
On the other hand, every plot in \(\SoChb([0,1])\) factors through a \cimap \(U \to [0,1]\) from an open subset of some Euclidean space into \([0,1]\).
The identity on \([0,1]\) does not have this property: any factorisation of the inclusion \([0,1] \to \R\) via some open set \(U\) must extend outside \([0,1]\).
Hence \(\SoCht([0,1]) \ne \SoChb([0,1])\).

One can extend this example to see that the main difference between Chen spaces and Souriau spaces is the ability to ``approach boundary points at speed''.
In the Chen realm, one can approach a point at speed and stop.
In the Souriau realm, one must always be able to go a little further.
Now suppose that one wishes to declare that a certain point cannot be approached from certain directions.
One therefore wishes to consider one\hyp{}sided derivatives at those points.
In the Chen realm, this presents no difficulties: one approaches at a steady speed along those directions that one is allowed to approach along.
In the Souriau realm, this is more problematical.
One has to approach along an allowed direction and then ``bounce back'', so ones speed has to approach zero as one nears the point of interest.

However, due to the result in \cite[24.5]{akpm} this difference is illusory.
One can define one\hyp{}sided derivatives in the Souriau realm using functions that are required to ``bounce back''.

\section{More on Adjunctions}
\label{sec:adjoint}

In section~\ref{sec:funwild} we had several adjunctions.
In this section we shall consider whether any other adjunctions are possible.
Our conclusion will be that there aren't.
One thing worth highlighting is that the examples in this section are not overly complicated.
All have very simple underlying sets, often \R, and reasonably simple structure.

\paragraph{Fr\"olicher and Smith.}
One surprise is that that there is no an adjunction pairing between the categories of Fr\"olicher spaces and Smith spaces.
One would expect an adjunction \(\SmFr \adjoint \FrSm\) since if \((X, \m{T}, \m{F})\) is a Smith space and
 \((X, \m{T}', \m{F}') = \SmFr\FrSm(X,\m{T},\m{F})\) then \(\m{F} \subseteq \m{F}'\) which suggests that the identity on \(X\) should lift to a morphism
 \((X, \m{T}', \m{F}') \to (X, \m{T}, \m{F})\).
The problem is that this need not be continuous.
One might suppose that one could fix this by altering the topology used in defining the Smith space from a Fr\"olicher space.
For a Fr\"olicher space \((X, \m{C}, \m{F})\) one would need to be able to choose a topology \(\m{T}\) on \(X\) such that \((X, \m{T}, \m{F})\) was a Smith space and if \((X, \m{T}', \m{F}')\) was another Smith space structure on \(X\) with
 \(\m{F}' \subseteq \m{F}\)
then the identity on \(X\) is continuous as a morphism
 \((X, \m{T}) \to (X, \m{T}')\).

\medskip

Let us show by example that this is not possible.
Our example will also show that \(\SmFr\) cannot be a right adjoint as it does not preserve limits.
Note that for any topological space \((X, \m{T})\) the triple \((X, \m{T}, \m{F})\) is a Smith space where \(\m{F}\) is the family of all continuous functionals from \(X\) to \R.

Let \(\m{T}_+\) be the topology on \R with basis
 \(\{[a, \infty), a \in \R\}\)
and \(\m{T}_-\) with basis
 \(\{(-\infty, a], a \in \R\}\).
Let \(\m{T}_{\R}\) be the standard topology on \R.
For either of these non\hyp{}standard topologies, a continuous map
 \(f \colon (\R, \m{T}_\pm) \to (\R, \m{T}_{\R})\)
is constant.
Let \(\m{F}_c\) denote the family of constant functionals on \R.
As remarked above, each of
 \((\R, \m{T}_\pm, \m{F}_c)\)
is a Smith space.
The identity on \R underlies morphisms of Smith spaces
 \((\R, \m{T}_\pm, \m{F}_c) \to (\R, \m{I}_{\R}, \m{F}_c)\)
where \(\m{I}_{\R}\) is the indiscrete topology on \R.
Let us consider the corresponding pullback.
It is clear that the underlying set of the pullback is \R and that the underlying set morphisms are the identity.
For these to be continuous, the topology on the pullback must be the discrete topology.
The family of functionals is then forced, by either the locality condition or the saturation condition, to be \emph{all} functionals.
That is to say, the pullback is
 \((\R, \m{D}_{\R}, \Set(\R,\R))\).

The corresponding Fr\"olicher spaces of
 \((\R, \m{T}_\pm, \m{F}_c)\)
and \((\R, \m{I}, \m{F}_c)\)
are all
 \((\R, \Set(\R,\R), \m{F}_c)\).
The pullback in \(\fcat\) is thus
 \((\R, \Set(\R,\R), \m{F}_c)\).
However, the Fr\"olicher space of \((\R, \m{D}_{\R}, \Set(\R,\R))\) is
 \((\R, \m{C}_c, \Set(\R,\R))\)
where \(\m{C}_c\) is the set of constant curves.
Hence \(\SmFr\) cannot be a right adjoint.

To see that this shows that we cannot choose the topology on a Fr\"olicher space in a manner suitable for Smith spaces, note that by this example the topology on \((X, \Set(\R,\R), \m{F}_c)\) would have to be finer than both \(\m{T}_+\) and \(\m{T}_-\).
The only such topology is the discrete topology, but by the completion property of Smith spaces the resulting functionals would be \emph{all} functionals.

\medskip

Let us now show that \(\FrSm\) cannot be a left adjoint by examining its behaviour on colimits.
Equip \R and \([0,1]\) with their standard structure as Smith spaces and as Fr\"olicher spaces (note that these correspond under \(\FrSm\)).
Define two morphisms \(a, b \colon \R \to [0,1]\) by \(a(t) = 1\) and \(b(t) = 1/(1 + t^2)\) (the exact form of \(b\) does not matter, just its image).
Consider the coequalisers in both categories.

The underlying set of the coequalisers is \(\{0,1\}\) in each case.
For each, a functional \(\{0,1\} \to \R\) is smooth if and only if it pulls back to a smooth functional on \([0,1]\).
As a smooth functional on \([0,1]\) is continuous, we see that the smooth functionals on \(\{0,1\}\) are the constant functionals.
The Fr\"olicher space structure is thus \((\{0,1\}, \Set(\R, \{0,1\}), \m{F}_c)\).
For the Smith space structure, we must also work out the topology.
As constant functions are always continuous, there are no constraints on the topology from those whence we see that the topology must be the quotient topology.
This is \(\{\emptyset, \{1\}, \{0,1\}\}\).
But the topology on \(\FrSm(\{0,1\}, \Set(\R, \{0,1\}), \m{F}_c)\) is the indiscrete topology no matter which method of choosing the topology is used.

\medskip

We can modify the above example to show that \(\SmFr\) also cannot be a left adjoint.
In the above, change the Smith structure on \([0,1]\) to be that with all continuous functions.
The corresponding Fr\"olicher space is \(([0,1], \m{C}_c, \Set([0,1], \R))\).
The rest of the structure stays the same as above.
Let us calculate the coequalisers in both categories.
As before, both have underlying sets \(\{0,1\}\).
The functionals on each are those functions \(\{0,1\} \to \R\) which pull back to ``smooth'' functionals on \([0,1]\).
The Smith structure thus produces all continuous functionals \(\{0,1\} \to \R\).
As the topology is the same as before, namely \(\{\emptyset, \{1\}, \{0,1\}\}\), this is simply the constant functionals.
The coequaliser in the category of Smith spaces is thus \((\{0,1\}, \m{T}, \m{F}_c)\) where \(\m{T}\) is the above topology.
The underlying Fr\"olicher space of this is \((\{0,1\}, \Set(\R, \{0,1\}), \m{F}_c)\).
On the other hand, the coequaliser in the category of Fr\"olicher spaces is clearly \((\{0,1\}, \m{C}_c, \Set(\{0,1\}, \R))\).
Hence \(\SmFr\) does not preserve colimits.

\medskip

Finally, for these functors, let us show that \(\FrSm\) cannot be a right adjoint by considering its action on limits.
For \(t \in \R\) define a Fr\"olicher space \((\R, \m{C}_t, \m{F}_t)\) by taking \R and putting a kink in at \(t\).
That is, the functionals are those maps \(\R \to \R\) which are continuous, are smooth on \(\R \ssetminus \{t\}\), and all left and right derivatives exist at \(t\) (but are not necessarily equal).
The corresponding curves are characterised by the property that they can only pass through \(t\) \emph{infinitely slowly}.
As an example, \((\R, \m{C}_0, \m{F}_0)\) can be viewed as the union of the positive \(x\)\enhyp{} and \(y\)\enhyp{}axes in \(\R^2\).
There is an obvious morphism from each of these spaces to \R with its usual Fr\"olicher structure.
Let us consider the limit of this family.
The underlying set is again \R and the curves in the Fr\"olicher space structure are those maps \(c \colon \R \to \R\) which are in each of the \(\m{C}_t\).
Such a curve must be smooth and everywhere infinitely slow, whence constant.
Hence the limiting Fr\"olicher space is \((\R, \m{C}_c, \Set(\R,\R))\).

Now let us transport this to the category of Smith spaces via \(\FrSm\).
The underlying set of the limit is again \R.
The topology is the standard topology, which immediately implies that \(\FrSm\) does not preserve limits, but let us examine the functionals as well to see that even if one could fix the topology then \(\FrSm\) would not preserve limits.
As we have seen before, the plots of the corresponding Smith space are constant whence the functionals are all \emph{continuous} functionals.
As the topology is the standard one, we see that the limiting Smith space is \((\R, \m{T}_\R, C(\R,\R))\).
This is not the result of applying \(\FrSm\) to \((\R, \m{C}_c, \Set(\R,\R))\).

\paragraph{Smith and Sikorski.}

From section~\ref{sec:funwild} we know that \(\SmSi \adjoint \SiSm\).
Let us show that \(\SmSi\) is not a right adjoint and \(\SiSm\) is not a left adjoint.

To show that \(\SiSm\) is not a left adjoint, let us consider its behaviour on colimits.
For \(k \in \N\), define a family of functionals \(\m{F}_k\) to be the set of functions \(f \colon \R \to \R\) which are continuous, are smooth on \(\R \ssetminus \{0\}\), all left and right derivatives exist at \(0\), and the left and right derivatives agree up to (at least) order \(k\).
It is straightforward to show that each of these families satisfies the axioms for a Sikorski space (with the standard topology on \R).
The colimit of this family obviously has functionals the set of smooth functions \(f \colon \R \to \R\).

Let us consider the underlying Smith space of
 \((\R, \m{T}_{\R}, \m{F}_k)\)
(\(\m{T}_{\R}\) being the standard topology on \R).
Let \(c \colon \R \to \R\) be a curve such that \(f c \in \Ci(\R,\R)\) for all \(f \in \m{F}_k\).
It is clear that \(c\) is smooth since the identity map is in \(\m{F}_k\).
Let us suppose, for a contradiction, that there is some \(t \in \R\) such that \(c(t) = 0\), that \(c\) passes through \(0\) at \(t\), and that \(c\) is not flat at \(t\).
Let us assume for simplicity that \(t = 0\), the general case follows by precomposition with a translation.
These conditions imply that the first non\hyp{}vanishing derivative of \(c\) at \(0\) is odd.
Thus there is some odd \(l\) and \(C \ne 0\) such that
 \(c(s) = C s^l + O(s^{l+1})\).
Let \(f \in \m{F}_k\) be the function \(f(x) = \sabs{x} x^k\) and consider \(f c\).
Expanding out, we see that
\begin{align*}
 (f c)(s) &%
 = \sabs{C s^l + O(s^{l+1})}(C s^l + O(s^{l+1}))^k \\
 &%
 = \sabs{C} \sabs{s}^l \sabs{1 + O(s)} C^k s^{l k}(1 + O(s)) \\
 &%
 = \sabs{C} C^k \sabs{s} s^{l (k+1) - 1}(1 + O(s)) &%
 &
 \text{as \(l\) is odd}.
\end{align*}
From this we deduce that the \(l(k+1)\)th derivative of \(f c\) does not exist at \(0\).
Hence if \(c \colon \R \to \R\) is such that \(f c \in \Ci(\R,\R)\) for all \(f \in \m{F}_k\) then \(c\) can only pass through \(0\) when flat.
This obvious generalises to more general plots.

The key point is that this condition is independent of \(k\).
Thus the underlying Smith space of
 \((\R, \m{T}_{\R}, \m{F}_k)\)
is independent of \(k\).
It is clearly also not the standard Smith structure on \R, in fact it is the Smith structure on \R with a kink at \(0\).
Hence the functor
 \(\SiSm \colon \kcat \to \scat\)
does not take colimits to colimits and so cannot be a left adjoint.

\medskip

Now let us consider the action of
 \(\SmSi \colon \scat \to \kcat\)
on limits.
For \(t \in \R\) define a Smith space
 \((\R, \m{T}_\R, \m{F}_t)\)
by putting a kink in \R at \(t\).
There is an obvious morphism from each of these spaces to \R with its standard Smith space structure.
We consider the limit of this.
The underlying set is again \R and the underlying topology is the standard one.
As we have said before,  the plots into the Smith space structure are constant whence the functionals are all continuous functionals.
Thus the limiting Smith space is \((\R, \m{T}_\R, C(\R,\R))\).

Now let us transport this to \(\kcat\) via \(\SmSi\).
The underlying set of the limit is again \(\R\).
The functionals in the Sikorski space structure on the limit must be the smallest family which contains all the \(\m{F}_t\).
This is the family of piecewise\hyp{}smooth maps, where we interpret ``piecewise\hyp{}smooth'' to mean smooth except at a finite number of points and at those points all left and right derivatives exist.
This is not
 \(\SmSi(\R, \m{T}_\R, C(\R,\R))\)
and hence \(\SmSi\) does not preserve limits.

\paragraph{Fr\"olicher and Souriau.}
Let us consider
 \(\SoFr \colon \dcat \to \fcat\).
This is a left adjoint and so preserves colimits.
Let us consider its action on limits.

Let
 \(\gamma \colon \R \to \R^2\)
be a smooth curve.
Define a diffeology on \(\R^2\), \(\m{D}_{\gamma}\), by starting with the family of smooth curves
 \(\beta \colon U \to \R^2\)
for which
 \(\beta^{-1}(\overline{\gamma(\R)})\)
has void interior, then take the minimal diffeology containing such curves.
The condition ensures that the resulting diffeology does not contain \(\gamma\).

The family
 \(\{\m{D}_{\gamma} : \gamma \in \Ci(\R, \R^2)\}\)
defines a family of Souriau spaces.
This has a limit which is formed by intersecting the diffeologies.
As
 \(\gamma \notin \m{D}_{\gamma}\),
we obtain the discrete diffeology containing only the constant maps.
The underlying Fr\"olicher structure of this is the corresponding discrete Fr\"olicher structure.

Let us consider the Fr\"olicher structure of \((\R^2, \m{D}_\gamma)\).
Clearly, a smooth map \(f \colon \R^2 \to \R\)
defines a morphism of Souriau spaces
 \((\R^2, \m{D}_\gamma) \to (\R, \m{D}_{\R})\).
The converse also holds as can be proved by examining the crucial step in the proof of Boman's result in \cite[3.4]{akpm}.
This step is the ``(4) \(\implies\) (3)''.
We need to show that a morphism
 \(f \colon (\R^2, \m{D}_\gamma) \to (\R, \m{D}_{\R})\)
is a \cimap.
Clearly, it is smooth away from the image of \(\gamma\).
As the closure of \(\im(\gamma)\) has void interior (since \(\gamma\) is smooth), we can choose the smooth curve, \(c\), from \cite[3.4]{akpm} with the crucial property that
 \(c^{-1}(\overline{\gamma(\R)})\)
has void interior.
Thus \(f\) is a \cimap.
Hence the underlying Fr\"olicher structure of \((\R^2, \m{D}_\gamma)\)
is the standard Fr\"olicher structure on \(\R^2\).
Thus the limit of the family
 \(\{\SoFr(\R^2, \m{D}_\gamma)\}\)
is again the standard Fr\"olicher structure on \(\R^2\).

Hence \(\SoFr\) does not preserve limits and so cannot be a right adjoint.

\medskip

The functor
 \(\FrSo \colon \fcat \to \dcat\)
is a right adjoint and so preserves limits.
Let us consider its action on colimits.
With all the spaces having their natural Fr\"olicher structures, consider the coequaliser of the maps
 \(x, o \colon \R \to \R^2\)
given by \(x(t) = (t,0)\) and \(o(t) = (0,0)\).
The underlying set of this coequaliser is \(\R^2\) ``pinched'' along the \(x\)\hyp{}axis, let us call this \(X\) and write \(q \colon \R^2 \to X\) for the projection.
The Fr\"olicher structure is given by taking those functionals \(f \colon X \to \R\) which pull\hyp{}back to \cimaps on \(\R^2\).
Let \((X, \m{C}, \m{F})\) be this Fr\"olicher structure.

Applying \(\FrSo\) to this diagram of Fr\"olicher spaces results in the standard diffeologies on \(\R\) and \(\R^2\).
The underlying set of the coequaliser is again \(\R^2\) ``pinched'' along the \(x\)\hyp{}axis, i.e.~\(X\), and the diffeology consists of those plots which factor through the projection \(\R^2 \to X\).
Let us write this as \((X, \m{D})\).
Let us compare \(\m{D}\) with the diffeology coming from applying \(\FrSo\) to \((X, \m{C}, \m{F})\).
By construction, \(\m{C}\) is in the diffeology of \(\FrSo(X, \m{C}, \m{F})\).

Let
 \(c_\flat \colon \R \to \R\)
be a strictly increasing curve which is smooth, maps \(0\) to \(0\), and is flat at \(0\).
Let
 \(c_\natural \colon \R \to \R^2\)
be the curve
\[
 c_\natural(t) = \begin{cases} (-1, c_\flat(t)) &
  t < 0, \\
  (0,0) & t = 0, \\
  (1, c_\flat(t)) & t > 0.
 \end{cases}
\]
Let
 \(c_\sharp = q c_\natural \colon \R \to X\).
Clearly, any two lifts of \(c_\sharp\) through \(q\) can only differ in their value at \(0\) since \(q\) is a bijection away from the \(x\)\hyp{}axis.
Thus there cannot be a continuous lift of \(c_\sharp\) and so it cannot be in \(\m{D}\).

Let us show that \(c_\sharp\) is in \(\m{C}\).
Let \(f \colon X \to \R\) be in \(\m{F}\).
Then
 \(f c_\sharp = f q c_\natural\)
and
 \(f q \colon \R^2 \to \R\)
is a \cimap.
Clearly, \(f q c_\natural\) is smooth away from \(0\).
Moreover, we see that all left and right limits of \(f q c_\natural\) and its derivatives exist at \(0\).
For \(f q c_\natural\) itself, these limits are the same as they are \(f q(-1, 0)\) and \(f q(1, 0)\).
For the derivatives, \(f q\) is a \cimap whence, as \(c_\flat\) is flat at \(0\), by the chain rule all the left and right limits of the derivatives at \(0\) are \(0\).
Hence \(f q c_\natural\) is a \cimap and so \(c_\sharp\) is in \(\m{C}\), whence in the diffeology of \(\FrSo(X, \m{C}, \m{F})\).

Hence \(\FrSo\) does not take colimits to colimits and so cannot be a left adjoint.

\paragraph{Souriau and Chen.}
The functor \(\ChSo \colon \ccat \to \dcat\) has both a left and a right adjoint.
Let us show that neither of its adjoints has further adjoints.

Consider first \(\SoCht \colon \dcat \to \ccat\).
This is a right adjoint and so preserves limits.
Let us consider its action on colimits.

Let \((X, \m{P})\) be a Chen space.
From the axioms of a Chen structure, every plot \(\phi \colon C \to X\) in \(\m{P}\) underlies a morphism of Chen spaces \(\phi \colon (C, \m{P}_C) \to (X, \m{P})\).
Since the identity lies in \(\m{P}_C\) we see that
\[
  (X, \m{P}) = \varinjlim_{\m{P}} (C, \m{P}_C).
\]
Moreover, \((C, \m{P}_C) = \SoCht(C, \m{D}_C)\) where \(\m{D}_C\) is the standard diffeology on \(C\).
Hence every Chen space is the colimit of things in the image of \(\SoCht\).
As \(\SoCht\) is not dense and \(\dcat\) is cocomplete, it cannot preserve colimits.

Now let us consider \(\SoChb \colon \dcat \to \ccat\).
This is a left adjoint and so preserves colimits.
Let us consider its action on limits.

First we observe that if \(U\) is an open subset of some Euclidean space equipped with its standard diffeology then \(\SoChb(U, \m{D}_U) = \SoCht(U, \m{D}_U)\).
Secondly we observe that the standard Chen structure and standard diffeology on \([0,1]\) are both limits in their respective categories of the family \(\{(-\epsilon, 1 + \epsilon)\}\), again with their standard structures.
Since \(\SoChb([0,1])\) is not the standard Chen structure on \([0,1]\) we see that \(\SoChb\) does not preserve limits.

\section{More on Equivalences}
\label{sec:equiv}

In the previous sections we have constructed a plethora of functors between the various categories of smooth spaces and shown that none of them are equivalences.
However that does not exclude the possibility that there are other functors between these categories that are equivalences.
In this section we consider this.
Our conclusion is that these categories are not equivalent.
Our strategy is to find things in each that may be termed \emph{categorical invariants} which must, therefore, be preserved by any equivalence.

The first such invariant is the (rather, a) terminal object.
This is obviously preserved (up to equivalence) by an equivalence of categories.
The reason that this is an important invariant is that all of our categories are equipped with faithful functors to the category of sets, establishing them as ``concrete categories''.
As we shall show, that the terminal object is a categorical invariant means that any \emph{arbitrary} equivalence defines an isomorphism of concrete categories.

The relevant properties of the categories are as follows.
All but one of these are standard properties of concrete categories.
\begin{description}
\item[Construct:] That is, concrete over \(\Set\).

\item[Cocomplete.]

\item[Amnestic:] This means that if \(A\) and \(B\) are two objects in one of our categories with the same underlying set such that the identity on the underlying set lifts to morphisms \(A \to B\) and \(B \to A\) (necessarily isomorphisms), then \(A = B\).

\item[Transportable:] This means that if \(A\) is an object in one of our categories and \(X\) is a set isomorphic to the underlying set of \(A\) then there is an object, say \(B\), in the category with underlying set \(X\) such that the isomorphism between \(X\) and the underlying set of \(A\) lifts to an isomorphism between \(B\) and \(A\).

As the categories are amnestic, the object \(B\) is unique.

\item[Terminally concrete:] By this we mean that the underlying\hyp{}set functor is equivalent to the evaluation of the hom\hyp{}functor on a terminal object.
\end{description}

\begin{proposition}
Let \(\m{A}\) and \(\m{B}\) be cocomplete, amnestic, transportable constructs that are terminally concrete.
If \(\m{A}\) and \(\m{B}\) are equivalent then they are isomorphic as constructs.
\end{proposition}

\begin{proof}
Let us write \(\sabs{\cdot}\) for the underlying set functors of both \(\m{A}\) and \(\m{B}\).
Let \(\func{F} \colon \m{A} \to \m{B}\) and \(\func{G} \colon \m{B} \to \m{A}\) be two functors giving an equivalence of categories.
As these are equivalences, they preserve categorical constructions up to natural isomorphism.
Since anything isomorphic to a terminal object is again a terminal object, if \(*\) is a terminal object for, say, \(\m{A}\) then \(\func{F}(*)\) is a terminal object for \(\m{B}\).

We therefore have natural isomorphisms
\[
  \sabs{A} \cong \m{A}(*, A) \cong \m{B}(\func{F}(*), \func{F}(A)) \cong \sabs{\func{F}(A)}
\]
and similarly a natural isomorphism \(\sabs{B} \cong \sabs{\func{G}(B)}\).

As \(\m{B}\) is transportable, for each object \(A\) of \(\m{A}\) there is an object \(\hat{\func{F}}(A)\) in \(\m{B}\) with underlying set \(\sabs{A}\) such that the isomorphism \(\sabs{A} \cong \sabs{\func{F}}(A)\) lifts to an isomorphism \(\hat{\func{F}}(A) \cong \func{F}(A)\).
We can then extend \(\hat{\func{F}}\) to a functor by using the isomorphism \(\hat{\func{F}}(A) \cong \func{F}(A)\) to transfer \(\func{F}(f) \colon \func{F}(A_1) \to \func{F}(A_2)\) to \(\hat{\func{F}}(A_1) \to \hat{\func{F}}(A_2)\).
The isomorphisms \(\hat{\func{F}}(A) \cong \func{F}(A)\) then fit together to define a natural isomorphism \(\hat{\func{F}} \cong \func{F}\).
The resulting functor \(\hat{\func{F}}\colon \m{A} \to \m{B}\) is then a concrete functor.

A similar construction results in \(\hat{\func{G}} \colon \m{B} \to \m{A}\).
Since \(\hat{\func{F}}\) is equivalent to \(\func{F}\) and \(\hat{\func{G}}\) to \(\func{G}\), \(\hat{\func{F}}\) and \(\hat{\func{G}}\) still define equivalences of categories.

Now let us consider the natural isomorphism \(1 \cong \hat{\func{G}}\hat{\func{F}}\).
As \(\m{A}\) is cocomplete and terminally concrete, the underlying\hyp{}set functor has a left adjoint given by \(X \mapsto \scoprod_X *\).
To see this, observe that there are natural isomorphisms
\[
  \Set(X, \sabs{A}) \cong \Set(*, \sabs{A})^X \cong \sabs{A}^X \cong \m{A}(*, A)^X \cong \m{A}(\scoprod_X *, A).
\]
Using the transportability, we can adjust this functor to a functor \(\func{D} \colon \Set \to \m{A}\) such that \(\sabs{\func{D}(X)} = X\).
The natural transformation \(\func{D}(\sabs{A}) \to A\) therefore covers the identity on \(\sabs{A}\).

As the functor \(\func{D}\) was defined using standard categorical constructions, \(\hat{\func{G}}\hat{\func{F}}\) preserves \(\func{D}\) up to natural isomorphism.
However, as everything is now set\hyp{}preserving, we see that \(\sabs{\hat{\func{G}} \hat{\func{F}} \func{D}(X)} = X\) from which we deduce that \(\hat{\func{G}}\hat{\func{F}}\func{D} = \func{D}\).

Applying the natural isomorphism \(1 \cong \hat{\func{G}}\hat{\func{F}}\) to the natural transformation \(\func{D}(\sabs{A}) \to A\) we obtain the following commutative diagram.
\begin{centre}
\begin{tikzpicture}[node distance=1.8cm, auto,>=latex', thick]
\node (da) {\(\func{D}(\sabs{A})\)};
\node[right of=da, node distance=2cm] (a) {\(A\)};
\node[below of=da] (gfda) {\(\hat{\func{G}}\hat{\func{F}}\func{D}(\sabs{A})\)};
\node[below of=a] (gfa) {\(\hat{\func{G}}\hat{\func{F}}(A)\)};
\path[->] (da) edge (a)
(a) edge (gfa)
(da) edge (gfda)
(gfda) edge (gfa);
\end{tikzpicture}
\end{centre}

Mapping down to \(\Set\) we obtain the diagram

\begin{centre}
\begin{tikzpicture}[node distance=1.8cm, auto,>=latex', thick]
\node (a) {\(\sabs{A}\)};
\node[right of=a, node distance=2cm] (b) {\(\sabs{A}\)};
\node[below of=a] (c) {\(\sabs{A}\)};
\node[below of=b] (d) {\(\sabs{A}\)};
\path[->] (a) edge (b)
(b) edge (d)
(a) edge (c)
(c) edge (d);
\end{tikzpicture}
\end{centre}

in which the horizontal maps are the identity as is the left\hyp{}hand vertical map.
The right\hyp{}hand vertical map is therefore the identity.
From this we conclude that the natural isomorphism \(1 \cong \hat{\func{G}}\hat{\func{F}}\) is also concrete.

A similar story holds for the natural isomorphism \(1 \cong \hat{\func{F}}\hat{\func{G}}\).
Hence the equivalence of categories defined by \(\hat{\func{F}}\) and \(\hat{\func{G}}\) is an equivalence of constructs.

Finally, we observe that an equivalence of amnestic constructs must be an isomorphism.
\end{proof}

A simple adaptation of the above also shows that if we are interested in adjunctions, then many of these have similar properties.

\begin{proposition}
Let \(\m{A}\) and \(\m{B}\) be cocomplete, amnestic, transportable constructs that are terminally concrete.
Suppose that \(\func{F} \colon \m{A} \to \m{B}\) is left adjoint to \(\func{G} \colon \m{B} \to \m{A}\) and \(\func{F}\) preserves terminal objects.
Then \(\func{G}\) is equivalent to a set\hyp{}preserving functor. \noproof
\end{proposition}

The above says that if we wish to look for equivalences between our categories then it is sufficient to look for isomorphisms of constructs.
Another way of saying that is to say that if two of our categories have underlying category \xcat then they are equivalent via an arbitrary equivalence if and only if they are equivalent via one which preserves the underlying category.

To show that this is also true when the underlying category is \ycat we need to analyse the ``topology classifier'': \(\{0, 1\}\) with its order topology.

\begin{proposition}
Let \(\acat\) be a category of smooth spaces built from our general recipe using \ycat as the underlying category.
Any set\enhyp{}preserving isomorphism of \acat preserves the underlying topological spaces.
\end{proposition}

\begin{proof}
To see this we observe that the topology on a set is determined by the set of continuous functions to \(\{0,1\}\) equipped with the order topology \(\{\emptyset, \{0\}, \{0,1\}\}\), which we shall denote by \(\m{T}_0\).

We define the structure of an \aobj on \((\{0,1\}, \m{T}_0)\) by taking the indiscrete structure.
This has all continuous maps from \tobjs to \((\{0,1\}, \m{T}_0)\) as input test functions and then the appropriate output test functions to satisfy the forcing condition for \acat.
This has the property that any continuous map \(\forfunc[\aobj] \to (\{0,1\}, \m{T}_0)\) lifts to an \amor.
Thus the topology on the underlying \yobj of an \aobj[\aobj] is determined by the set of \amors from \aobj to the indiscrete \(\acat\)\enhyp{}structure on \((\{0,1\}, \m{T}_0)\).

A set\enhyp{}preserving automorphism of \acat takes the indiscrete \(\acat\)\enhyp{}structure on \((\{0,1\}, \m{T}_0)\) to another \aobj with underlying set \(\{0,1\}\).
Now \(\{0,1\}\) has four topologies: discrete, indiscrete, and the two order topologies.
Each of these has a corresponding indiscrete \(\acat\)\enhyp{}structure.
As these indiscrete structures are expressible as colimits, a set\enhyp{}preserving automorphism of \acat must take indiscrete structures to indiscrete structures.
Thus it defines an order\enhyp{}isomorphism of the lattice consisting of the indiscrete \(\acat\)\enhyp{}structures on the four topologies on \(\{0,1\}\).
In particular, \((\{0,1\}, \m{T}_0)\) with the indiscrete \(\acat\)\enhyp{}structure either maps to itself or to the indiscrete \(\acat\)\enhyp{}structure on the opposite order topology.

Suppose that there is a set\enhyp{}preserving automorphism of \(\acat\) which sends the indiscrete structure on \((\{0,1\}, \m{T}_0)\) to that on the other order topology.

Let \(X\) be an infinite set and let \(\m{T}\) be the topology whose closed sets are the finite subsets.
Equip \((X, \m{T})\) with its indiscrete \(\acat\)\enhyp{}structure.
The \amors between indiscrete \(\acat\)\enhyp{}structures are simply the continuous maps on the underlying topological spaces and thus the \amors from the indiscrete \(\acat\)\enhyp{}structure on \((X,\m{T})\) to that on \((\{0,1\}, \m{T}_0)\) are the characteristic functions of the closed sets.

Now under the supposed set\enhyp{}preserving automorphism of \(\acat\), the indiscrete structure on \((X, \m{T})\) is sent to some \aobj[\aobj] with underlying set \(X\) and the property that the \amors from \(\aobj\) to the indiscrete \(\acat\)\enhyp{}structure on \((\{0,1\}, \m{T}_1)\) are precisely the characteristic functions of the closed sets in \(\m{T}\).

Now \(\aobj\) has underlying topological space \((X, \m{T}')\) and \(\amors\) are continuous.
This means that a closed set in \(\m{T}\) is open in \(\m{T}'\).
Hence singleton sets are open and so \(\m{T}'\) is the discrete topology.
Then \emph{all} maps \(X \to \{0,1\}\) lift to \amors from \(\aobj\) to the indiscrete \(\acat\)\enhyp{}structure on \((\{0,1\}, \m{T}_1)\).
As \(X\) is infinite, this contradicts the statement that the \amors are the characteristic functions of the closed sets in \(\m{T}\).

Hence any set\enhyp{}preserving automorphism of \acat also preserves the underlying topologies.
\end{proof}

To further distinguish between the categories we note that under a set\enhyp{}preserving isomorphism, not only the underlying set is preserved but also the endomorphism monoid of the object (as a submonoid of the endomorphism monoid of the underlying set).

The first application of this enables us to distinguish the category of Fr\"olicher spaces from those of Chen and Souriau spaces.

\begin{proposition}
\label{prop:uniquer}
The only Fr\"olicher structures on \R whose endomorphism monoid contains \(\Ci(\R,\R)\) are the standard, the discrete, and the indiscrete structures.
In particular, the only Fr\"olicher structre on \R whose endomorphism monoid is precisely \(\Ci(\R,\R)\) is the standard structure.
\end{proposition}

The proof of this depends on a modification of a result of Takens~\cite{ft}.
(Interestingly, this paper was motivated by reading about what we are calling Souriau spaces.)

\begin{theorem}{{\cite[Theorem~1]{ft}}}
Let \(\Phi \colon M_1 \to M_2\) be a bijection between two smooth \(n\)\enhyp{}manifolds such that \(\lambda \colon M_2 \to M_2\) is a diffeomorphism iff \(\Phi^{-1} \circ \lambda \circ \Phi\) is a diffeomorphism.
Then \(\Phi\) is a diffeomorphism. \noproof
\end{theorem}

\begin{proof}[{Proof of Proposition~\ref{prop:uniquer}}]
Obviously those three structures have the property that their endomorphism monoids contain \(\Ci(\R,\R)\).
Thus, to show the converse, let \((\R, \m{C}, \m{F})\) be a Fr\"olicher space with endomorphism monoid containing \(\Ci(\R,\R)\) which is neither discrete nor indiscrete.

Then for a smooth map \(\psi \colon \R \to \R\), \(\phi \psi \alpha \in \Ci(\R,\R)\) for all \(\phi \in \m{F}\) and \(\alpha \in \m{C}\).
In particular, this implies that if \(\psi \in \Ci(\R,\R)\) and \(\alpha \in \m{C}\) then \(\psi \alpha \in \m{C}\), similarly for \(\m{F}\).

We start by showing that \(\m{C}\) consists of continuous maps.
Suppose, for a contradiction, that \(\alpha \in \m{C}\) is not continuous at, say, \(x \in \R\).
Then there is a sequence \((x_n) \to x\) such that \(\alpha(x_n) \not\to \alpha(x)\).
By passing to a subsequence if necessary, we can assume that there is some \(\epsilon > 0\) such that \(\abs{\alpha(x_n) - \alpha(x)} \ge \epsilon\) for all \(n\).
Let \(\psi\) be a bump function at \(\alpha(x)\) which vanishes outside \((\alpha(x) - \epsilon, \alpha(x) + \epsilon)\).
Then \(\psi \alpha \in \m{C}\) but \(\psi \alpha(x) = 1\) and \(\psi \alpha(x_n) = 0\) for all \(n\).
Now for \(\phi \in \m{F}\), \(\phi \psi \alpha \in \Ci(\R,\R)\).
Thus \(\phi \psi \alpha\) is continuous and so \((\phi \psi \alpha(x_n)) \to \phi \psi \alpha(x)\).
Hence \(\phi(0) = \phi(1)\).
There was nothing special about \(0\) and \(1\) and thus \(\phi\) is constant.
As this holds for arbitrary \(\phi\), \(\m{F}\) consists only of constant functions.
However, the only Fr\"olicher structure on \R for which this holds is the indiscrete one which we assumed was not the structure under consideration.
Thus \(\m{C}\) consists of continuous functions.

Now we want to show that at least one curve in \(\m{C}\) is injective on some interval.
Suppose, for another contradiction, that this is not true.
Then for every curve \(\alpha \in \m{C}\) and interval \((a,b) \subseteq \R\), \(\alpha\) is not injective on \((a,b)\).
Let \(\phi \in \m{F}\).
Then \(\phi \alpha\) has the same property.
As \(\phi \alpha\) is smooth, this implies that \(\phi \alpha\) has zero differential everywhere and hence is constant.
As this holds for all \(\phi\) and all \(\alpha\), either \(\m{C}\) or \(\m{F}\) must consist only of constant functions.
This results in either the discrete or indiscrete Fr\"olicher structures, which again are ruled out by assumption.

Thus there is some \(\alpha \in \m{C}\) which is injective on some interval.
By restricting first to a compact subinterval and then to an open subinterval of this, we see that \(\alpha\) is a homeomorphism from some open interval to another open interval.
By pre\hyp{}\hyp{}composing with an appropriate smooth function, we obtain a homeomorphism \(\beta \colon \R \to (a,b)\) in \(\m{C}\).
Choose a diffeomorphism \(\psi \colon (a,b) \to \R\).
We claim that \(\psi \beta \in \m{C}\).
To see this, note that for any bounded interval \(I\) in \R, there is a smooth function \(\psi_I \colon \R \to \R\) such that \(\psi_I \beta\) and \(\psi \beta\) agree on \(I\).
By our assumption on \((\R, \m{C}, \m{F})\), \(\psi_I \beta \in \m{C}\) and thus \(\psi \beta \in \m{C}\).

Hence \(\m{C}\) contains a homeomorphism, say \(\gamma\).

Now we turn our attention to \(\m{F}\).
For \(\phi \in \m{F}\), \(\phi \gamma \in \Ci(\R,\R)\).
As \(\gamma\) is a homeomorphism, we can write \(\phi\) as \(\phi \gamma \gamma^{-1}\).
Thus \(\m{F} \subseteq \Ci(\R,\R) \gamma^{-1}\).
Since \(\m{F}\) must contain a non\hyp{}constant map, there is some non\hyp{}constant \(\psi \in \Ci(\R,\R)\) with \(\psi \gamma^{-1} \in \m{F}\).
As \(\psi\) is non\hyp{}constant, there are bounded intervals, say \(I, J \subseteq \R\), such that \(\psi\) restricts to a diffeomorphism \(I \to J\).
By restricting if necessary, we can assume that the inverse is the restriction of a smooth map \(\theta \colon \R \to \R\).
Then \(\theta \psi\) is the identity on \(I\) and so \(\theta \psi \gamma^{-1}\) equals \(\gamma^{-1}\) on \(\gamma(I) (= (\gamma^{-1})^{-1}(I))\).
Since \(\m{F}\) is invariant under translation (as \(\Ci(\R,\R)\) acts on it on the \emph{right}), we see that \(\gamma^{-1}\) is locally in \(\m{F}\) and hence in \(\m{F}\).
Thus \(\m{F} = \Ci(\R,\R) \gamma^{-1}\).

We therefore see that \(\m{C} = \gamma \Ci(\R,\R)\).
Thus the Fr\"olicher structure is \((\R, \gamma \Ci(\R,\R), \Ci(\R,\R) \gamma^{-1})\).
The endomorphism monoid of this structure is \(\gamma \Ci(\R,\R) \gamma^{-1}\).
Thus \(\gamma \colon \R \to \R\) is a homeomorphism with the property that \(\gamma \Ci(\R,\R) \gamma^{-1} \supseteq \Ci(\R,\R)\).
Equivalently, that \(\gamma^{-1} \Ci(\R,\R) \gamma \subseteq \Ci(\R,\R)\).
This does not quite fit the hypotheses for Takens' theorem as that requires equality here.
However, careful examination of \cite{ft} shows that the proof still applies.
It seems that the ``if and only if'' part of the hypotheses of Takens' theorem is required to prove that the bijections are homeomorphisms.
As we already have that property, we can assume weaker conditions.

In detail, the proof of \cite[Lemma~3.1]{ft} shows that if (using notation of \cite{ft}) \(\Phi \colon \R \to \R\) is a homeomorphism such that \(\Phi^{-1}(\Phi(y) + c)\) is differentiable for all \(y, c \in \R\) then \(\Phi\) and \(\Phi^{-1}\) are differentiable everywhere.
Then the rest of the proof of \cite[Theorem~2]{ft} for the case \(n = 1\) applies as stated.
Hence our Fr\"olicher structure on \R is the standard one.
\end{proof}

\begin{corollary}
The only set\enhyp{}preserving automorphism of the category of Fr\"olicher spaces is the identity. \noproof
\end{corollary}

The proof of the proposition involves much of the structure specific to Fr\"olicher spaces and thus would seem difficult to generalise to the other categories of interest, save perhaps for Smith spaces.
However, there is still something that can be said without too much extra work by looking at underlying Fr\"olicher spaces.

\begin{defn}
In each category, a \emph{pseudo\enhyp{}\R{}} is an object with underlying object \R (with its standard topology if appropriate) and endomorphism monoid \(\Ci(\R,\R)\).
\end{defn}

For Fr\"olicher spaces there is only one pseudo\enhyp{}\R.
We strongly suspect that this is true for the other spaces as well (though note that it is not true when the forcing conditions are very weak) but even without direct proof of this we can still show that the standard structure on \R is special.

The proposition above shows that there are at most three choices for the underlying Fr\"olicher space of a pseudo\enhyp{}\R.

\begin{proposition}
In the categories of Chen and Souriau spaces, there are no pseudo\enhyp{}\R{}s which map to the indiscrete Fr\"olicher structure on \R and only one that maps to the standard Fr\"olicher structure on \R.
In particular, if \(\R_\alpha\) is a pseudo\enhyp{}\R and \R denotes the ``standard'' \R then
\[
  \Hom{\acat}{\R_\alpha}{\R} = \begin{cases}
  \Ci(\R,\R) & \text{if } \R_\alpha = \R, \\
  \R & \text{otherwise}
  \end{cases}
\]
where \(\acat\) is either \(\ccat\) or \(\dcat\).

In the categories of Smith and Sikorski spaces, there are no pseudo\enhyp{}\R{}s which map to the discrete Fr\"olicher structure on \R and only one that maps to the standard Fr\"olicher structure on \R.
In particular, if \(\R_\alpha\) is a pseudo\enhyp{}\R and \R denotes the ``standard'' \R then
\[
  \Hom{\acat}{\R}{\R_\alpha} = \begin{cases}
  \Ci(\R,\R) & \text{if } \R_\alpha = \R, \\
  \R & \text{otherwise}
  \end{cases}
\]
where \(\acat\) is either \(\kcat\) or \(\scat\).
\end{proposition}

\begin{proof}
The arguments are essentially the same for all four cases.
We shall consider Souriau spaces in detail and then indicate the necessary changes for the others.

When forming a Fr\"olicher space from a Souriau space one only adds functions from \R to those already in the plots.
The output functions stay the same.
Thus if a Souriau space maps to a discrete Fr\"olicher space, the only plots from \R in the Souriau structure can have been constant plots.
Therefore the only plots from any test object can have been locally constant and so the Souriau structure must have been the discrete one.
In particular, only the discrete Souriau structure on \R maps to the discrete Fr\"olicher structure and the discrete Souriau structure is not a pseudo\enhyp{}\R.

Now suppose that a \(\R_\alpha\) is a pseudo\enhyp{}\R which maps to the standard Fr\"olicher structure on \R.
By the above, the Souriau structure cannot have been the discrete one and thus there is a non\enhyp{}constant map \(\R \to \R\) in the plots of \(\R_\alpha\).
As this becomes a smooth in the Fr\"olicher structure, it must be a \cimap.
It is therefore a diffeomorphism when restricted to some interval.
By precomposition with an appropriate \cimap, we can obtain the identity map on some interval in the Souriau structure on \(\R_\alpha\).
As the endomorphism monoid of \(\R_\alpha\) is \(\Ci(\R,\R)\), the Souriau structure must be translation invariant.
The identity on \R is therefore locally in the Souriau structure of \(\R_\alpha\) and hence, by the sheaf condition, the identity itself is in the structure.
Thus the Souriau structure on \(\R_\alpha\) contains \(\Ci(\R,\R)\).
On the other hand, it can contain no more than this as it maps to the standard Fr\"olicher structure on \R.
Hence \(\R_\alpha\) is the standard Souriau structure on \R.

All the other pseudo\enhyp{}\R{}s must map to the indiscrete Fr\"olicher structure.
As the functor from Souriau spaces to Fr\"olicher spaces is faithful, the set of morphisms from another pseudo\enhyp{}\R{} \(\R_\alpha\) to \R must be contained in the set of morphisms on the underlying Fr\"olicher spaces.
However, the only morphisms from the indiscrete Fr\"olicher structure on \R to the standard one are the constant maps.

The argument is exactly the same for Chen spaces.
For Smith and Sikorski, one swaps the r\^oles of the input and output functions in the various arguments.
\end{proof}

The key in this proof is obviously that the forcing conditions in each case are sufficient to guarantee that if the identity on \R is locally a test function then it is a test function.

This result singles out the standard structure on \R in each category because although, in the Chen and Souriau categories, there are no morphisms into \R with its standard structure from a pseudo\enhyp{}\R, there must be morphisms out of it because a pseudo\enhyp{}\R cannot be the discrete structure.
It can be thought of, therefore, as a sort of ``initial'' pseudo\enhyp{}\R for Chen, Souriau, and Fr\"olicher spaces and a ``final'' pseudo\enhyp{}\R for Sikorski, Smith, and Fr\"olicher spaces.

\begin{corollary}
Any set\enhyp{}preserving isomorphism between two of the categories of Chen, Souriau, or Fr\"olicher spaces preserves \R with its standard structure.
Similarly between two of the categories of Sikorski, Smith, or Fr\"olicher spaces.
\end{corollary}

\begin{proof}
The two cases are formally similar so we shall just do one.
Let \(\acat\) and \(\bcat\) be two of the categories of Chen, Souriau, or Fr\"olicher spaces and suppose that \(\func{F} \colon \acat \to \bcat\) is a set\enhyp{}preserving isomorphism.
Then \(\func{F}\) defines a bijection from the set of pseudo\enhyp{}\R{}s in \(\acat\) to that in \(\bcat\).
Let us write \R for the standard structure on \R in each category.

Let \(\R_\alpha\) be such that \(\func{F}(\R_\alpha) = \R\).
Then
\[
  \Hom{\acat}{\R_\alpha}{\R} = \Hom{\bcat}{\func{F}(\R_\alpha)}{\func{F}(\R)} = \Hom{\bcat}{\R}{\func{F}(\R)}.
\]
As \(\func{F}(\R)\) is a pseudo\enhyp{}\R, it is not discrete and so the latter hom\hyp{}set must contain something non\hyp{}trivial.
Thus \(\Hom{\acat}{\R_\alpha}{\R} \ne \R\) and so \(\R_\alpha = \R\).
Hence \(\func{F}\) preserves the ``standard'' \R.
\end{proof}

\begin{corollary}
Any set\enhyp{}preserving isomorphism between two of the categories of Chen, Souriau, or Fr\"olicher spaces on the one hand, or between two of the categories of Sikorski, Smith, or Fr\"olicher spaces on the other, preserves the test objects.
\end{corollary}

\begin{proof}
Once \R with its standard structure is fixed then so are the product structures on the spaces \(\R^n\) and the subspace structures on any subsets, in particular on the test objects.
\end{proof}

\begin{corollary}
There are no equivalences between any of the categories of Chen, Souriau, and Fr\"olicher spaces.
Nor are there any between any of the categories of Sikorski, Smith, and Fr\"olicher spaces.
\end{corollary}

\begin{proof}
The argument is the same in all cases so let us take Chen and Souriau as an example.
Any equivalence between these two defines a set\enhyp{}preserving isomorphism which must preserve the test objects.
We therefore find that if \(\func{F} \colon \dcat \to \ccat\) is set\enhyp{}preserving isomorphism then for a \dobj[\dobj] and \tobj[\tobj]
\[
  \itest_{\dobj}(\tobj) = \Hom{\dcat}{\sfunc(\tobj)}{\dobj} = \Hom{\ccat}{\sfunc(\tobj)}{\func{F}(\dobj)}
\]
and thus applying the restriction functor to \(\func{F}(\dobj)\) yields \(\dobj\) again.
However, the restriction functor is not an isomorphism as there are at least two distinct Chen structures on \([0,1]\) which restrict to the same Souriau structure.
Therefore \(\func{F}\) cannot exist.

This strategy works for each pair: we find two objects in the ``higher'' category that restrict to the same object in the ``lower'' one.

For Souriau and Fr\"olicher we take \(\R^2\) with its standard Souriau structure and the ``wire'' structure where all input test functions factor through \R.

For Smith and Fr\"olicher we take \R with its usual topology and all continuous maps, and \R with its discrete topology and all maps.

For Sikorski and Smith we take \R with its usual topology and all continuous maps, and \R with its usual topology and all piecewise\enhyp{}smooth maps.
\end{proof}

There is one final situation to rule out.
Our strategy of focussing on \R works well so long as we restrict ourselves to one ``side'' at a time.
It does not rule out an equivalence between, say, Souriau spaces and Sikorski spaces.
Indeed, swapping ``sides'' like this is a more complicated issue and so we shall take the easy way out by using the fact that in our particular examples we can use the underlying categories to rule out any equivalences.

\begin{proposition}
There are no equivalences between either \ccat or \dcat and \kcat or \scat.
\end{proposition}

\begin{proof}
It is sufficient to look for a set\enhyp{}preserving isomorphism.
Any such isomorphism must preserve the number of structures on a given set.
In the categories of Smith and Sikorski spaces there are four objects with underlying set \(\{0,1\}\).
These are:
\begin{enumerate}
\item \(\{0,1\}\) with the indiscrete topology and constant functions,
\item \(\{0,1\}\) with the order topology and constant functions,
\item \(\{0,1\}\) with the reverse order topology and constant functions,
\item \(\{0,1\}\) with the discrete topology and all functions.
\end{enumerate}
The functions in the first three are forced by the fact that all three have only constant continuous functions to \R.
The forcing conditions for both Smith and Sikorski spaces force the functions for the discrete topology to be all functions.

In the categories of Chen, Souriau, and Fr\"olicher spaces there are only two objects with underlying set \(\{0,1\}\).
These are:
\begin{enumerate}
\item \(\{0,1\}\) with only constant (input) functions, and
\item \(\{0,1\}\) with all (input) functions.
\end{enumerate}

Since these categories are amnestic, any set\enhyp{}preserving isomorphism must restrict to a bijection on the fibres above a given set.
Thus there cannot be an equivalence between one of the categories of Smith or Sikorski spaces on the one hand, and one of the categories of Chen, Souriau, or Fr\"olicher spaces on the other.
\end{proof}

In this section we have proceeded on a mixture of general results and  case\enhyp{}by\enhyp{}case analysis.
It is entirely possible that a fuller analysis of the general structure would lead to an elimination of the more ad hoc aspects, but as the goal was to analyse the specific examples of these categories where a simple ad hoc argument sufficed we deemed it more appropriate to give it than to search for what could be a more complicated but more general result.

\section{Topology}
\label{sec:topology}

Two of the main categories under consideration and three of Chen's definitions involve topology directly.
This proves to be somewhat of a nuisance.
If one takes seriously the maxim that anything related to smooth structure must be detectable by test functions, one must make certain assumptions as to the relationship between the topology and the test functions.
In simple language, if we cannot detect the difference between two topologies with test functions then we should not distinguish between them.

This does not mean that topology does not have a r\^ole to play in the theory of smooth structures.
Rather, it means that the topology must depend on the smooth structure and not the other way around.
It is easy to consistently define a topology on a Fr\"olicher space: either take the curvaceous topology or the functional topology.
However, this topology may not be the one first thought of.

By looking at the evolution of Chen's definitions, it is clear that he came to this opinion on the r\^ole of topology and he eventually disposed of it.
Indeed, one can avoid the issue of topology in Chen's definitions by ensuring that it is sufficiently weak that any ``reasonable'' map is continuous.
The only situation where this might cause a problem is in Chen's first definition where he used Hausdorff spaces.
Fortunately he quickly dropped this condition and, eventually, dropped the topology condition altogether.

However, topology is used in a more intricate fashion in Smith spaces and Sikorski spaces.
In both there is the requirement that all the functionals be continuous.
This puts a limit on how weak the topology could be.
If this were the only place that the topology was used then we could remove it again by ensuring that it is sufficiently strong that any ``reasonable'' map is continuous.
However the other use of the topology in each definition puts a limit on how strong the topology can be.
Nonetheless, it is still possible to modify the definitions to remove the topology.

For Smith spaces it is obvious that if one removes the topology from the picture, and thus the requirement that all the maps involved be continuous, then one simply gets Fr\"olicher spaces.
The completeness axiom is, by Boman's result~\cite{jb3}, equivalent to that of Fr\"olicher.

For Sikorski spaces it is less obvious what to do.
The axioms that the functionals should form an algebra and be closed under post\hyp{}composition by smooth functions do not involve the topology and so can be left as they are.
The interaction between the family of functions and the topology comes in in two places: the continuity of the functionals (referred to above) and the fact that the functionals are locally detectable.
Let us consider these.

Let \(X\) be a set and suppose that \(\m{F}\) is a family of functionals on \(X\) for which there exists some topology \(\m{T}\) making \((X, \m{T}, \m{F})\) into a Sikorski space.
Let \(\m{T}_f\) denote the topology on \(X\) induced by the family of functionals, \(\m{F}\).
As the functionals in \(\m{F}\) must be continuous, the topology \(\m{T}\) is at least as fine as \(\m{T}_f\).
If \(f \colon X \to \R\) is \(\m{T}_f\)\enhyp{}locally in \(\m{F}\) then, since \(\m{T}_f \subseteq \m{T}\), it is \(\m{T}\)\enhyp{}locally in \(\m{F}\) and hence, as \((X, \m{T}, \m{F})\) is a Sikorski space, in \(\m{F}\).
Thus \((X, \m{T}_f, \m{F})\) is a Sikorski space.
In conclusion, if \(\m{F}\) is a family of functions on \(X\) for which \emph{some} topology exists making it into a Sikorski space then we can take the topology to be that defined by \(\m{F}\).

We can use this to drop the topology from the definition.
First, the fact that the functionals in \(\m{F}\) are continuous is now a tautology.
This leaves us with locally detectable.

Let \(f \colon X \to \R\) and \(x \in X\) be such that there is a neighbourhood \(V\) of \(x\) and \(g \in \m{F}\) with \(f \restrict_V = g \restrict_V\).
As \(\m{F}\) is an algebra and is closed under post\hyp{}composition by smooth functions, a basis for the functional topology is given by the sets \(h^{-1}(\R \ssetminus \{0\})\) for \(h \in \m{F}\).
Thus there is some \(h \in \m{F}\) such that \(h(x) \ne 0\) and \(h^{-1}(\R \ssetminus \{0\}) \subseteq V\).
As \(h\) is zero outside \(V\), \(h f = h g\) as functions on \(X\).
As \(\m{F}\) is an algebra, this means that \(h f \in \m{F}\).

Conversely, suppose that \(f \colon X \to \R\) and \(x \in X\) are such that there is some \(h \in \m{F}\) with \(h(x) \ne 0\) and \(h f \in \m{F}\).
As \(h(x) \ne 0\) we can find a function \(g \in \m{F}\) such that \(g(y) = h(y)^{-1}\) in a neighbourhood of \(x\).
Then \(g h f \in \m{F}\) and \(g h f\) agrees with \(f\) in a neighbourhood of \(x\).

Thus a suitable non\hyp{}topological replacement for the ``locally detectable'' axiom of a Sikorski space would be: if \(f \colon X \to \R\) is a functional such that for each \(x \in X\) there is some \(f_x \in \m{F}\) with \(f_x(x) \ne 0\) and \(f f_x \in \m{F}\) then \(f \in \m{F}\).

However, note that this condition does not fit into our general picture.
It cannot be tested by diagrams of the form

\begin{centre}
\begin{tikzpicture}[node distance=1.5cm, auto,>=latex', thick]
    \node (lblank) {};
    \node[right of=lblank] (uVtobj) {\(\uVtobj\)};
    \node[right of=uVtobj] (tobj) {\(\tobj\)};
    \node[right of=tobj] (rblank) {};
\draw[->]
                  (uVtobj) edge[dotted] (tobj)
                  (tobj) edge node[diaglabel] {\(\tmor\)} (rblank)
                  (lblank) edge node[diaglabel] {\(\uVtmor\)} (uVtobj);
\end{tikzpicture}
\end{centre}

Rather we have to extend the notion of a trial to include diagrams of the form

\begin{centre}
\begin{tikzpicture}[node distance=1.5cm, auto,>=latex', thick]
    \node (lblank) {};
    \node[right of=lblank] (times) {\(\times\)};
    \node[right of=uVtobj] (times2) {\(\times\)};
    \node[right of=tobj] (rblank) {};
    \node[above of=times, node distance=1em] (uVtobj) {\(\uVtobj\)};
    \node[below of=times, node distance=1em] (uVtobj2) {\(\uVtobj'\)};
    \node[above of=times2, node distance=1em] (tobj) {\(\tobj\)};
    \node[below of=times2, node distance=1em] (tobj2) {\(\uVtobj\)};
\draw[->]
                  (uVtobj) edge[dotted] (tobj)
                  (uVtobj2) edge node[diaglabel,swap] {\(\phi\)} (tobj2)
                  (times2) edge node[diaglabel] {\(\tmor\)} (rblank)
                  (lblank) edge node[diaglabel] {\(\uVtmor\)} (times);
\end{tikzpicture}
\end{centre}

This is clearly doable, but given that the only example that requires it is this of de\hyp{}topologised Sikorski spaces we have elected not to further burden the rest of the paper.

\section{Non\hyp{}Set\enhyp{}Based Theories}
\label{sec:nonset}

Our general recipe does not cover everything that one would wish to treat as a smooth manifold.
The common property of all our categories is that they are fundamentally set\enhyp{}based theories.
However, this approach does not cover other extensions such as synthetic and non\hyp{}commutative differential geometry where objects  do not have underlying sets.
It is, nonetheless, easy to extend part of the general recipe to non\hyp{}set\enhyp{}based theories.
Instead of starting with triples \((\uobj, \itest, \otest)\) one  starts with pairs \((\itest, \otest, \comp)\) where \(\itest\) and \(\otest\) are, respectively, a contravariant and a covariant functor on the test category and \(\comp\) is a natural transformation of bifunctors from \(\itest \times \otest\) to the hom\hyp{}functor on the test category.

Providing the test category is essentially small, these are the objects of a category which is known, in folklore at least, as the \emph{Isbell Envelope} of the test category.
It can also be understood in terms of profunctors.
The two functors \(\itest \colon \tcat \to \xcat\) and \(\otest \colon \tcat \to \xcat\) are profunctors \(\tcat \to \ocat\) and \(\ocat \to \tcat\) respectively.
Their composition, as profunctors, is the product \(\itest \times \otest\) and the natural transformation \(\itest \times \otest \to \Hom{\tcat}{-}{-}\) is a morphism of profunctors.
Thus we obtain what may be termed a \emph{lax factorisation of \(\Hom{\tcat}{-}{-}\) through \(\ocat\)}.

However it is described, the main difficulty is in selecting the appropriate generalisation of the forcing conditions.
With an underlying category, the forcing conditions are answers to the question ``Which, of the \emph{available} morphisms on the underlying objects should be viewed as smooth?''.
Without an underlying category, the forcing conditions need to answer a more subtle question.
What this question should be is something that will become evident with further study of the extant examples of non\hyp{}set\enhyp{}based theories of generalised smooth objects; a study that we defer to another paper.

\bibliography{arxiv,articles,books,misc}

\begin{thebibliography}{Bom67}

\bibitem[BH]{0807.1704}
John~C. Baez and Alexander~E. Hoffnung.
\newblock Convenient categories of smooth spaces.

\bibitem[Bom67]{jb3}
Jan Boman.
\newblock Differentiability of a function and of its compositions with
  functions of one variable.
\newblock {\em Math. Scand.}, 20:249--268, 1967.

\bibitem[Che73]{kc}
Kuo-tsai Chen.
\newblock Iterated integrals of differential forms and loop space homology.
\newblock {\em Ann. of Math. (2)}, 97:217--246, 1973.

\bibitem[Che75]{kc5}
Kuo~Tsai Chen.
\newblock Iterated integrals, fundamental groups and covering spaces.
\newblock {\em Trans. Amer. Math. Soc.}, 206:83--98, 1975.

\bibitem[Che77]{kc3}
Kuo~Tsai Chen.
\newblock Iterated path integrals.
\newblock {\em Bull. Amer. Math. Soc.}, 83(5):831--879, 1977.

\bibitem[Che86]{kc4}
Kuo~Tsai Chen.
\newblock On differentiable spaces.
\newblock In {\em Categories in continuum physics (Buffalo, N.Y., 1982)},
  volume 1174 of {\em Lecture Notes in Math.}, pages 38--42. Springer, Berlin,
  1986.

\bibitem[CS62]{ycjs}
Yeaton~H. Clifton and J.~Wolfgang Smith.
\newblock Topological objects and sheaves.
\newblock {\em Trans. Amer. Math. Soc.}, 105:436--452, 1962.

\bibitem[Fr{\"o}82]{af}
Alfred Fr{\"o}licher.
\newblock Smooth structures.
\newblock In {\em Category theory (Gummersbach, 1981)}, volume 962 of {\em
  Lecture Notes in Math.}, pages 69--81. Springer, Berlin, 1982.

\bibitem[Isb60]{ji3}
J.~R. Isbell.
\newblock Adequate subcategories.
\newblock {\em Illinois J. Math.}, 4:541--552, 1960.

\bibitem[KM97]{akpm}
Andreas Kriegl and Peter~W. Michor.
\newblock {\em The convenient setting of global analysis}, volume~53 of {\em
  Mathematical Surveys and Monographs}.
\newblock American Mathematical Society, Providence, RI, 1997.

\bibitem[Mos79]{mm2}
Mark~A. Mostow.
\newblock The differentiable space structures of {M}ilnor classifying spaces,
  simplicial complexes, and geometric realizations.
\newblock {\em J. Differential Geom.}, 14(2):255--293, 1979.

\bibitem[Sch08]{us2}
Urs Schreiber.
\newblock Comparative smootheology, 2008.

\bibitem[Sik72]{rs3}
Roman Sikorski.
\newblock Differential modules.
\newblock {\em Colloq. Math.}, 24:45--79, 1971/72.

\bibitem[Smi66]{js2}
J.~Wolfgang Smith.
\newblock The de {R}ham theorem for general spaces.
\newblock {\em T\^ohoku Math. J. (2)}, 18:115--137, 1966.

\bibitem[Sou80]{js}
J.-M. Souriau.
\newblock Groupes diff\'erentiels.
\newblock In {\em Differential geometrical methods in mathematical physics
  (Proc. Conf., Aix-en-Provence/Salamanca, 1979)}, volume 836 of {\em Lecture
  Notes in Math.}, pages 91--128. Springer, Berlin, 1980.

\bibitem[Tak79]{ft}
Floris Takens.
\newblock Characterization of a differentiable structure by its group of
  diffeomorphisms.
\newblock {\em Bol. Soc. Brasil. Mat.}, 10(1):17--25, 1979.

\end{thebibliography}

\end{document}